\documentclass{cmim-l}
\usepackage{amsmath, amsbsy, amsthm, amsfonts, array,rlepsf,multicol,amssymb, graphicx}
\long\def\comm#1\ent{}
\newtheorem{theorem}{Theorem}[section]

\newtheorem{claim}[theorem]{Claim}

\newtheorem{proposition}[theorem]{Proposition}
\newtheorem{lemma}[theorem]{Lemma}
\newtheorem{corollary}[theorem]{Corollary}

\newtheorem{observation}[theorem]{Observation}

\newtheorem*{theorem*}{Theorem}
\newtheorem{notation}[theorem]{Notation}
\newtheorem{case}{Case}
\newtheorem{subcase}{Case}
\numberwithin{subcase}{case}
\theoremstyle{definition}
\newtheorem{definition}[theorem]{Definition}

\newtheorem{remark}[theorem]{Remark}

\newcommand{\ie}{i.e.,\,}

\newcommand{\diam}{{\rm Diam}}
\newcommand{\age}{{\rm Age}}

\newcommand{\acl}{{\rm acl}}
\newcommand{\SU}{{\rm SU}}

\newcommand{\qftp}{{\rm qftp}}
\newcommand{\tp}{{\rm tp}}
\newcommand{\Th}{{\rm Th}}
\newcommand{\Sym}{{\rm Sym}}
\newcommand{\Lstp}{{\rm Lstp}}
\newcommand{\stp}{{\rm stp}}
\newcommand{\aut}{{\rm Aut}}

\newcommand{\indep}[1][]{\mathop{\raisebox{-.9ex}{$\underset{#1}{\smile}$}\makebox[-2.3ex]{$\mid$}
\hspace{2.3ex}}} 

\newcommand{\notindep}[1][]{\mathop{\raisebox{-.9ex}{$\underset{#1}{\smile}$}\makebox[-2.3ex]{$\mid$}\makebox[+0.5ex]{$\not$}
\hspace{2.3ex}}} 

\makeindex

\begin{document}
\title{\Huge{\textbf{The Classification of Homogeneous Simple 3-graphs}}}

\author
{\huge{Andr\'es Aranda}\\}
 \maketitle
 \thispagestyle{empty}

\titlepage

\tableofcontents \pagestyle{headings}
\chapter{Introduction}
\section{Homogeneous Structures}
Homogeneous structures appear in the work of Roland Fra\"iss\'e from the 1950s as relational structures with very large automorphism groups (see \cite{fraisse1954extension}, \cite{fraisse1986theory}), but some trace the origins of the subject to Cantor's proof that any two countable dense linearly ordered sets without endpoints are isomorphic. That theorem is proved by a back-and-forth argument, which in model-theoretic terms says that the theory of $(\mathbb Q,<)$ eliminates quantifiers in the language $\{<\}$. 

This subject is a meeting point for permutation group theory, model theory, and combinatorics. From the model-theoretic perspective, homogeneous structures have many desirable properties: they eliminate quantifiers, are prime, have few types, algebraic closure does not grow too quickly. All these properties made a full classification, at least for some restricted languages, accessible. There exist, for example, complete classifications of the finite and countably infinite homogeneous posets (Schmerl, \cite{schmerl1979countable}), graphs (Gardiner, \cite{gardiner1976homogeneous}, Lachlan and Woodrow \cite{lachlan1980countable}), tournaments (Woodrow, \cite{Woodrow1976}, Lachlan \cite{lachlan1984countable}), and digraphs (Cherlin, \cite{cherlin1998classification}).

During the 1970s and 80s, stability theory was a rapidly growing subject. Abstractions from the dimension or rank concepts in ``real life" theories were put to work, and whole families of theories were classified. Gardiner and Lachlan found that most finite homogeneous graphs and digraphs could be classified in a similar way: there was a partition of the set of structures into families parametrised by a few numbers. This parallel discovery led to Lachlan and Shelah's study of stable homogeneous structures \cite{lachlan1984stable}, and to Cherlin and Hrushovski's work on structures with few types in \cite{cherlin2003finite}. 

\begin{definition}\label{DefHom}
	A countable first-order structure $M$ for the relational language $L=\{R_i:i\in I\}$ is \emph{homogeneous} if any isomorphism between finite substructures extends to an automorphism of $M$.
\end{definition}

We will be dealing with finite languages all the time. It is essential to have a relational language: if the language has function symbols, we would have to change ``finite substructures" to ``finitely generated substructures," since functions can be iterated. Notice that this definition is stronger than the usual definition of homogeneity in model theory. The condition there is that partial \emph{elementary maps} extend to automorphisms. This is one reason why our homogeneous structures are often called \emph{ultra}homogeneous. Any partial elementary map is a local isomorphism, and so every ultrahomogeneous structure is homogeneous, but the converse is not true. The countability assumption is not necessary, but we will not consider homogeneous structures of any higher cardinality.

If $M$ is any (not necessarily homogeneous) first-order relational structure, the set of all finite structures isomorphic to substructures of $M$ is called the \emph{age} of $M$, denoted by $\age(M)$. It is clear from Definition \ref{DefHom} that the age of a homogeneous structure $M$ is of particular importance if we wish to understand $M$. 

Given countable relational structure $M$ for a countable language, the following are true:
\begin{enumerate}
	\item{$\age(M)$ has countably many members, since $M$ itself is countable.}
	\item{$\age(M)$ is closed under isomorphism, by definition.}
	\item{$\age(M)$ is closed under forming substructures: given $A\in\age(M)$, any substructure $B$ of $A$ will be finite, and a composition of the embeddings $B\rightarrow A$ and $A\rightarrow M$ proves that $B\in\age(M)$.}
	\item{$\age(M)$ has the \emph{Joint Embedding Property} or JEP: given two structures $A,B\in\age(M)$, there exist embeddings $f:A\rightarrow C$ and $g:B\rightarrow C$ for some $C\in\age(M)$.}
\end{enumerate}

The next theorem completes the picture:

\begin{theorem}[Fra\"iss\'e]
Let $L$ be a countable first-order relational language, and $\mathcal C$ a class of finite $L$-structures.
	\begin{enumerate}
		\item{There exists a countable structure $A$ whose age is equal to $\mathcal C$ if and only if $\mathcal C$ satisfies properties 1-4.}
		\item{There exists a homogeneous structure $A$ whose age is equal to $\mathcal C$ if and only if $\mathcal C$ satisfies 1-4 and the amalgamation property: given $A,B,C\in\mathcal C$ with embeddings $f_1:A\rightarrow B$ and $f_2:A\rightarrow C$, there exists $D\in\mathcal C$ and embeddings $g_1:B\rightarrow D$ and $g_2:C\rightarrow D$ such that $g_1\circ f_1=g_2\circ f_2$. Furthermore, this structure is unique up to isomorphism.}
	\end{enumerate}
\end{theorem}

In the course of the classification, the Lachlan-Woodrow Theorem will be used many times. 
\begin{theorem}[Lachlan-Woodrow 1980]\label{LachlanWoodrow}
	Let $G$ be an infinite homogeneous graph. Then either $G$ or $G^c$ is of one of the following forms:
	\begin{enumerate}
		\item{$I_m[K_n]$ with $\max(m,n)=\infty$,} 
		\item{Generic omitting $K_{n+1}$,}
		\item{Generic (the Random Graph)}
	\end{enumerate}
\end{theorem}

Consider a group of permutations $G$ acting on a set $X$. Then $G$ acts on each Cartesian power of $X$ coordinatewise. Peter Cameron introduced the term \emph{oligomorphic} action to describe the situation where $G$ acts on a countably infinite set $X$ and $G$ has finitely many orbits on $X^n$ for each natural number $n$ (see \cite{cameron1990oligomorphic}). The following theorem is an elaborate version of Ryll-Nardzewski's Theorem.

\begin{theorem}\label{RyllNardzewski}
	Let $M$ be a countably infinite structure over a countable language and $T={\rm Th}(M)$. The following are equivalent:
	\begin{enumerate}
		\item{$M$ is $\omega$-categorical\label{Ryll1}}
		\item{Every type in $S_n(T)$ is isolated, for all $n\in\omega$\label{Ryll2}}
		\item{Each type space $S_n(T)$ is finite\label{Ryll3}}
		\item{$(M,\aut(M))$ is oligomorphic\label{Ryll4}}
		\item{For each $n>0$ there are only finitely many formulas $\varphi(x_1,\ldots,x_n)$ up to ${\rm Th}(M)$-equivalence.\label{Ryll5}}
	\end{enumerate}
\end{theorem}

We have chosen two properties to guide our classification: first, the mostly combinatorial property of homogeneity, and second, the model-theoretical property of simplicity. One of the main theorems of simplicity theory is the Independence Theorem, which in our relational settings allows us to find simultaneous solutions to sufficiently independent (in the sense of forking) systems of relations. We will state the Independence Theorem later; first, a few basic facts about $\omega$-categorical and homogeneous structures.

\begin{proposition}
	Let $M$ be a countably infinite structure homogeneous over a finite relational language. Then $M$ is $\omega$-categorical.
\end{proposition}
\begin{proof}
	The language is finite: there can be only finitely many isomorphism types of substructures of $M$ of size $n$; by homogeneity, any two isomorphic finite substructures are in the same orbit, so by Theorem \ref{RyllNardzewski}, $M$ is $\omega$-categorical.
\end{proof}

\begin{proposition}
	The unique model $M$ of cardinality $\kappa$ of a countable $\kappa$-categorical theory is saturated.
	\label{CategoricalSaturated}
\end{proposition}
\begin{proof}
	By the L\"owenheim-Skolem theorem and countability.
\end{proof}

Recall that a \emph{small} substructure of a saturated model of cardinality $\kappa$ is a substructure of any cardinality $\lambda<\kappa$. In homogeneous models, partial elementary maps extend to automorphisms. It is not hard to prove that saturated models are homogeneous; as a consequence, 
\begin{proposition}\label{PropSaturation}
	In a saturated model $M$, two small substructures have the same type if and only if they belong to the same orbit under $\aut(M)$.
\end{proposition}

As a direct consequence of Theorem \ref{RyllNardzewski} and Proposition \ref{PropSaturation}, we have:
\begin{proposition}
	Let $M$ be a countable $\omega$-categorical structure and $A$ a finite subset of $M$. A subset $X\subset M$ is definable over $A$ if and only if $X$ is a union of orbits of the set of automorphisms of $M$ fixing $A$ pointwise.
\end{proposition}

\begin{proposition}
	Let $M$ be a countable $\omega$-categorical structure over a relational language $L$. Then $M$ is homogeneous if and only if $\Th(M)$ eliminates quantifiers in the language $L$.
\end{proposition}
\begin{proof}
	If $\Th(M)$ eliminates quantifiers, homogeneity follows from saturation, by Proposition \ref{CategoricalSaturated}.

	Given an $n$-tuple $\bar a$ in $M$, its isomorphism type in the language $L$ can be expressed by a quantifier-free formula. And in a homogeneous structure, the isomorphism type of $\bar a$ determines its orbit under the action of $\aut(M)$, and therefore its complete type. This is enough as we have shown that the quantifier-free type (\ie the isomorphism type of the substructure induced on the tuple) determines the complete type of the tuple. 
\end{proof}

Quantifier elimination is a matter of language; we can always force it on a structure by adding relation symbols to the language for each possible formula. If the structure we start with is $\omega$-categorical, then we need only add finitely many predicates for each natural number $n$, corresponding to the finitely many elements of $S_n(T)$ or, equivalently, to the orbits of $\aut(M)$. 

\begin{definition}
	An \emph{$n$-graph} is a structure $(M, R_1,\ldots,R_n)$ in which each $R_i$ is binary, irreflexive and symmetric; also, for all distinct $x,y\in M$ exactly one of the $R_i$ holds and $n\geq2$. We assume that all the relations in the language are realised in a homogeneous $n$-graph.  

	For any relation $P$ in the language of an $n$-graph $M$ and any element $a$, $P(a)$ denotes the set $\{x\in M: P(a,x)\}$. We often refer to this as the $P${\emph-neighbourhood of }$a$
\label{DefnGraph}
\end{definition}
	Some more definitions:
\begin{definition}~
\begin{enumerate}
	\item{A \emph{path} of colour $i$ and length $n$ between $x$ and $y$ is a sequence of distinct vertices $x_0, x_1,\ldots, x_n$ such that $x_0=x$, $x_n=y$ and for $0\leq j\leq n-1$ the edge $(x_j,x_{j+1})$ is of colour $i$.}
	\item{Two vertices $x,y$ in an edge-coloured graph $(M, R_1,\ldots,R_n)$ are $R_i$\emph{-connected} if there exists a path of colour $i$ between them; a subset $A$ of $M$ is $R_i$-connected if any $a, a'\in A$ are $R_i$-connected by a path in $A$. A maximal $R_i$-connected subset of $M$ is an \emph{$R_i$-connected component.}}
	\item{The \emph{$R_i$-distance} between two vertices $x,y$ in an edge-coloured graph, denoted by $d_i(x,y)$, is the length of a minimal $R_i$-path between $x$ and $y$ ($\infty$ if no such path exists). The \emph{$R_i$-diameter} of an $R_i$-connected graph $A$ is defined as the supremum of $\{d_i(x,y)|x,y\in A\}$.}
	\item{An $n$-graph is \emph{$R$-multipartite} with $k$ ($k>1$ possibly infinite) parts if there exists a (not necessarily definable) partition $P_1,\ldots,P_k$ of its vertex set into nonempty subsets such that if two vertices $x,y$ are $R$-adjacent then they do not belong to the same $P_i$. We will say that $G$ is \emph{$R$-complete-multipartite} if $G$ is $R$-multipartite with at least two parts and for all pairs $a,b$ from distinct classes, $R(a,b)$ holds.}
	\item{For any relation $R$, $n\in\omega$, and $a$, $R^n(a)$ is the set of vertices at $R$-distance $n$ from $a$.}
	\item{A \emph{half-graph} for colour $R$ with $m$ pairs in an $n$-coloured graph $M$ is a set of vertices $\{a_i:i\in m\}\cup\{b_i:i\in m\}\subset M$ such that $R(a_i,b_j)$ holds iff $i<j$.}
\end{enumerate}
\end{definition}

In this document, we are concerned with homogeneous 3-graphs $M$ (that is, 3-graphs homogeneous in the language $L=\{R,S,T\}$) with simple theory. The present work is a classification of a restricted class of homogeneous structures with simple unstable theory (in fact, supersimple with finite $\SU$-rank; more on this later), and as such, extends Lachlan's classification of stable homogeneous 3-graphs:

\begin{theorem}[Lachlan 1986, \cite{lachlan1986binary}]\label{Lachlan3graphs}
	Every stable homogeneous 3-graph is isomorphic to one of the following:
	\begin{multicols}{2}
	\begin{enumerate}
		\itemsep-0.25em
		\item{$P_{**}$}
		\item{$Z$}
		\item{$Z'$}
		\item{$Q_*^i$}
		\item{$P^i_*$}
		\item{$P^i[K_m^i]$}
		\item{$K_m^i[Q^i]$}
		\item{$Q^i[K_m^i]$}
		\item{$K_m^i[P^i]$}
		\item{$K_m^i\times K_n^j$}
		\item{$K_m^i[K_n^j[K_p^k]]$}
	\end{enumerate}
	\end{multicols}
where $\{i,j,k\}=\{R,S,T\}$ and $1\leq m,n,p\leq\omega$.
\end{theorem}

Items 1 to 5 are finite 3-graphs; for 6-11, if at least one of $m,n,p$ is infinite, the 3-graph is infinite. We will not explain what $Z$, the asterisks, and primes mean, since we are concerned only with infinite graphs. In the $j,k$-graph $P^i$ there are five vertices, and both the $j$-edges and the $k$-edges form a pentagon. The $j,k$-graph $Q^i$ is defined on 9 vertices; the $j$- and $k$-edges form a copy of $K_3\times K_3$. 

For $1\leq m,n\leq\omega$, $K_m^i\times K_n^j$ is the 3-graph with vertex set $m\times n$ and relations
\[((a_1,b_1),(a_2,b_2))\in\begin{cases}
i &\mbox{if } a_1\neq a_2\wedge b_1=b_2\\
j &\mbox{if } a_1= a_2\wedge b_1\neq b_2\\
k &\mbox{if } a_1\neq a_2\wedge b_1\neq b_2\\
\end{cases}
\]
where we again assume $\{i,j,k\}=\{R,S,T\}$.

And if $G$, $H$ are 3-graphs, then $G[H]$ is the 3-graph with vertex set $V(G)\times V(H)$ and in which the 3-graph induced on $\{(a,v):v\in V(H)\}$ is isomorphic to $H$ for each $a\in V(G)$, and for any function $f:V(G)\rightarrow V(H)$, the 3-graph induced on $\{(a,f(a)):a\in V(G)\}$ is isomorphic to $G$. More formally, $P((a,b),(c,d))$ holds in $G[H]$ if $a=c$ and $H\models P(b,d)$, or if $G\models P(a,c)$, where $P\in\{R,S,T\}$.

We often divide binary relations in two groups: forking and nonforking. We mean:
\begin{definition}
	Let $L=\{R_1,\ldots,R_n\}$ be a binary relational language. We say that $R_i$ is a \emph{forking} relation if $R(a,b)$ implies that $tp(a/b)$ forks over $\varnothing$. Otherwise, $R_i$ is \emph{nonforking}.
\end{definition}

Recall that given $A\subset B$ and $p\in S(B)$, a Morley sequence over $A$ is an $A$-indiscernible sequence $(\bar a_i:i\in I)$ that satisfies $\bar a_i\indep[A](\bar a_j:j<i)$ and $\tp(\bar a_i/Ba_0\ldots a_{i-1})$ does not fork over $A$ for all $i\in I$.

In many of our arguments, we make implicit use of the following theorem, especially the last statement, to justify the non-existence of an infinite clique of some particular colour in a neighbourhood of a vertex:
\begin{theorem}
Let $T$ be a first-order theory. The following are equivalent:
\begin{enumerate}
\item{$T$ is simple.}
\item{Forking (dividing) satisfies symmetry.}
\item{A formula $\varphi(\bar x,\bar a)$ does not divide over $A$ if and only if for some Morley sequence $I$ in $\tp(\bar a/A)$ the set $\{\varphi(\bar x,\bar c):c\in I\}$ is consistent.}
\end{enumerate}
\end{theorem}

The central theorem of simplicity is:
\begin{theorem}
	If $B\indep[A]C$, $\tp(\bar b/AB)$ and $\tp(\bar c/AC)$ do not fork over $A$, and $\Lstp(\bar b/A)=\Lstp(\bar c/A)$, then there is $\bar a\models\Lstp(\bar b/A)\cup\tp(\bar b/AB)\cup\tp(\bar c/AC)$, with $\bar a\indep[A]BC$.
\end{theorem}

In the primitive case, our method of classification relies on the following statement, related to the stable forking conjecture (see \cite{brower2012weak}): for a predicate $P$ in the language of a homogeneous simple primitive $3$-graph, if the formula $P(x,a)$ divides, then $P$ is a stable predicate. We will prove that this statement in Chapter \ref{ChapStableForking}.

The following theorem was recently proved by Vera Koponen in \cite{koponen2014binary}:
\begin{theorem}\label{Koponen}
Suppose that $M$ is a countable, binary, homogeneous and simple structure. Let $T$ be the complete theory of $M$. Then $T$ is supersimple with finite $\SU$-rank which is at most $|S_2(T)|$.
\end{theorem}

\begin{definition}
The \SU-rank is the least function from the collection of all types over parameters in the monster model to ${\rm On}\cup\{\infty\}$ satisfying for each ordinal $\alpha$ that $\SU(p)\geq\alpha+1$ if there is a forking extension $q$ of $p$ with $\SU(q)\geq\alpha$ .
\label{DefSU}
\end{definition}

The \SU-rank is invariant under definable bijections. Additionally, if $q$ is a nonforking extension of $p$, then $\SU(q)=\SU(p)$. A theory $T$ is supersimple if and only if $\SU(p)<\infty$ for all real types $p$. In the following theorem, we denote the Hessenberg sum of ordinals by $\oplus$.

\begin{theorem}[Lascar inequalities]
The \SU-rank satisfies the following inequalities:
\begin{enumerate}
	\item{$\SU(a/bA)+\SU(b/A)\leq\SU(ab/A)\leq\SU(a/bA)\oplus\SU(b/A)$.}
	\item{Suppose $\SU(a/Ab)<\infty$ and $\SU(a/A)\geq\SU(a/Ab)\oplus\alpha$. Then $\SU(b/A)\geq\SU(b/Aa)+\alpha$.}
	\item{Suppose $\SU(a/Ab)<\infty$ and $\SU(a/A)\geq\SU(a/Ab)+\omega^\alpha n$. Then $\SU(b/A)\geq\SU(b/Aa)+\omega^\alpha n$.}
	\item{If $a\indep[A] b$, then $\SU(ab/A)=\SU(a/A)\oplus\SU(b/A)$.}
\end{enumerate}
\label{LascarIneqs}
\end{theorem}

Chapter \ref{ChapGenRes} is a collection of easy results that will be used again and again in the rest of the classification. In Chapter \ref{ChapImprimitiveFC}, we will focus on simple unstable homogeneous 3-graphs in which the reflexive closure of the relation $R$ defines an equivalence relation with finite classes. From this it follows in particular that $S$ and $T$ do not define equivalence relations, as in that case $M$ would be a stable graph. We show there that there is only one such structure such that all the predicates are realised in the union of two classes, and classify the rest of them.

We prove in Chapter \ref{ChapPrimitive} that primitive homogeneous simple 3-graphs have $\SU$-rank 1, which enables us to prove that the only such graph is the analogue in three predicates of the Random Graph. After that, we use the results in Chapter \ref{ChapGenRes} to elucidate the structure of imprimitive homogeneous 3-graphs with infinite classes.
\chapter{General results on binary supersimple structures}\label{ChapGenRes}
\setcounter{equation}{0}
\setcounter{theorem}{0}
\setcounter{case}{0}
\setcounter{subcase}{0}

We collect in this chapter a number of results that will be used later. We start with an easy but extremely useful proposition saying that all the theories we are interested in are \emph{low}. The relevance of this is that in arguments using the Independence Theorem, lowness allows us to perform the amalgamation of the nonforking extensions $\tp(b/AB)$ and $\tp(c/AC)$ if $\stp(b/A)=\stp(c/A)$. This condition is generally easier to verify than the standard $\Lstp(b/A)=\Lstp(c/A)$, and in many of the cases that we will encounter, satisfied automatically.

Recall that a simple theory is low if for every formula $\varphi(\bar x,\bar a)$ there exists a natural number $n_\varphi$ such that given any indiscernible sequence $(\bar a_i:i\in\omega)$, if the set $\{\varphi(\bar x,\bar a_i):i\in\omega\}$ is inconsistent, then it is $n_\varphi$-inconsistent.

\begin{notation}
In the rest of this work, we often say that a relation $P$ defines an equivalence relation. Since each predicate is interpreted as an irreflexive relation, this is not strictly true. What we mean is that the reflexive closure of $P$ defines an equivalence relation, or, equivalently, that every triangle with two sides of type $P$ is a $K_3^P$.
\end{notation}

\begin{proposition}
	Let $T$ be an $\omega$-categorical simple theory eliminating quantifiers in a finite relational language. Then $T$ is low.
	\label{CatSimpLow}
\end{proposition}
\begin{proof}
Let $\varphi(x,a)$ be a formula in $L$. Denote by $m$ the highest arity for a relation in $L$, and let $\ell(a)$ be the length of the tuple $a$. Given any indiscernible sequence $(a_i:i\in\omega)$, the first $m$ tuples of the sequence determine the type over $\varnothing$ of $a_{i_0}...a_{i_k}$ for any $i_0<\ldots<i_k$ and any $k<\omega$.

By the Ryll-Nardzewski Theorem, there are ony finitely many types of $(\ell(a)\times m)$-tuples, so there are only finitely many kinds of indiscernible sequences over $\varnothing$. We claim that, given an $A$-indiscernible sequence $(d_i:i\in\omega)$, the set $D=\{\varphi(x,d_i):i\in\omega\}$ is consistent if and only if for any $\varnothing$-indiscernible sequence $(c_i:i\in\omega)$ such that $\tp(d_0\ldots d_{m-1})=\tp(c_0\ldots c_{m-1})$, the set $C=\{\varphi(x,c_i):i\in\omega\}$ is consistent. If $D$ is consistent, then viewing $(d_i:i\in\omega)$ as indiscernible over $\varnothing$ shows one
direction.

For the other direction, suppose that $C$ is consistent but $D$ is $k$-inconsistent for some $k\in\omega$. Let $u$ satisfy $C$. In particular, $u$ satisfies
$\varphi(x,c_0)\wedge\ldots\wedge\varphi(x,c_{k-1})$. Using homogeneity, there is an automorphism $\sigma$ of $M$ taking $c_0...c_{k-1}$ to $d_0...d_{k-1}$, so $\sigma(u)$ contradicts the $k$-inconsistency of $D$.

Let $\Phi_j(x)=\{\varphi(\bar x,i):i\in I_j\}$. If $\Phi_j(x)$ is inconsistent, then by indiscernibility it is $n_j$-inconsistent for some minimal $n_j\in\omega$. If we define $n_\varphi:=\max_{j\in\{1,\ldots,k\}}n_j$, then it is clear that for any indiscernible sequence $I$ of $\ell(\bar a)$-tuples, if $\{\varphi(x,i):i\in I\}$ is inconsistent, then it is $n_\varphi$-inconsistent.
\end{proof}

The next theorem appears as Theorem 6.4.6 in Wagner's book \cite{wagner2000simple}.
\begin{theorem}
	Let $T$ be a low theory. Then Lascar strong type is the same as strong type, over any set $A$.
\end{theorem}

The immediate corollary is:
\begin{corollary}
		Let $T$ be an $\omega$-categorical simple theory eliminating quantifiers in a finite relational language. Then the Lascar strong type of any tuple is the same as its strong type, over any set $A$.\hfill$\Box$
\label{LascarTypes}
\end{corollary}

Recall that an equivalence relation with finitely many classes is referred to as a \emph{finite equivalence relation}. The classes of an $A$-definable finite equivalence relation correspond to strong types over $A$ in a saturated model. 

\begin{proposition}
\label{PairwiseIndep}
	If $M$ is a binary homogeneous simple structure in which there are no $\varnothing$-definable finite equivalence relations on $M$, then for each $n\in\omega$ greater than 1, whenever $a_1,\ldots,a_n$ are pairwise independent elements of $M$, we have for each $1\leq i\leq n$ that $a_i\indep a_1,\ldots,a_{i-1},a_{i+1},\ldots, a_n$.
\end{proposition}
\begin{proof}
	We proceed by induction on $n$. The proposition is trivial for $n=2$; suppose that it holds for all $n\leq n_0$ and $a_1,\ldots,a_{n_0+1}$ are pairwise independent but such that $\tp(a_1/a_2,\ldots,a_{n_0+1})$ divides over $\varnothing$. By the induction hypothesis, $a_1\indep a_2,\ldots,a_{n_0}$ and $a_1\indep a_{n_0+1}$, so those two types are nonforking extensions of $\tp(a_1)$. We also have $a_{n_0+1}\indep a_2,\ldots,a_{n_0}$ by induction. Let $b\models\tp(a_1/a_{n_0+1})$ and $b'\models\tp(a_1/a_2,\ldots,a_{n_0})$; this also ensures that $\stp(b)=\stp(b')$, and because $\Th(M)$ is low by Proposition \ref{CatSimpLow}, they are of the same Lascar strong type. Therefore, $\Lstp(b/\varnothing)=\Lstp(b'/\varnothing)$. By the Independence Theorem, $\Lstp(b)\cup\tp(a_1/a_{n_0+1})\cup\tp(a_1/a_2,\ldots,a_{n_0})$ is a consistent set of formulas and is realised by some $a'\indep a_2,\ldots,a_{n_0+1}$. But in this case, because the language is binary, $$\tp(a_1/a_2,\ldots,a_{n_0+1})=\tp(a'/a_2,\ldots,a_{n_0+1}),$$a contradiction.
\end{proof}

By Proposition \ref{CatSimpLow}, we can carry out the argument in Proposition \ref{PairwiseIndep} over any set of parameters, as in any low theory $a\equiv^{\stp}_A b$ if and only if $a\equiv^{\Lstp}_A b$.

Reformulating \ref{PairwiseIndep} for sequences:
\begin{observation}\label{MorleySqn}
	In a binary homogeneous primitive simple structure, if $(a_i:i\in\omega)$ is an $\varnothing$-indiscernible sequence of singletons such that $a_0\indep a_1$, then $(a_i:i\in\omega)$ is a Morley sequence over $\varnothing$.\hfill$\Box$
\end{observation}
\begin{proposition}
	In a supersimple unstable primitive rank 1 homogeneous $n$-graph $(M;R_1,\ldots,R_n)$, $n>1$, each of the $R_i$ is unstable.
	\label{Year2}
\end{proposition}
\begin{proof}
	In \SU-rank 1 structures, forking is algebraic, so $\tp(a/b)$ forks iff over $\varnothing$ iff $a\in \acl(b)\setminus \acl(\varnothing)$. Therefore, each relation is non-algebraic, by primitivity, and so each relation is nonforking. Using the Independence Theorem to amalgamate partial structures over the empty set (cf. Propositions \ref{PairwiseIndep}, \ref{NonforkingAmalgamation}), we can embed infinite half-graphs for each of the $R_i$ into $M$, witnessing instability. See also Theorem \ref{PrimitiveAlice}.
\end{proof}

\begin{remark}
	{\rm The argument in Proposition \ref{PairwiseIndep} can be carried out in finitely homogeneous binary simple structures even over sets of parameters as long as we guarantee that the realisations of the types we wish to amalgamate have the same \emph{strong} type over the set of parameters, by Proposition \ref{CatSimpLow}.}
\end{remark}

\begin{definition}
	Let $L$ be a finite relational language in which each relation is binary. We will say that a family $\mathcal B$ of finite $L$-structures is the \emph{age of a random $L$-structure} if $B$ is an amalgamation class and all the minimal forbidden structures of $B$ are of size at most 2.
\end{definition}

\begin{proposition}
	Let $M$ be a binary homogeneous simple structure in which there are no $\varnothing$-definable finite equivalence relations on $M$. Suppose that all the relations in $L=\{R_1,\ldots,R_m\}$ are realised in $M$, and $R_1,\ldots,R_k$ are the only forking relations. Then the  subfamily of $\age(M)$ consisting of all finite $\{R_1,\ldots,R_k\}$-free substructures of $M$ is the age of a random $L\setminus\{R_1,\ldots,R_k\}$-structure.
\label{NonforkingAmalgamation}
\end{proposition}
\begin{proof}
	We aim to show that any finite structure not realising any of $R_1,\ldots, R_k$ embeds in $M$. All the $\{R_1,\ldots,R_k\}$-free structures of size 2 are realised in $M$ because the $R_i$ isolate 2-types. Consider an $\{R_1,\ldots,R_k\}$-free structure $B$ on $n+1$ points. We wish to show that this structure can be embedded into $M$, or, equivalently, that its isomorphism type belongs to $\rm{Age(M)}$.

Let $A=\{a_1,\ldots,a_n\}$ realise the substructure of $B$ on the first $n$ points, embedded in $M$, so $a_1\indep a_2,\ldots,a_n$. By the induction hypothesis, the type $p_1$ of $a_{n+1}$ in $B$ over $a_1$, and $p_2$, the type of $a_{n+1}$ over $a_2,\ldots,a_n$ are nonforking extensions of the unique strong type over the empty set, which by lowness (Proposition \ref{CatSimpLow}) is Lascar strong, and therefore by the Independence Theorem there is a single element $b$ of $M$ simultaneously satisfying both types, so using that $B$ is a binary structure, we get $\tp(b/a_1)\cup\tp(b/a_2,\ldots,a_n)\vdash\tp(b/a_1,\ldots,a_n)$, and conclude that $B$ can be embedded into $M$.
\end{proof}

By the same argument:
\begin{observation}
	Let $M$ be a homogeneous 3-graph of \SU-rank 2 with no definable finite equivalence relations on $M$, and suppose $S,T$ are nonforking relations. Then all finite $S,T$ structures can be embedded into the \SU-rank 2 homogeneous 3-graphs $S(a)$ and $T(a)$ for any vertex $a$.
\label{AllFiniteRFree}
\end{observation}
\begin{proof}
This is a direct consequence of Proposition \ref{NonforkingAmalgamation}.
\end{proof}

The following observation is folklore, but we include a proof for completeness.
\begin{observation}
	In a primitive $\omega$-categorical structure, $\acl(a)=\{a\}$.
	\label{AlgClosure}
\end{observation}
\begin{proof}
	The relation $x\sim y$ that holds if $\acl(x)=\acl(y)$ is an equivalence relation. It is clearly reflexive and transitive, and it is symmetric because if $y\in\acl(x)$, then $\acl(y)\subseteq\acl(x)$ and $|\acl(y)|=|\acl(x)|$, so the algebraic closures of $x$ and $y$ are equal as, by $\omega$-categoricity, they are finite sets. Hence $\sim$ is a symmetric relation, and clearly invariant. By primitivity, the $\sim$-classes are finite, and this relation is trivial. 
\end{proof}

Given a natural number $m$ and an irreflexive symmetric relation $R$, we denote the structure on $m$ vertices $v_0,\ldots,v_{m-1}$ in which for all distinct $v_i,v_j$ the formula $R(v_i,v_j)$ holds by $K_m^R$. In the following observation, a \emph{minimal} finite equivalence relation is a proper finite equivalence relation with minimal number of classes.

\begin{proposition}
	If $(M;R_0,\ldots,R_k)$ is a simple homogeneous transitive $k+1$-graph in which $R_0$ is a minimal finite equivalence relation with $m$ classes, and $R_1$ is a nonforking relation realised across any two $R_0$-classes, then the action of $\aut(M)$ induced on $M/R_0$ is $k+1$-transitive. 
	\label{EmbeddingCompleteGraphs}
\end{proposition}
\begin{proof}
	It suffices to show that $M$ embeds $K_{k+1}^{R_1}$. First note that we can embed the triangle $R_1R_1R_1$ across any three $R_0$-classes. To see this, consider $a,b$ with $R_1(a,b)$. By transitivity, $a$ and $b$ are of the same type over the empty set. The relation $R_1$ is realised between any two classes; consider $a',b'$ in the same $R_0$-class such that $R_1(a,a')$ and $R_1(b,b')$. Then $a'$ and $b'$ have the same (Lascar) strong type over $\varnothing$ and $\tp(a'/a),\tp(b'/b)$ are nonforking extensions of the unique 1-type over the empty set; we can apply the Independence Theorem to find an element $c$ in the same $R_0$ class such that $abc$ is a $K_3^{R_1}$.

	The result follows by iterating the same argument, amalgamating nonforking ($R_1$) extensions of smaller complete graphs over the empty set. We can only iterate as many times as the number of $R_0$-classes.
\end{proof}
\begin{proposition}
	Let $M$ be a simple homogeneous transitive 3-graph in which $R$ defines an equivalence relation, and assume that the induced action of $\aut(M)$ on $M/R$ is transitive, but not 2-transitive, so for any pair of distinct $R$-classes $C,C'$ only one of $S,T$ is realised across $C,C'$. Then the $S,T$-graph induced on a set $X$ containing exactly one element from each $R$-class is homogeneous.
\end{proposition}
\begin{proof}
	Consider the graph defined on $M/R$ with predicates $\hat S,\hat T$ which hold of two distinct classes $a/R,b/R$ if for some/any $\alpha\in a/R,\beta\in b/R$ we have $S(\alpha,\beta)$ (respectively, $T(\alpha,\beta)$). This graph is clearly isomorphic to the graph induced on $X$.
\begin{claim}
	The graph interpreted in $M/R$ as described in the preceding paragraph is homogeneous in the language $\{\hat S,\hat T\}$.
\end{claim}
\begin{proof}\label{ClaimTrick}
	Let $\pi$ denote the quotient map $M\rightarrow M/R$. Given two isomorphic finite substructures $A,A'$ of $M/R$, then any transversals to $\pi^{-1}(A)$ and $\pi^{-1}(A')$ are isomorphic, so by the homogeneity of $M$ there exists an automorphism $\sigma$ taking $\pi^{-1}(A)$ to $\pi^{-1}(A')$. The map $\pi\sigma\pi^{-1}$ is an automorphism of $M/R$ taking $A$ to $A'$.
\end{proof}
And the result follows.
\end{proof}
The argument from Claim \ref{ClaimTrick} will appear again in the future. 

\begin{observation}\label{sop}
	In any homogeneous transitive $n$-graph $(M,R_1,\ldots,R_n)$, if $R_i(a)$ is an $R_i$-complete graph, then for any $b\in R_i(a)$ we have $\{a\}\cup R_i(a)=\{b\}\cup R_i(b)$.
\end{observation}
\begin{proof}
	If $c\in R_i(b)\setminus R_i(a)$, then both $a$ and $c$ are in $R_i(b)$, which is $R_i$-complete by transitivity, and therefore $R_i(a,c)$ holds, contradiction.
\end{proof}
\begin{observation}
	If $(M,R_1,\ldots,R_n)$, where $n>1$, is a primitive homogeneous $n$-graph, then for all $i$ with $1\leq i\leq n$, the structure $R_i(a)$ is not $R_i$-complete.
	\label{NotRComplete}
\end{observation}
\begin{proof}
	Suppose not. Then, using Observation \ref{sop} and homogeneity, there is $i$ with $1\leq i\leq n$ such that for all $a,b$ with $R_i(a,b)$ we have $\{a\}\cup R_i(a)=\{b\}\cup R_i(b)$. Hence, $\{a\}\cup R_i(a)$ is an $R_i$-connected component. This contradicts primitivity, since $|R_i(a)|>0$ and as $n>1$, $\{a\}\cup R_i(a)\neq M$.
\end{proof}

\begin{definition}\label{DefMultipartite}
An $n$-graph is \emph{$R$-multipartite} with $k$ ($k>1$, possibly infinite) parts if there exists a (not necessarily definable) partition $P_1,\ldots,P_k$ of its vertex set into nonempty subsets such that if two vertices $x,y$ are $R$-adjacent then they do not belong to the same $P_i$. We will say that $G$ is \emph{$R$-complete-multipartite} if $G$ is $R$-multipartite with at least two parts and for all pairs $a,b$ from distinct classes, $R(a,b)$ holds.
\end{definition}

\begin{proposition}\label{PropMultipartite}
	Let $(M;R_1,\ldots,R_n)$ be an $R_i$-connected transitive homogeneous $n$-graph. If for some $a\in M$ the set $R_i(a)$ is $R_i$-complete-multipartite, then $M$ is $R_i$-complete-multipartite (and in particular is not primitive).
\end{proposition}
\begin{proof}
	For simplicity, we will write $R$ and not $R_i$. Note first that the partition of $R(a)$ is invariant over $a$, defined by $R(a,x)\wedge R(a,y)\wedge\neg R(x,y)=:E_a(x,y)$. Take any $b\in R(a)$. By homogeneity, $R(b)$ consists of $a/E_b$ together with $R(a)\setminus(b/E_a)$. We claim that this is all there is in $M$. First note that there are no more classes in $R(b)\setminus R(a)$: if we had $c\in R(b)\setminus R(a)$ not $E_b$-equivalent to $a$, then by homogeneity we would have $R(a,c)$, contradicting $c\notin R(a)$. Therefore, $a/E_b\cup R(a)$ is an $R$-connected component of $M$; by connectedness, it is all of $M$, $\diam_R(M)=2$, and $\neg R(x,y)$ is an equivalence relation.
\end{proof}

\begin{observation}\label{ObsFiniteDiameter}
	If $(M,R_1,\ldots,R_n)$ is an $\omega$-categorical $n$-graph, then each connected component of $(M,R_i)$ has finite diameter.
\end{observation}
\begin{proof}
	Each of the $R_i$-distances is preserved by automorphisms. If one of the connected components of $(M,R_i)$ has infinite diameter, then there are infinitely many 2-types, contradicting $\omega$-categoricity.
\end{proof}

As a consequence of this observation, in $\omega$-categorical edge-coloured graphs the relation $E_i(x,y)$ which holds if there is a path of colour $i$ between $x$ and $y$ is definable. Also, in primitive $n$-coloured graphs, each $(M,R_i)$ is connected, since the equivalence relation $x\sim_{R_i}y$ that holds if $x$ and $y$ are $R_i$-connected is invariant under $\aut(M)$.

\begin{observation}\label{diam}
	If $(M,R_1,\ldots,R_n)$ is a homogeneous $n$-graph, then the diameter of each connected component of $(M,R_i)$ is at most $n$.
\end{observation}
\begin{proof}
	Suppose there are $a,b\in M$ at $R_i$-distance $n+1$, so there are distinct $a=x_0,x_1,\ldots,x_{n+1}=b$ such that $R_i(x_j,x_{j+1})$ for $0\leq j\leq n$ and $R_i$ does not hold in any other pair from $\{x_0,\ldots,x_{n+1}\}$. Then the $n$ pairs $(a,x_j)$ ($2\leq j\leq n+1$) are coloured in $n-1$ colours, so at least two of them have the same colour. Using homogeneity, there is an automorphism of $M$ taking the pair with the smaller index in the second coordinate to the other pair, and therefore we can find a shorter path from $a$ to $b$.
\end{proof}

\chapter{The Imprimitive Case: Finite Classes}\label{ChapImprimitiveFC}
\setcounter{equation}{0}
\setcounter{theorem}{0}
\setcounter{case}{0}
\setcounter{subcase}{0}

In this chapter we classify the homogeneous simple unstable 3-graphs with an invariant equivalence relation. We will assume for definiteness that the equivalence relation is the reflexive closure of the predicate $R$. Note that this is not a limitation in any sense, since by Ryll-Nardzewski's Theorem in our context an invariant equivalence relation is defined by a disjunction of atomic formulas: given that we want the classes to be finite, this means that a disjunction of two atomic formulas cannot be an equivalence relation if $M$ is to be unstable.

\section{Imprimitive Structures with Finite Classes}
Let us describe the construction of an imprimitive homogeneous 3-graph with classes of size 2. Start by enumerating the random graph as $\{w_i:i\in\omega\}$, and define a 3-graph $C(\Gamma)$ on countably many vertices $\{v_i:i\in\omega\}$ where $R$ holds for pairs of vertices of the form $v_{2n}v_{2n+1}$, 
\[
S(v_i,v_j)\mbox{ if}
\begin{cases}
i\neq j, i=2m, j=2n, E(w_m,w_n)\\
i\neq j, i=2m+1, j=2n+1, E(w_m,w_n)\\
i\neq j, i=2m, j=2n+1, \neg E(w_m,w_n)\\
i\neq j, i=2m+1, j=2n, \neg E(w_m,w_n)\\
\end{cases}
\]
and all other pairs of distinct vertices satisfy $T$ ($E$ denotes the edge relation in the random graph). This structure is a finite cover in the sense of Evans (see \cite{evans1996splitting}, \cite{evans1995finite}) of a reduct of the random graph. Its theory is supersimple of rank 1, as it can be interpreted in $\Gamma\times\{0,1\}$

Our main result in this chapter is:
\begin{theorem*}
	Up to isomorphism, the only imprimitive simple unstable homogeneous 3-graph with finite classes such that all relations are realised in the union of two $R$-classes is $C(\Gamma)$.
\end{theorem*}

\subsection{The proof}\label{sec:observations}
Let $M$ be a homogeneous structure with an invariant equivalence relation $E$, and denote by $M/E$ the set of equivalence classes modulo $E$ in $M$. Then there is a homomorphism $f:\aut(M)\rightarrow\Sym(M/E)$, given by $f(\sigma)(\ulcorner a\urcorner)=\ulcorner\sigma(a)\urcorner$, so that $\aut(M)$ acts on $M/E$. We refer to this action as the induced action of $\aut(M)$ on $M/E$. The orbit of a tuple of classes under this action is determined by the isomorphism type of the union of those classes in $M$.

Recall that a permutation group $G$ on $\Omega$ is $k$-transitive if it acts transitively on the set of $k$-tuples of distinct elements of $\Omega$.

\begin{observation}\label{ObsTransRE}
	Let $M$ be a homogeneous structure with an invariant equivalence relation $E$, and suppose that there is a symmetric binary predicate $S$ such that whenever $A,B$ are distinct $E$-classes, there exist $a\in A,b\in B$ such that $S(a,b)$ holds. Then the induced action of $\aut(M)$ on $M/E$ is 2-transitive.
\end{observation}
\begin{proof}
	Let $(A,B)$ and $(A',B')$ be pairs of distinct $E$-classes. We wish to prove that there exists $\sigma\in\aut(M)$ such that the image under $\sigma$ of $A$ (respectively, $B$) is $A'$ ($B'$). By hypothesis, there exists $a\in A, b\in B$ and $a'\in A, b'\in B$ such that $S(a,b)$, $S(a',b')$ holds, so that the function $a\mapsto a', b\mapsto b'$ is a local isomorphism in $M$. By homogeneity, this function is induced by some $\sigma\in\aut(M)$, and by invariance of $E$, $A$ is mapped to $A'$ and $B$ to $B'$.
\end{proof}

\begin{remark}\label{Rmk1}
The conclusion of Observation \ref{ObsTransRE} can be strengthened to $k$-transitivity in the case of a transversal symmetric $k$-ary predicate. Note that if $M$ is a homogeneous $n$-graph in which the reflexive closure of $R$ is an equivalence relation and $S$ is realised in the union of any two distinct $R$-classes, then all other relations in the language are also realised in the union of any two classes (otherwise, our global assumption that all relations are realised in $M$ is contradicted).
\end{remark}

\begin{remark}\label{Rmk2}
	If all the definable binary relations in $M$ are symmetric, then the converse to Observation \ref{ObsTransRE} is also true (as is the conclusion that all binary relations are realised in the union of any two distinct classes).
\end{remark}

\begin{observation}\label{NotCompBipart}
	Let $M$ be an imprimitive homogeneous 3-graph with $R$-classes of size $n<\omega$, and suppose that $\aut(M)$ acts 2-transitively on $M/R$. Let $A,B$ be distinct $R$-classes in $M$ and $a\in A$. Then $1\leq |S(a)\cap B|<n$.
\end{observation}
\begin{proof}
	Let $r$ denote $|S(a)\cap B|$. Suppose for a contradiction that $r=n$. By homogeneity and symmetry of $S$, for all $b\in B$ there is an automorphism of $M$ taking $b\mapsto a$ and $a\mapsto b$. This automorphism takes $B$ to $A$ by invariance and $S(a)\cap B$ to $S(b)\cap A$, so that $T$ is not realised in $A\cup B$, contradicting the hypothesis of 2-transitivity of the induced action of $\aut(M)$ on $M/R$ by Remark \ref{Rmk2}.
\end{proof}

\begin{observation}\label{ObsTransOnClasses}
	Let $M$ be an imprimitive homogeneous 3-graph with $R$-classes of size $n<\omega$, and suppose that $\aut(M)$ acts 2-transitively on $M/R$. Let $A,B$ be distinct $R$-classes in $M$. Then $\aut(M)_{\{B\}}\cap\aut(M)_{\{A\}}$ acts transitively on $B$ and on $A$.
\end{observation}
\begin{proof}
	Take any $b,b'\in B$. By Observation \ref{NotCompBipart}, $S(b)\cap A\neq\varnothing$ and $S(b')\cap A\neq\varnothing$, so we can find $a_1,a_2\in A$ such that $S(b,a_1)$ and $S(b',a_2)$ hold, so that $b\mapsto b'$, $a_1\mapsto a_2$ is a local isomorphism. The conclusion follows by homogeneity and invariance of $R$.
\end{proof}

It follows from Observation \ref{ObsTransOnClasses} and our general setting that whenever $B,C$ and $B',C'$ are pairs of distinct $R$-classes, then the structure induced by $M$ on $B\cup C$ is isomorphic to that on $B'\cup C'$, and that for any vertices $b,c$ and classes $K,K'$ not including $b,c$, $|S(b)\cap K|=|S(b')\cap K'|$.

\begin{observation}\label{ObsEasyCase}
	Let $M$ be an imprimitive homogeneous 3-graph with $R$-classes of size $n<\omega$, and suppose that $\aut(M)$ acts 2-transitively on $M/R$ and $S$ is not an equivalence relation. Let $A,B$ be distinct $R$-classes in $M$ and $r\mathrel{\mathop:}=|S(a)\cap B|$ for any $a\in A$. If $r=1$, then $n=2$.
\end{observation}
\begin{proof}
	Since $S$ is not an equivalence relation, there exist pairwise inequivalent $b,c,d\in M$ such that $S(b,c)\wedge S(c,d)\wedge T(b,d)$.

Let $B,C,D$ denote the classes of $b,c,d$. By 2-transitivity on $M/R$, there exists $d'\neq d$ in $D$ such that $S(b,d')$. By the same reason, there exists $c'\neq c$ in $C$ such that $S(d',c')$ holds. 

If $n>2$, then for any $c''\in C, c''\neq c,c'$ we have $\qftp(c'/bcd)=\qftp(c''/bcd)$ because $r=1$ implies that both will satisfy $T(b,x)\wedge T(d,x)\wedge R(c,x)$, but $\tp(c'/bcd)\neq\tp(c''/bcd)$, since $c'$ satisfies the formula $\varphi(y)=\exists x(R(d,x)\wedge S(b,x)\wedge R(x,y))$, but $c''$ does not, contradicting homogeneity. In the following diagram, the heavy lines represent $R$-cliques, solid lines are $S$-edges, and dotted lines are $T$-edges.
\[
\includegraphics[scale=0.7]{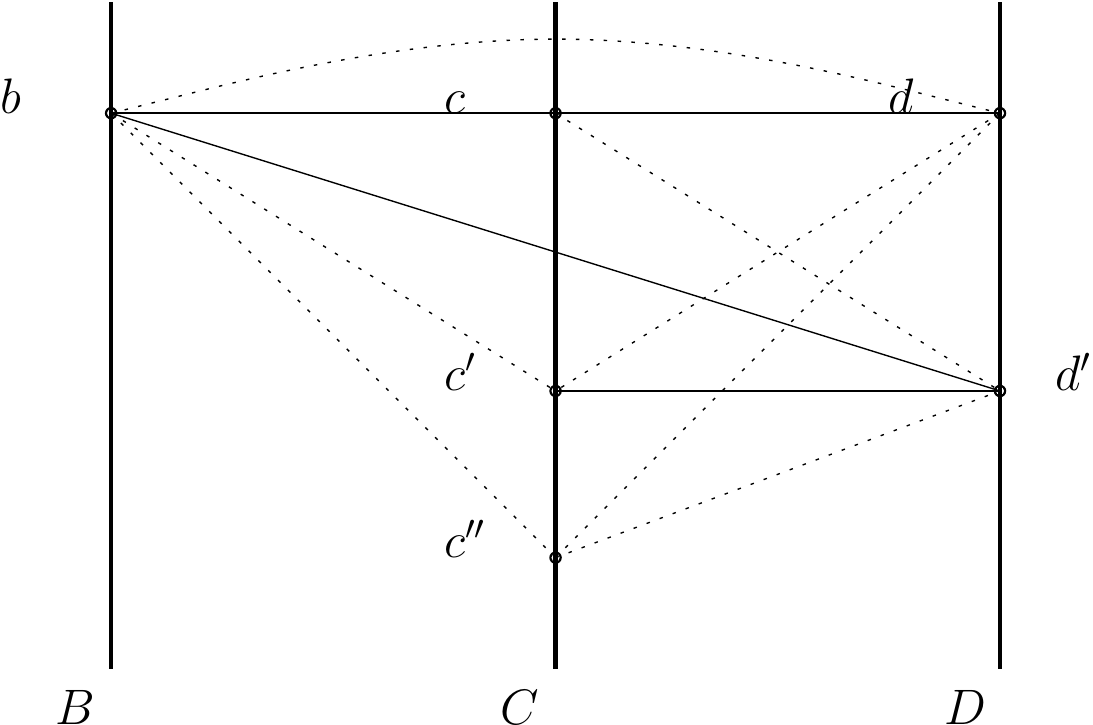}
\]
\end{proof}

As we stated before, we are interested only in simple unstable 3-graphs, because the stable ones have already been classified. A first question is which of the binary relations divide over the empty set, \ie for which $P\in\{R,S,T\}$ do we have that the formula $P(x,a)$ divides. Clearly, since we are interested in structures $M$ in which $R$ defines an equivalence relation with infinitely many classes, we have that $R$ divides. Moreover, since there are Morley sequences in any type in a simple theory, we get that at least one of $S,T$ is nonforking. Let us assume without loss of generality that $S$ is nonforking.

\begin{proposition}\label{PropSTNf}
	Let $M$ be an imprimitive simple unstable 3-graph in which $R$ defines an equivalence relation with finite classes, and assume that $S$ is nonforking. Then $T$ is nonforking.
\end{proposition}
\begin{proof}
	Suppose for a contradiction that $T(x,a)$ divides. Then for any Morley sequence $(c_i)_{i\in\omega}$ we have that $\{T(x,c_i):i\in\omega\}$ is inconsistent. From the fact that $R$ is algebraic and $T$ divides we can derive that any Morley sequence of vertices must be an infinite $S$-clique. Therefore, the $T$-neighbourhood of any vertex contains no infinite $S$-cliques.

Since $S$ and $T$ are unstable, there exists an infinite half-graph: we have an infinite collection of vertices $a_i,b_i$ such that $S(a_i,b_j)$ holds iff $i\leq j$. From this half-graph we can extract using Ramsey's theorem an ``indiscernible" half-graph, \ie a half-graph which, considered as a sequence of pairs $a_ib_i (i\in\omega)$ is indiscernible over $\varnothing$. We will abuse notation and name the elements of this indiscernible half-graph $a_i$ and $b_i$ ($i\in\omega$). It follows from the algebraicity of $R$ and the argument in the preceding paragraph that the $a_i$ and the $b_i$ form infinite $T$-cliques. By homogeneity, the $T$-neighbourhood of any vertex contains a copy of this indiscernible half-graph and is therefore, considered as a separete structure, an unstable 3-graph.

If $|T(a)\cap B|=1$ for any $R$-class $B$ not containing $a$, then $T(a)$ is $R$-free. By our argument, $T(a)$ is a simple unstable homogeneous $S,T$-graph, so by the Lachlan-Woodrow theorem $T(a)$ is isomorphic to the Random Graph. This contradicts our hypothesis, since the Random Graph contains infinite cliques and infinite independent sets.

Suppose then that $|T(a)\cap B|>1$. Then $R$ defines an equivalence relation with finite classes in $T(a)$, which is unstable and embeds no infinite $S$-cliques. We could have two orbits of pairs of $R$-classes in $T(a)$, so that only one of $S$ and $T$ is realised in any union of two $R$-classes. In this case, consider the (definable) structure $T(a)/R$ with the (definable) relations $\hat T$ and $\hat S$ that hold between two elements $x,y$ of $T(a)/R$ if the pair of classes in $T(a)$ that $x,y$ represent realise only the corresonding relation. Then $T(a)/R$ is a simple unstable graph, isomorphic to the Random graph by the Lachlan-Woodrow theorem, and therefore $T(a)$ embeds infinite $S$-cliques. We have reached a contradiction again.

The last case is when both $S$ and $T$ are realised in the union of any two $R$-classes in $T(a)$. Then $T(a)$ is a homogeneous simple unstable 3-graph and $\aut(M)_a$ acts 2-transitively on the set of $R$-classes induced in $T(a)$. Notice that in $T(a)$ as a separate structure $S$ is still nonforking, since the indiscernible half-graph shows that the elements of $T(a)$ are $S$-related to an infinite $T$-clique within $T(a)$. So $T(a)$ still satisfies the hypotheses of the proposition, and we can iterate the argument (take intersections $T(a_1)\cap T(a_2)\cap\ldots\cap T(a_n)$ where $a_{i+1}\in T(a_i)$) until we reach a simple unstable $R$-free graph, and therefore a contradiction as in the other cases.
\end{proof}

\begin{remark}\label{trivialremark}
Note that a union $U=A\cup B$ of two $R$-classes is homogeneous in a restricted sense: suppose that $C,D$ are isomorphic subsets of $U$. If their union includes an $S$- or a $T$-edge, then the extension of the isomorphism to an automorphism of $M$ will fix $U$ setwise and so its restriction to $U$ will be an automorphism of $U$. But there is no guarantee that if each of $C,D$ are $R$-cliques of the same size, both contained in the same class $A$, there will be and extension of the isomorphism fixing $B$ setwise as well.  
\end{remark}

\begin{observation}\label{ObsNoFER}
	Let $M$ be a simple unstable homogeneous 3-graph in which $R$ defines an equivalence relation with finite classes. Then $R$ is the only invariant equivalence relation on $M$. In particular, there are no invariant proper equivalence relations with finitely many classes in $M$.
\end{observation}
\begin{proof}
	Since $M$ is unstable and $R$ is stable and forking, we have by Proposition \ref{PropSTNf} that $S$ and $T$ are nonforking unstable relations. In particular, $S$ and $T$ do not define equivalence relations. Since $R$ is realised and an equivalence relation, $S\vee T$, $R\vee S$ and $R\vee T$ do not define equivalence relations either. The result follows by quantifier elimination.
\end{proof}

Let $(I,<)$ be a linearly ordered set. The sequence $(a_i:i\in I)$ is $A$-independent if for every $i\in I$, $a_i\indep[A]a_{<i}$. Recall two lemmata from simplicity theory (5.14 and 5.20 (4) in \cite{casanovas2011simple}):

\begin{lemma}
Let $(a_i:i\in I)$ be $A$-independent. If $J,K$ are subsets of $I$ such that $J<K$ (that is, $j<k$ for any $j\in J, k\in K$), then $\tp((a_i:i\in K)/A(a_i:i\in J))$ does not divide over $A$.
\label{Casanovas514}
\end{lemma}
\begin{lemma}
	If $T$ is simple, then $A\indep[B]C\Leftrightarrow A\indep[B]\acl(C)\Leftrightarrow\acl(A)\indep[B]C\Leftrightarrow A\mathop{\raisebox{-.9ex}{$\underset{\acl(B)}{\smile}$}\makebox[-4.9ex]{$\mid$}
\hspace{5.0ex}}C$
\label{Casanovas520}
\end{lemma}

\begin{observation}\label{ObsRFree}
		Let $M$ be a simple unstable homogeneous 3-graph in which $R$ defines an equivalence relation with finite classes. Then any $R$-free set is $\varnothing$-independent.
\end{observation}
\begin{proof}
	We can use the fact that $S$, $T$ are nonforking relations (Proposition \ref{PropSTNf}) and the Independence Theorem to construct any $R$-free structure, one vertex at a time. At each step, we obtain a vertex that is independent from all the previous vertices. 
\end{proof}
\begin{observation}
	If $X,Y$ are unions of $R$-classes and $X\cap Y=\varnothing$, then $X\indep Y$.
\label{ObsIndepClasses}
\end{observation}
\begin{proof}
A transversal to the set $XY$ is independent by Observation \ref{ObsRFree}. We can order it in such a way that the elements $(c_i:i\in I)$ from a transversal to $X$ appear before the elements $(c_j:j\in J)$ of a transversal to $Y$. By Lemma \ref{Casanovas514}, $\tp((c_i:i\in I)/(c_j:j\in J))$ does not fork over $\varnothing$, and so by Lemma \ref{Casanovas520} $X\indep Y$.
\end{proof}
\begin{remark}\label{rmksubsets}
By monotonicity, if $K\subset A$, $K'\subset B$, and $A,B$ are distinct $R$-classes we get $K\indep K'$. We will make use of this fact in the next proposition.
\end{remark}

\begin{observation}\label{ObsPartition}
	Let $M$ be a homogeneous 3-graph in which $R$ defines an equivalence relation with finite classes and $S$ is not an equivalence relation, and such that $\aut(M)$ acts 2-transitively on $M/R$. If for distinct $R$-classes $A,B$ there exists a nontrivial partition $A=A_1\cup\ldots\cup A_m$, $B=B_1\cup\ldots\cup B_m$ such that the structure induced on $A_i\cup B_i$ is a $K_r^S$ for some $r$, and there are no other $S$-edges in $A\cup B$, then $m=2$.
\end{observation}
\begin{proof}
The partition would in this case be definable over $a\in A,b\in B$ as an equivalence relation $\sim$. Consider $a\in A$ and three elements $b,b',b''\in B$ such that $b\sim b'$ and $b\not\sim b''$. Then $abb'$ and $abb''$ form triangles $TTR$ with a common vertex $a$ and are therefore isomorphic, but there is no way to extend this isomorphism to the two classes since for $b,b'$ there exists $a'\in A$ such that $S(a',b)\wedge S(a',b')$, but this formula is not satisfied by $b,b''$.
\end{proof}

\begin{lemma}\label{Lemma1}
	Let $M$ be an imprimitive simple unstable homogeneous 3-graph with $R$-classes of size $n<\omega$, and suppose that $\aut(M)$ acts 2-transitively on $M/R$. Let $A,B$ be distinct $R$-classes in $M$ and $r\mathrel{\mathop:}=|S(b)\cap C|$ for any $b\in B$. Then $n=2r$.
\end{lemma}
\begin{proof}
	Let us denote $S(a)\cap B$ by $B_a$. Each $a\in A$ picks out (via $B_a$) one of the $n\choose r$ subsets of $B$ of size $r$. By Observation \ref{ObsTransOnClasses}, each element of $B$ is in the $S$-neighbourhood of some element from $A$, so we have $$B=\bigcup_{a\in A}(S(a)\cap B)$$
We can assume by Observation \ref{ObsEasyCase} that $n>2$. Suppose for a contradiction that $2r<n$. Then we have for all distinct $a,a'\in A$, $T(a)\cap T(a')\cap B\neq\varnothing$. From this it follows that for all $a\neq a'\in A$ the formula $B_a\cap B_{a'}\neq\varnothing$ holds; to see this, suppose that we had distinct $a,a',a''\in A$ such that $B_a\cap B_{a'}\cap B=\varnothing, B_a\cap B_{a''}\neq\varnothing$. We can find $b,b'\in B$ such that $T(b,a)\wedge T(b,a'')$ and $T(b',a)\wedge T(b',a')$ hold, so by homogeneity we can move $baa''$ to $b'aa'$ (see diagram below).

\[
\includegraphics[scale=0.7]{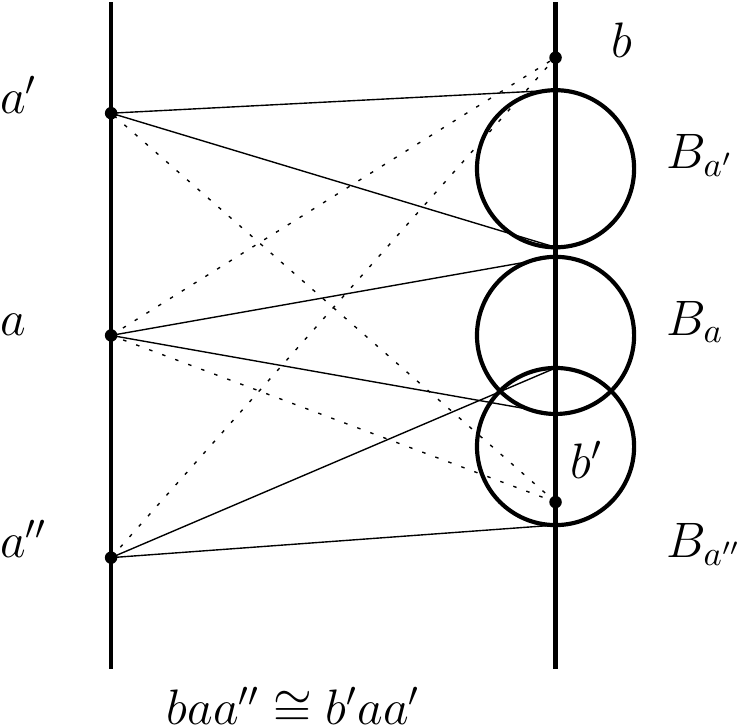}
\]

But any automorphism doing that should take $B_a\cap B_{a''}$ to $B_a\cap B_{a'}$, which is impossible since one of these is empty and the other is not. Furthermore, the number $r_2=|S(a)\cap S(a')\cap B|$ is constant for all distinct $a$ and $a'$ in $A$.

We have two possible cases: either there exist some $a,a'$ with $B_a=B_{a'}$, or for all $a\neq a'$ we have $B_a\neq B_{a'}$.

\begin{case}\label{Case1}
If for some $a,a'$ we have $B_a=B_a'$, define a relation $\sim$ on $A$ by $a\sim a'$ if $B_a=B_a'$. Note that for each $\sim$-class $k$, the set $k'=\bigcap_{\kappa\in k} S(\kappa)\cap B$ is an equivalence class of the corresponding equivalence relation defined in $B$ with respect to $A$. Also, the structure induced on $k\cup k'$ is $K_{r,r}^S$. Since $2r<n$, the equivalence relations have each at least three classes, each with at least two elements by Observation \ref{ObsEasyCase}. This cannot happen by Observation \ref{ObsPartition}.
\end{case}

\begin{case}\label{Case2}
Suppose then that for all $a\neq a'$ we have $B_a\neq B_{a'}$. Take $x,x'\in B$ and let $C$ be an $R$-class, $C_x=S(x)\cap C$, $C_{x'}=S(x')\cap C$. Choose $X\subset C$ of size $r$ such that $X\cap C_{x'}\neq\varnothing$ and $X\cap C_x=\varnothing$. Such an $X$ exists because there are more than $r$ elements in $C\setminus C_x$, by our assumption $2r<n$. By homogeneity, there exists a $\delta$ such that $C_\delta=X$. Now choose $Y\subset B$ of size $r$ containing $x,x'$ (we can find $Y$ because as a consequence of Observation \ref{ObsEasyCase}, $r\geq 2$); by the same argument there is a $\beta$ such that $B_\beta=Y$. By Remark \ref{rmksubsets}, $X\indep Y$, and also $\tp(\beta/Y)$ and $\tp(\delta/X)$ are nonforking over $\varnothing$. Furthermore, $\beta$ and $\delta$ have the same Lascar strong type over $\varnothing$ by Observation \ref{ObsNoFER}, so by the Independence Theorem there is $a\in M$ which is $S$-related to $x,x'$ and to $c$, so we have $S(a,x)\wedge T(x,c)\wedge S(a,c)$ and $S(a,x')\wedge T(x',c)\wedge S(a,c)$. 

\[
\includegraphics[scale=0.7]{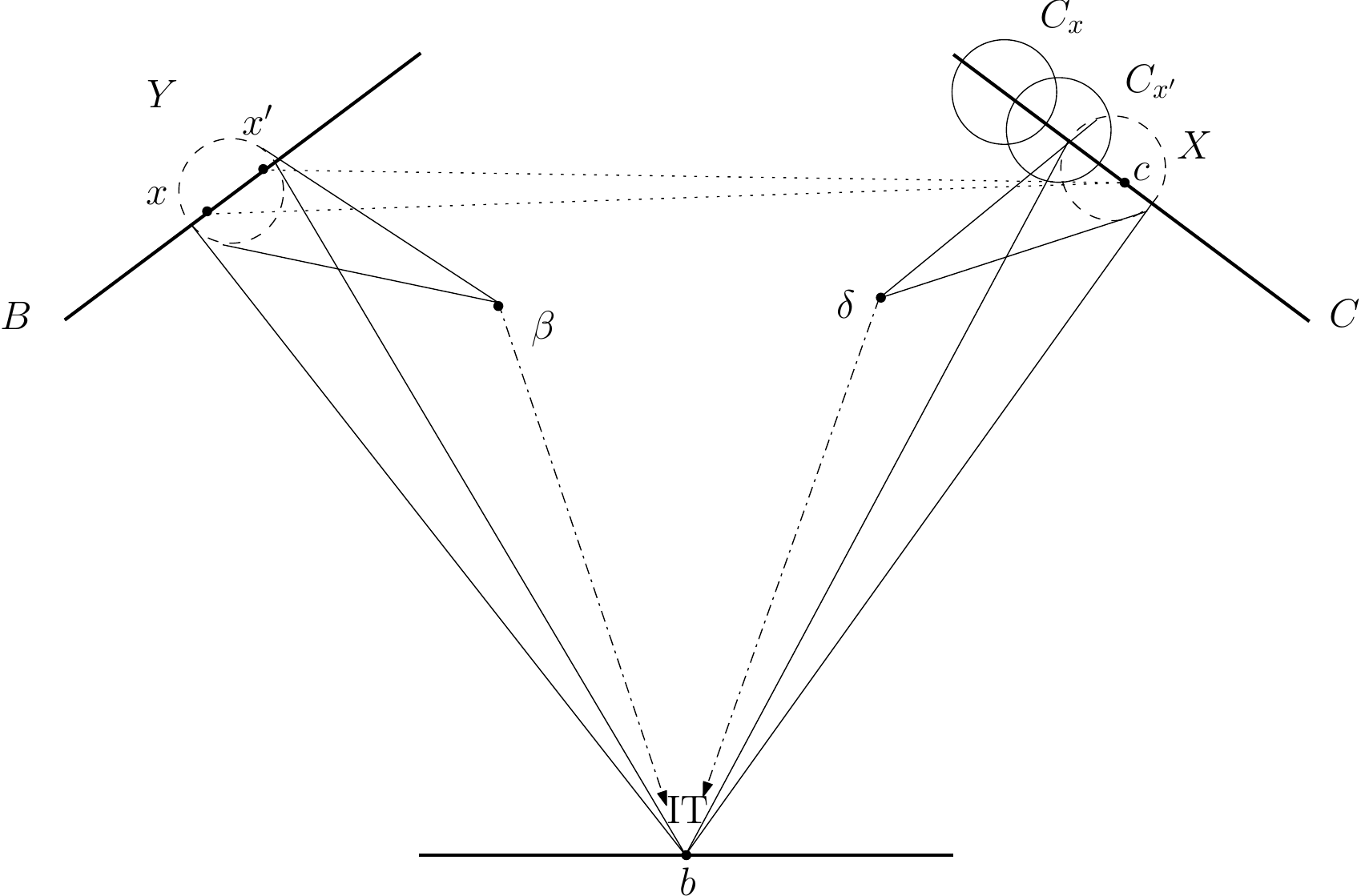}
\]

By homogeneity, there is an automorphism $\sigma$ of $M$ fixing $ac$ taking $x$ to $x'$, but then $\sigma$ takes $S(x)\cap S(b)\cap C=C_x\cap X\neq\varnothing$ to $S(x')\cap S(b)\cap C=C_{x'}\cap X=\varnothing$, a contradiction.
\end{case}

If $2r>n$, then we can carry out the same argument using $T$-neighbourhoods instead of $S$-neighbourhoods.
\end{proof}

The argument in the preceding proposition actually proves something stronger:

\begin{corollary}\label{CorExistsPartition}
	Let $M$ be an imprimitive simple unstable homogeneous 3-graph with $R$-classes of size $n<\omega$, and suppose that $\aut(M)$ acts 2-transitively on $M/R$. Let $A=a/R,B=b/R$ be distinct $R$-classes in $M$. Then there exists a partition $A=A_1\cup A_2$, $B=B_1\cup B_2$ such that the structure induced by $M$ on $A_i\cup B_i$ is $K_{r,r}^S$ and there are no more $S$-edges in $A\cup B$.
\end{corollary}
\begin{proof}
	Suppose that there exist $c,c'\in A$ such that $B_c=B_c'$. Then we can define an equivalence relation $\sim_B$ on $A$ by the formula $$\varphi(u,v;a,b): (\bar R(a,u)\wedge \bar R(a,v))\rightarrow\forall x(\bar R(x,b)\rightarrow(S(x,u)\leftrightarrow S(x,v)))$$Where $\bar R$ is the reflexive closure of $R$. Therefore, we have that for all $c,c'\in A$ either $B_c\cap B_{c'}=\varnothing$ or $B_c=B_{c'}$, so that a set of representatives for $\sim_B$ classes in $A$ induces a partition of $B$ into sets of size $r$. By symmetry of $S$ and homogeneity, the $\sim_B$-classes are also of size $r$, and by the definition of $\sim_B$ the structure induced on $c/\sim_B\cup B_c$ is $K_{r,r}^S$. The argument from Case \ref{Case1} Lemma \ref{Lemma1} shows that in this situation there are exactly two equivalence classes.

	The argument from Case \ref{Case2} in Lemma \ref{Lemma1} says that we cannot have $B_c\neq B_{c'}$ for all $c\neq c'$ in $A$.
\end{proof}

In other words, in a homogeneous imprimitive 3-graph with finite classes, the classes have size $2r$ for some $r\geq 1$ and the structure induced on a pair of distinct $R$-classes is that of two disjoint copies of $K_{r,r}^S$ (and therefore the rest of the edges form two disjoint copies of $K_{r,r}^T$). Given two distinct $R$-classes $B,C$ and subsets $B_1\subset B, C_1\subset C$, we write $S(B_1,C_1)$ if the structure induced on $B_1\cup C_1$ is a complete bipartite graph (where we interpret $S$ as edges and $R$ as nonedges). 

\begin{remark}\label{RmkInvariantSubsets}
The equivalence relations on $A$ and $B$ defined in the proof of Corollary \ref{CorExistsPartition} are invariant under automorphisms fixing some $a\in A$ and $b\in B$.
\end{remark}

\begin{observation}
	Let $M$ be an imprimitive homogeneous simple unstable 3-graph in which the reflexive closure of $R$ is a nontrivial proper equivalence relation on $M$ with finite classes, and suppose that $\aut(M)$ acts 2-transitively on $M/R$. Then $\SU(M)=1$.
\end{observation}
\begin{proof}
By transitivity, there is a unique 1-type over $\varnothing$, $p$. Let $q\in S(A)$ be an extension of $p$ to $A$. We have two possibilities:
\begin{enumerate}
\item{The type $q(x)$ contains the formula $R(x,a)$ for some $a\in A$. In this case, $q$ is clearly an algebraic extension, and therefore forking and of rank 0, so $p$ is of rank 1.}
\item{The type $q(x)$ does not contain $R(x,a)$ for any $a\in A$. Then $q$ is the type of an element in an $R$-class that is not represented in $A$. By Observation \ref{ObsIndepClasses}, $q$ is a nonforking extension of $p$.}
\end{enumerate}
\end{proof}

\begin{lemma}\label{Lemma2}
	Let $M$ be an imprimitive homogeneous simple unstable 3-graph in which the reflexive closure of $R$ is a nontrivial proper equivalence relation on $M$ with finite classes, and suppose that $\aut(M)$ acts 2-transitively on $M/R$. Then the $R$-classes have size 2.
\end{lemma}
\begin{proof}
The proof of this proposition depends on Observations \ref{ObsNoFER}, \ref{ObsRFree}, and \ref{ObsIndepClasses}.

Suppose for a contradiction that the $R$-classes in $M$ have four or more elements. Consider two distinct classes $A$ and $B$. We know from Corollary \ref{CorExistsPartition} that there exist $r$-sets $A_0,A_1$ and $B_0,B_1$ such that $S(A_0,A_1)$ and $S(B_0,B_1)$. By Observation \ref{ObsIndepClasses}, $A\indep B$; take $r$-sets $A_2\subset A, B_2\subset B$ such that all of $A_2\cap A_0, A_2\cap A_1, B_2\cap B_0, B_2\cap B_1$ are nonempty (we can do this because $r\geq 2$. By monotonicity of independence, $A_2\indep B_2$. The formulas $\bigwedge_{a\in A_2} S(x,a)$ and $\bigwedge_{b\in B_2}S(x,b)$ isolate nonforking extensions $p_{A_2}, p_{B_2}$ of the unique type over $\varnothing$; any $c_0\models p_{A_2}, c_1\models p_{B_2}$ in distinct $R$-classes are independent realisations of nonforking extensions of $p$ and satisfy the same Lascar strong type over $\varnothing$ by Observation \ref{ObsNoFER}, so we can use the Independence Theorem to find $c\indep AB$ that is $S$-related to $A_2$ and $B_2$. Similarly, we can find some $d$ that is $S$-related to $A_0$ and $B_2$. Now if we take $a\in A_2\cap A_0,b\in B_1\cap B_2$ then $abc$ and $abd$ form triangles of type $SST$ with $T(a,b)$. 

\[
\includegraphics[scale=0.9]{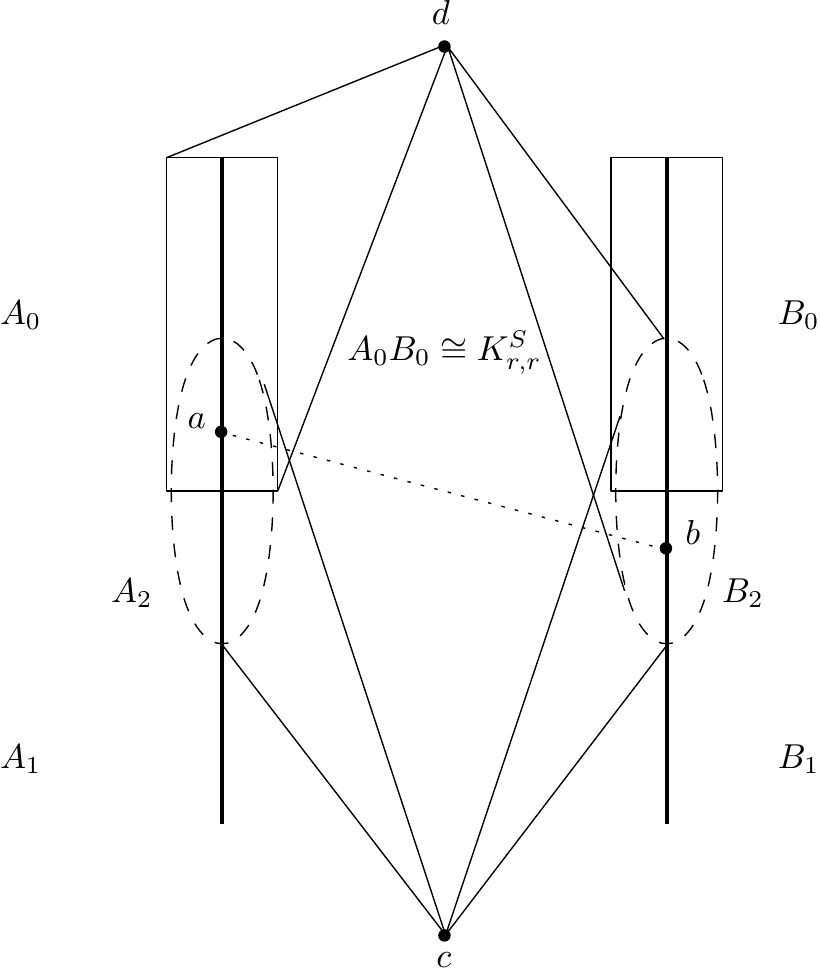}
\]

The partial isomorphism $a\mapsto a,b\mapsto b, c\mapsto d$ cannot be extended to an automorphism by Remark \ref{RmkInvariantSubsets}, since it fixes $a$ and $b$, and therefore it fixes their $\sim$ equivalence class, but at the same time the automorphism should take the $\sim$-class of $a$, $A_0=S(d)\cap A$, to $A_2=S(c)\cap A$. This contradicts the homogeneity of $M$.
\end{proof}

We now count with all the ingredients for the proof of our main result.
\begin{theorem}\label{ThmCGamma}
	Up to isomorphism, the only imprimitive simple unstable homogeneous 3-graph with finite classes such that is all predicates are realised in the union of two classes is $C(\Gamma)$.
\end{theorem}
\begin{proof}
	Let $M$ denote an imprimitive simple unstable homogeneous 3-graph with finite classes. We know by Lemma \ref{Lemma2} that the $R$-classes in $M$ have size 2, and by Corollary \ref{CorExistsPartition} that the structure induced on any pair of $R$-classes is the graph on four vertices with two $R$-edges, two $S$-edges, and two $T$-edges, and each of these pairs of edges spans the four vertices.

	Our argument makes use of the Independence Theorem and Observation \ref{ObsIndepClasses}. Notice that, since the automorphism group of $M$ acts 2-transitively on the set of $R$-edges, then there are no invariant equivalence relations on the set of classes. Therefore, there is a unique Lascar strong type of $R$-classes/edges.

	Take any element $A$ of $\age(C(\Gamma))$. We may assume without loss of generality that $A$ is a union of $R$-classes. We argue inductively that the structure induced by $C(\Gamma)$ on any $n$-tuple of classes can be embedded into $M$. For $n=1$, this is clear as the structure is simply an $R$-edge. Similarly, by Corollary \ref{CorExistsPartition}, we have the result for $n=2$. Now suppose that we can embed the structure induced by $C(\Gamma)$ on any $n$-tuple of $R$-classes into $M$, say into $\bar a_1,\ldots,\bar a_n$. By Observation \ref{ObsIndepClasses}, $\bar a_1\indep \bar a_2,\ldots \bar a_n$; by the primitivity of the action of $\aut(M)$ on $M/R$, all classes have the same Lascar strong type over $\varnothing$. So let $\bar b_0\models\tp(\bar a_{n+1}/\bar a_1)$ and $\bar b_1\models\tp(\bar a_{n+1}/\bar a_2,\ldots,\bar a_n)$. These two types are realised in $M$ by quantifier elimination and the induction hypothesis. By the Independence Theorem, there exists a class $\bar b_{n+1}$ realising $\tp(\bar a_{n+1}/\bar a_1)\cup\tp(\bar a_{n+1}/\bar a_2,\ldots,\bar a_n)$, $\bar b_{n+1}\indep\bar a_1,\ldots,\bar a_{n}$. The structure induced by $M$ on $\bar a_1,\ldots,\bar a_n\bar b_{n+1}$ is isomorphic to the original structure in $C(\Gamma)$. This proves $\age(C(\Gamma))\subseteq\age (M)$; the same argument proves the equality, so $C(\Gamma)$ and $M$ are homogeneous structures of the same age. They are therefore isomorphic.
\end{proof}

This concludes the analysis of imprimitive unstable structures with finite classes such that all predicates are realised in the union of two classes. What happens if only two predicates are realised in the union of two classes? In that case, the structure induced on a pair of distinct classes $A,B$ is either $K_{n,n}^S$ or $K_{n,n}^T$; as a consequence, all sets containing exactly one element from each $R$-class are isomorphic.

\begin{observation}
	Let $M$ be a simple homogeneous 3-graph in which $R$ defines an equivalence relation. If $\aut(M)$ acts transitively, but not 2-transitively on $M/R$, then the $S,T$-graph induced on any $X\subset M$ containing exactly one element from each $R$-class is homogeneous.
\label{ObsInterpretedGraph}
\end{observation}
\begin{proof}
	Consider the graph defined on $M/R$ with predicates $\hat S,\hat T$ which hold of two distinct classes $a/R,b/R$ if for some/any $\alpha\in a/R,\beta\in b/R$ we have $S(\alpha,\beta)$ (respectively, $T(\alpha,\beta)$). This graph is clearly isomorphic to the graph induced on $X$.
\begin{claim}
	The graph interpreted in $M/R$ as described in the preceding paragraph is homogeneous in the language $\{\hat S,\hat T\}$.
\label{ClaimInterpretsHenson}
\end{claim}
\begin{proof}
	Let $\pi$ denote the quotient map $M\rightarrow M/R$. Given two isomorphic finite substructures $A,A'$ of $M/R$, then any transversals to $\pi^{-1}(A)$ and $\pi^{-1}(A')$ are isomorphic, so by the homogeneity of $M$ there exists an automorphism $\sigma$ taking $\pi^{-1}(A)$ to $\pi^{-1}(A')$. The map $\pi\sigma\pi^{-1}$ is an automorphism of $M/R$ taking $A$ to $A'$.
\end{proof}
And the result follows.
\end{proof}

The graph on $M/R$ described in Observation \ref{ObsInterpretedGraph} is interpretable in $M$, so it must be a simple graph; if $S,T$ are unstable in $M$ it follows that $\hat S,\hat T$ are unstable in $M/R$. By the Lachlan-Woodrow Theorem, the graph on $M/R$ is isomorphic to the Random Graph. 

\begin{corollary}\label{CorFiniteClasses}
	Let $M$ be a simple homogeneous 3-graph in which $R$ defines an equivalence relation. If $\aut(M)$ acts transitively, but not 2-transitively on $M/R$. Then $M\cong\Gamma[K_n^R]$ for some $n\in\omega$.
\end{corollary}

\chapter{Forking in Primitive Simple 3-graphs}\label{ChapStableForking}
\setcounter{equation}{0}
\setcounter{theorem}{0}
\setcounter{case}{0}
\setcounter{subcase}{0}

This chapter contains a proof of the following statement: given two vertices $a,b$ in a primitive simple homogeneous simple 3-graph $M$, if $\tp(a/b)$ divides over $\varnothing$, then the formula (relation) isolating $\tp(a/b)$ is stable. 

Supposing that we have a counterexample to that statement, since we must have at least one nonforking relation, the first distinction of cases is on the number of unstable forking relations, which can be either one or two.

This is a long argument by contradiction, and it can be at times confusing. The structure of the argument is as follows:
\[
\includegraphics[scale=0.8]{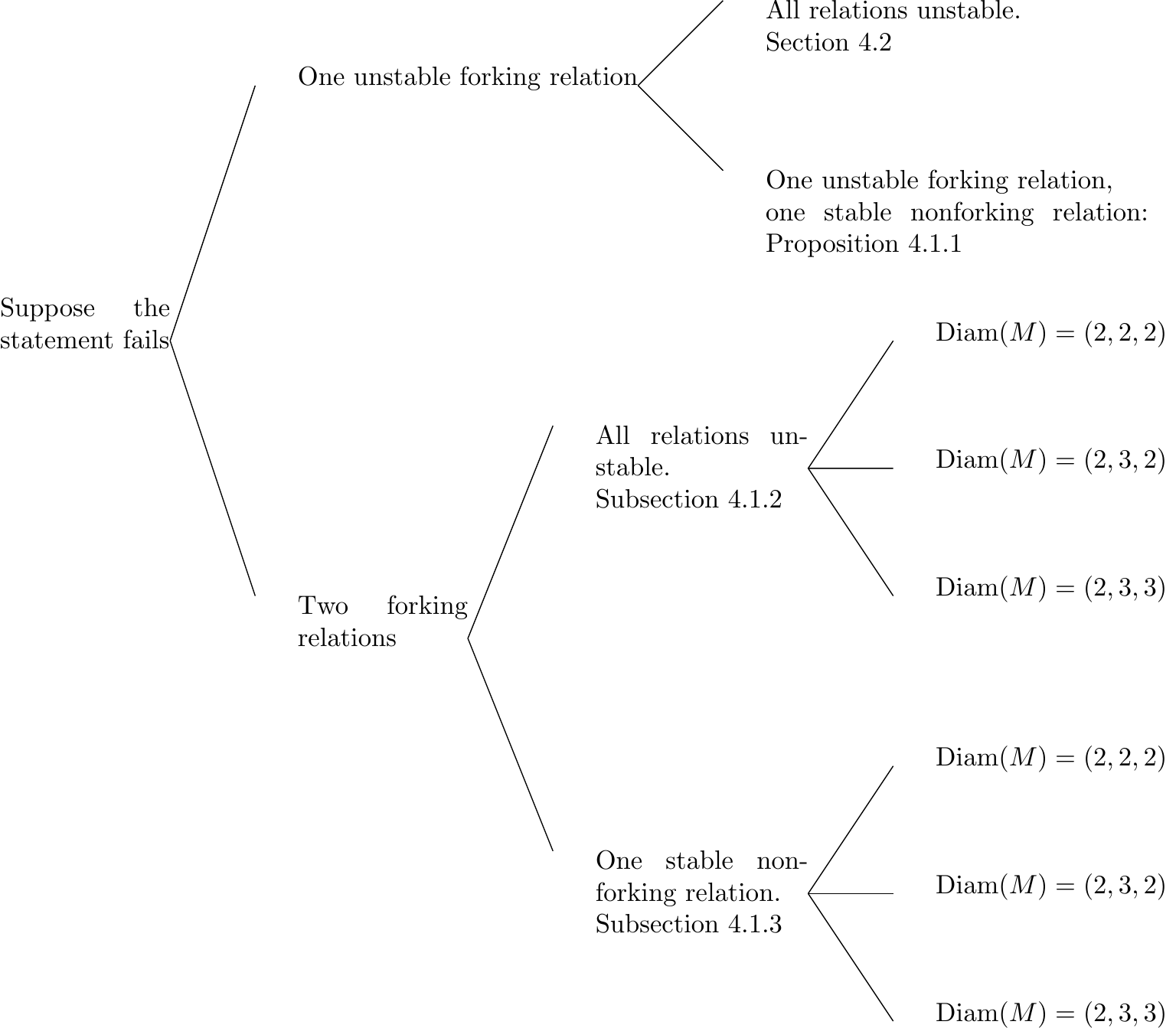}
\]

\section{Half-graphs in Primitive Homogeneous Simple 3-graphs}\label{SecUseful}
\begin{proposition}\label{UnstableNonforking}
	Let $M$ be a simple unstable primitive homogeneous $n$-graph. If there are at least two non-forking relations, then all the non-forking relations are unstable.
\end{proposition}
\begin{proof}
	Suppose that we have at least two nonforking relations, $R_1$ and $R_2$. We will prove that all finite $\{R_1,R_2\}$-graphs can be embedded into $M$, so in particular we can find an infinite half-graph for $R_2$.

We will prove by induction that we can embed every finite graph of size $m$ as an independent set (meaning that if $a_1,\ldots,a_m$ are the vertices of the induced graph, then $a_i\indep[\varnothing]a_1\ldots a_{i-1}a_{i+1}\ldots a_m$). The case $m=2$ is our basis for induction, and follows trivially from the fact that $R_1$ and $R_2$ are nonforking.

Suppose that we can embed all $\{R_1,R_2\}$-graphs of size at most $k$ as independent sets in $M$, and we wish to embed an $\{R_1,R_2\}$-graph $G$ of size $k+1$ into $M$. Enumerate the vertices of the graph as $v_1,\ldots,v_{k+1}$. We can embed the subgraph induced on $v_1,\ldots,v_k$ as an independent set $a_1,\ldots,a_k$ in $M$. We have in particular $a_k\indep a_1\ldots a_{k-1}$. By applying the induction hypothesis again, we can find a $\beta$ which satisfies the same atomic formulas over $a_1,\ldots,a_{k-1}$ as $v_{k+1}$ over $v_1,\ldots,v_{k-1}$. Similarly, we can find $\beta'\indep a_k$ satisfying the same atomic relation over $a_k$ as $v_{k+1}$ over $v_k$. We can apply the Independence Theorem to find a common solution $b$ to $\tp(\beta/a_1\ldots a_{k-1})$ and $\tp(\beta'/a_k)$ such that $b\indep a_1\ldots a_{k+1}$. The graph induced by $M$ on $a_1,\ldots,a_k,b$ is isomorphic to $G$
\end{proof}

In the case of simple 3-graphs, we have at least one and up to three non-forking relations. Under primitivity and homogeneity, assuming that the atomic formula isolating $\tp(b/a)$ is stable if $\tp(b/a)$ divides over $\varnothing$ and Koponen's result on the finiteness of rank of binary homogeneous simple structures, it is not too hard to prove that if we have three non-forking relations then each of them is unstable and $M$ is the random 3-graph (see Theorem \ref{PrimitiveAlice}). If we have two non-forking relations, then they are unstable; the remaining relation could be stable and forking (in which case we have stable forking in the formulation we have chosen for this document), or it could be unstable and forking.


We accumulate in this subsection some easy results, the conclusions of which are used repeatedly in the main proofs. 

In the course of the proofs, we will make extensive use of the Lachlan-Woodrow Theorem \ref{LachlanWoodrow} from \cite{lachlan1980countable}.

\begin{remark}\label{RmkSimpleGraphs}
From the list list of structures in the Lachlan-Woodrow Theorem, graphs in the first category are $\omega$-stable of \SU-rank 1 if just one of $m,n$ is $\omega$; the graph $I_\omega[K_\omega]$ is of rank 2. The Random Graph is supersimple unstable of \SU-rank 1, and the homogeneous $K_n$-free graphs are not simple.
\end{remark}

\begin{definition}
Given an unstable theory $\mathcal T$ and $M\models\mathcal T$, we say that a predicate $P$ is unstable in a set $X\subset M$ if we can find witnesses to the instability of $P$ within the set $X$, that is, if there exist $\bar a_i,\bar b_i\in X$ ($i\in\omega$) such that $P(\bar a_i,\bar b_j)$ holds iff $i\leq j$. Similarly, we say that a predicate $P$ is nonforking in a set defined by a formula $\varphi(\bar x,\bar a)$ if the formula $\varphi(\bar x,\bar a)\wedge\varphi(\bar a,\bar b)\wedge P(\bar x,\bar b)$ does not fork over $\bar a$. If no set $X$ is specified, $X=M$ is assumed.
\end{definition}

\begin{definition}\label{DefHG}
	A \emph{half-graph} for colour $P$ in an $n$-graph is a graph induced on set of vertices $a_i,b_i$ ($i\in\omega$) witnessing the instability of the formula $P(x,y)$.
\end{definition}

We will use half-graphs in the rest of the argument because of the information that they give us about the structure of the neighbourhoods of a vertex in a homogeneous 3-graph.

\begin{proposition}\label{IndiscerniblePairs}
	Suppose that $(a_i,b_i)_{i\in\omega}$ is an infinite half-graph for some relation $R$ in a homogeneous $n$-graph. Then we can find an infinite half-graph $(a_i',b_i')_{i\in\omega}$ that is indiscernible as a sequence of pairs $a_i'b_i'$ of type $R(a,b)$.
\end{proposition}
\begin{proof}
	The half-graph $(a_i,b_i)_{i\in\omega}$ is an infinite sequence of pairs of type $R$. Colour the pairs of $R$-edges $(a_ib_i,a_jb_j)$, $i\neq j$ according to the type of the four-vertex set $\{a_i,b_i,a_j,b_j\}$. By Ramsey's theorem, there is an infinite monochromatic $X\subset\omega$ for this colouring. The set of pairs $\{a_ib_i:i\in X\}$ is indiscernible because the language is binary, and it is a half-graph for $R$ with the ordering induced by $\omega$ on $X$.
\end{proof}

\begin{definition}
	We denote the set of relations from $L$ which are unstable in models of a theory $\mathcal T$ as $L^u(\mathcal T)$. Given two distinct binary relation symbols $R,R'$ in $L^u(\mathcal T)$, we say that $R$ and $R'$ are \emph{compatible} if there exists an indiscernible sequence of pairs of type $R$ which witnesses the instability of $R$ and $R'$. We denote compatibility by $R\sim_{\mathcal T}R'$, omitting $\mathcal T$ whenever it is clear which theory we are referring to. Given witnesses $\alpha_i,\beta_i$ to the compatibility $R\sim_{\mathcal T}R'$, we refer to the sets $A=\{\alpha_i:i\in\omega\}$ and $B=\{\beta_i:i\in\omega\}$ as the \emph{horizontal} cliques of $R\sim_{\mathcal T}R'$.
\end{definition}

\begin{observation}\label{IndiscernibleHalfGraphs}
	Let $M$ be an unstable $n$-graph and suppose that $R$ is an unstable relation. Then there exists $R'\in L^u(\Th(M))$ such that $R$ and $R'$ are compatible. 
\end{observation}
\begin{proof}
	By instability of $R$, there exist parameters $a_i$ and $b_i$, $i\in\omega$, such that $R(a_i,b_j)$ holds iff $i\leq j$. Consider the set of $R$-edges $\{a_ib_i:i\in\omega\}$, and colour the pairs of distinct edges in this set according to the 4-type they satisfy. By Ramsey's theorem, there is an infinite monochromatic set. As a set of vertices, this set witnesses the instability of $R$, and because it is indiscernible over $\varnothing$ as a set of $R$-edges, all the edges $a_ib_j$ with $j<i$ are of the same type $R'\neq R$.
\end{proof}

\begin{definition}
	We call a set of parameters witnessing the compatibility of two unstable relations $R,R'$ an \emph{indiscernible half-graph} for $R,R'$.
\end{definition}

\begin{remark}\label{RmkCompatibilityGraph}
	Compatibility is a graph relation on $L^u(\mathcal T)$ when $\mathcal T$ is the theory of an infinite $n$-graph ($L=\{R_1,\ldots,R_n\}$). The compatibility graph on $L^u(\mathcal T)$ has no isolated vertices (all vertices have degree at least 1). In particular, when $|L|=3$ and all relations are unstable, the compatibility graph is connected.
\end{remark}

Indiscernible half-graphs for $R,R'$ give us valuable information about the $R$- and $R'$-neighbourhoods of vertices.

\begin{proposition}\label{PropCliques}
	Let $M$ be a simple primitive homogeneous $n$-graph in which $R$ divides and $S_1,\ldots,S_k$ are the nonforking relations with respect to $\Th(M)$ in the language of the $n$-graph. Then for any $a\in M$ the set $R(a)$ does not contain any infinite $S_i$-cliques for all $i\in\{1,\ldots k\}$.
\end{proposition}
\begin{proof}
	It follows from Observation \ref{MorleySqn} that any Morley sequence over $\varnothing$ is an infinite $S_i$-clique for some $i\in\{1,\ldots,k\}$. By simplicity, dividing is witnessed by Morley sequences, so for any $S_i$-clique enumerated as $(a_j:j\in\omega)$, the set $\{R(x,a_j):j\in\omega\}$ is inconsistent.
\end{proof}

As a consequence,
\begin{observation}\label{NoInfiniteCliques}
	If $(M;R,S,T)$ is a homogeneous primitive simple 3-coloured graph in which $R$ is a forking relation and $S, T$ are nonforking, then there are no infinite $S$- or $T$-cliques in $R(a)$.
\end{observation}\hfill$\Box$

\begin{notation}
	Whenever we draw a 3-graph, $R$ is represented by plain lines, $S$ by dashed lines, and $T$ by dotted lines.
\end{notation}

In any primitive $\omega$-categorical $n$-graph $M$, we necessarily have finite diameter for each of the predicates $R_i$, since $R_i$-connectivity is an equivalence relation and there are only finitely many types of pairs of elements (see Observation \ref{ObsFiniteDiameter}). We denote the $R_i$-diameter of $M$ by $\diam_{R_i}(M)$. Given a predicate $R_i$, we denote the set of elements at $R_i$-distance $m$ from $a$ by $R_i^m(a)$.

\begin{definition}\label{DefRefineComp}
We will use the notation $R\sim^ST$ to indicate that there exists an indiscernible half-graph witnessing $R\sim T$ such that one of the horizontal cliques is of colour $S$.
\end{definition}
\begin{proposition}\label{PropHomNeigh}
Let $M$ be a homogeneous binary structure and $S$ a predicate in the language of $M$. Then $S(a)$ is a homogeneous $m$-graph for some $m\leq n$.
\end{proposition}
\begin{proof}
$A\cong B$ for finite $A,B\subset S(a)$ implies $aA\cong aB$, so by the homogeneity of $M$ there is $\sigma\in\aut(M/a)$ taking $A$ to $B$ ($\sigma$ clearly fixes $S(a)$ setwise). The restriction of $\sigma$ to $S(a)$ is an automorphism of $S(a)$ taking $A$ to $B$.
\end{proof}
\begin{observation}
	Let $M$ be a primitive simple homogeneous graph and suppose that $S$ divides, $\diam_S(M)=3$, and $S^2(a)=T(a)$. If $T\sim^SS$, then $S(a)$ is isomorphic to the Random Graph.
\label{ObsIsoRG}
\end{observation}
\begin{proof}
	It follows from $\diam_S(M)=3$ and $S^2(a)=T(a)$ that $S(a)$ is $R$-free. Consider an indiscernible graph witnessing $T\sim^SS$:
\[
\includegraphics[scale=0.7]{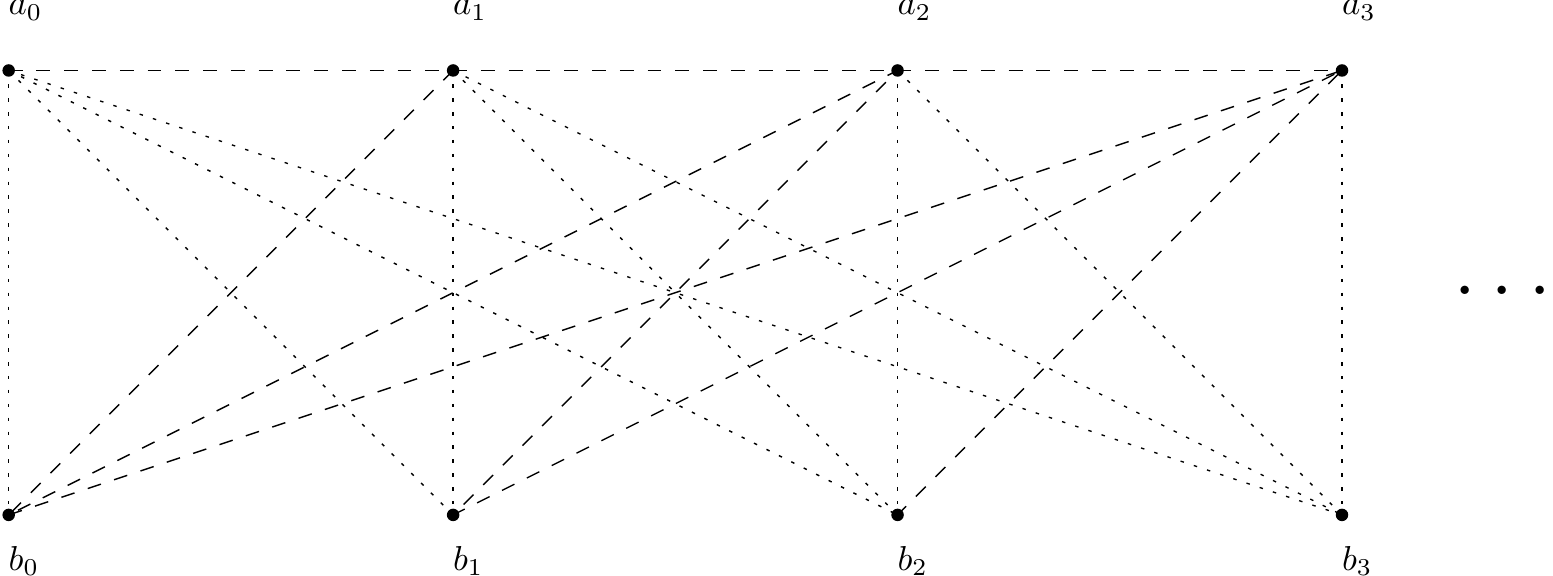}
\]
The indiscernible half-graph sketched above can be embedded in $S(a)$ by compactness and homogeneity, since each initial segment $(a_ib_i)_{i\leq n}$ can be embedded in $S(a)$ by transitivity of $M$ and the fact that $(a_ib_i)_{i\leq n}$ is contained in $S(a_{n+1})$. Therefore $S(a)$ is isomorphic to a homogeneous (by Proposition \ref{PropHomNeigh}) unstable graph, which must be simple because it is interpretable in $M$. The result follows from Remark \ref{RmkSimpleGraphs}.
\end{proof}

\begin{proposition}
	Let $M$ be a primitive simple homogeneous graph and suppose that $S$ divides, $R$ is nonforking, $\diam_S(M)=2$ and $S\sim^ST$. If $R$ defines an equivalence relation on $S(a)$, then $S$ and $T$ are nonforking in $S(a)$, and $S(a)$ is isomorphic to $C(\Gamma)$ or to $\Gamma^{S,T}[K_n^R]$ for some $n\in\omega$.
\end{proposition}
\begin{proof}
	It follows from Proposition \ref{PropCliques} that the $R$-classes in $S(a)$ are finite. From $S\sim^ST$ we get that $S,T$ are unstable in $S(a)$. The conclusions follow from Theorem \ref{ThmCGamma} and Corollary \ref{CorFiniteClasses}.
\end{proof}

\begin{proposition}\label{PropTwoInfiniteCliques}
	Let $M$ be a primitive simple unstable homogeneous 3-graph. Then $M$ embeds infinite monochromatic cliques in at least two colours.
\end{proposition}
\begin{proof}
	At least one of the relations, say $R$, is nonforking, so by the Independence Theorem and homogeneity $M$ embeds infinite $R$-cliques. Of the other relations, at least one, without loss of generality $S$, is unstable, and therefore non-algebraic. 

If $S$ is nonforking, then there are infinite $S$-cliques by the Independence Theorem. And if $S$ divides, then the $S$-neighbourhood of any $a\in M$ does not embed infinite $R$-cliques by Proposition \ref{PropCliques}, but it must embed an infinite monochromatic clique by Ramsey's Theorem. 
\end{proof}

\begin{proposition}\label{PropNoSnotTnoCliquesR}
	There are no primitive homogeneous simple 3-graphs $M$ in which $R\sim^SS$, $T$ is nonforking, $S$ and $R$ divide, and $R$ does not form infinite cliques.
\end{proposition}
\begin{proof}
	Suppose for a contradiction that such an $M$ exists, and consider $S(a)$ for some $a\in M$. There are two cases to analyse:
\begin{case}
	If $SST\in\age(S(a))$, then we make the following claim:
\begin{claim}
	$S(a)$ is primitive. 
\end{claim}
\begin{proof}
From $S\sim^SR$, we get that $S$ and $R$ are unstable in $S(a)$, so the only formulas that could define an equivalence relation in $S(a)$ are $T$ and $S\vee R$. 

To eliminate the possibility of $S\vee R$ defining an equivalence relation, note that if it did define one, then by the instability of $S,R$ its classes would be infinite. And since $T$ does not form infinite cliques  in $S(a)$, $S\vee R$ would have only finitely many infinite classes. Each class is a simple unstable homogeneous graph, isomorphic to the Random Graph. This contradicts the hypothesis that $R$ does not form infinite cliques.

Finally, $T$ does not define an equivalence relation because in that case we could also find infinite $R$-cliques within $S(a)$, by Theorem \ref{ThmCGamma} and Corollary \ref{CorFiniteClasses}.
\end{proof}
From this claim it follows that $S(a)$ is a homogeneous primitive simple 3-graph in which there are no infinite $R$-or $T$-cliques, contradicting Proposition \ref{PropTwoInfiniteCliques}.
\end{case}
\begin{case}
	If $SST\not\in\age(S(a))$, then $S(a)$ is an unstable homogeneous $R,T$-graph, again contradicting (by the Lachlan-Woodrow Theorem) the hypothesis of no infinite $R$-cliques in $M$.
\end{case}
\end{proof}
\setcounter{case}{0}
\setcounter{subcase}{0}

\begin{proposition}\label{PropImprimitiveInfClasses}
	There are no homogeneous simple 3-graphs in which $S$ defines an equivalence relation with infinitely many infinite classes, $R$ and $T$ are unstable, and $R$ does not form infinite cliques.
\end{proposition}
\begin{proof}
	Under these hypotheses, $T$ is the only nonforking relation in the language of $M$: $S$ clearly divides as it is an equivalence relation with infinitely many classes; it is not possible for $R$ to be nonforking because there is only one strong type of elements in $M$, so the Independence Theorem would give us infinite $R$-cliques if $R$ were non-dividing. 

	There are two cases, depending on whether the triangle $RST$ embeds into $M$ or not. 
\begin{case}
If $RST$ does not embed into $M$, then the structure induced on the union of a pair of distinct $S$-classes $K, K'$ is either $T$-free or $R$-free. Therefore, the half-graphs witnessing $R\sim T$ are transversal to an infinite number of $S$-classes; in other words, any transversal to all $S$-classes is unstable and $K_n^R$-free.

	Consider the infinite graph defined on $M/S$ with predicates $\hat R,\hat T$ which hold of two distinct classes $a/S,b/S$ if for some/any $\alpha\in a/S,\beta\in b/S$ we have $R(\alpha,\beta)$ (respectively, $T(\alpha,\beta)$). 
\begin{claim}
	The graph interpreted in $M/S$ as described in the preceding paragraph is homogeneous in the language $\{\hat R,\hat T\}$.
\end{claim}
\begin{proof}
By the same argument as in Corollary \ref{CorFiniteClasses}.
\end{proof}

	As a consequence of the claim, the Lachlan-Woodrow Theorem, and the facts that in $M$ the predicates $R$ and $T$ are unstable, and that $M$ does not embed infinite $R$-cliques, we have that $M/S$ is isomorphic to some universal homogeneous $K_n$-free graph (for some $n\in\omega$), which are not simple. This contradicts the simplicity of $M$.
\end{case}
\begin{case}
If $RST$ embeds into $M$, then for any $a\in M$ the set $R(a)$ meets every $S$-class in $M$ not containing $a$. The reason is that, since $RST\in\age(M)$, then there exists an $S$-class with two elements $c,c'$ such that $R(a,c)\wedge T(a,c')$. An element $b$ in any class that does not contain $a$ satisfies $R(a,b)$ or $T(a,b)$, so by homogeneity $ab\cong ac$ or $ab\cong ac'$, and by homogeneity there is a $b'$ in the same class as $b$ that satifies the other formula.

By the usual argument, there are no infinite $T$-cliques in $R(a)$ for any $a\in M$, and by the hypothesis of no infinite $R$-cliques in $M$, we get that the only infinite cliques in $R(a)$ are $S$-cliques. But a transversal to $R(a)$ must contain an infinite monochromatic clique, by Ramsey's Theorem. We have reached a contradiction.
\end{case}
\end{proof}

\section{One unstable forking relation}\label{SubsecOneUnstable}

If $R$ is the only forking relation, and it is unstable, then it follows from Proposition \ref{UnstableNonforking} that all relations are unstable. Let us look more closely into the infinite half-graphs witnessing the instability of the forking relation $R$. By Remark \ref{RmkCompatibilityGraph}, the compatibility graph is connected.

\begin{observation}\label{UnstableInR}
	Let $M$ be a primitive homogeneous simple 3-graph, and suppose that $R$ is a forking unstable relation and $S,T$ are nonforking and unstable. Then we can find witnesses to the instability of $R$ within $R(a)$.
\end{observation}
\begin{proof}
	Let $(a_i,b_i)_{i\in\omega}$ be a half-graph for $R$. Since $S,T$ are nonforking, it follows from Proposition \ref{PropCliques} that $R(a)$ does not contain infinite $S$- or $T$- cliques. We know that $R(a)$ is an infinite set because $R$ is unstable, so the only infinite cliques in $R(a)$ are of colour $R$. From this it follows that when we extract an indiscernible half-graph from a set of witnesses for the instability of $R$ as in Proposition \ref{IndiscerniblePairs}, then the $a_i$ and the $b_i$ form infinite $R$-cliques. Clearly all the elements of the indiscernible half-graph are in $R(a_0)$.
\end{proof}

\begin{remark}
	As a direct consequence of Observation \ref{UnstableInR}, there are no primitive homogeneous simple 3-graphs $M$ such that $R$ is forking and unstable and the $R$-diameter of $M$ is 3, as in this case $R(a)$ would be an infinite simple unstable homogeneous graph, isomorphic to the Random Graph by the Lachlan-Woodrow Theorem. But in this is imposible since there are no infinite $S$- or $T$-cliques in $R(a)$.
\end{remark}

\begin{lemma}
	There are no primitive homogeneous simple 3-graphs in which all the predicates are unstable and only one of them, $R$, divides.
\label{PropNoSimpleOneDividing}
\end{lemma}
\begin{proof}
	Suppose for a contradiction that $M$ is a 3-graph as in the statement of this proposition. By simplicity and instability of $R$, $R(a)$ is an infinite simple 3-graph not embedding infinite cliques of colour $S$ or $T$. By Observation \ref{UnstableInR}, it is unstable. It follows from Proposition \ref{PropTwoInfiniteCliques} that $R(a)$ is not primitive.

	By Observation \ref{UnstableInR}, there is at least one more predicate that is unstable in $R(a)$. Let us suppose without loss of generality that $S$ is such a predicate. From the instability of $R,S$ we get directly that $R,S,R\vee T,S\vee T$ do not define equivalence relations. If $T$ defined an equivalence relation on $R(a)$, then its classes would be finite and by Theorem \ref{ThmCGamma} and Corollary \ref{CorFiniteClasses} we could find infinite $T$-cliques. And if $R\vee S$ defines an equivalence relation, then it has finitely many infinite classes (since $R(a)$ does not embed infinite $T$-cliques by Proposition \ref{PropCliques}), each of which is a homogeneous unstable graph, isomorphic to the Random Graph by the Lachlan-Woodrow Theorem. Again we find infinite $S$-cliques in $R(a)$, a contradiction. So there are no invariant proper nontrivial equivalence relations on $R(a)$, contradicting the first paragraph of this proof.
\end{proof}

\section{Unstable 3-graphs not embedding an infinite $R$-clique.}

\begin{proposition}
	Let $M$ be a simple homogeneous 3-graph in which $S$ and $T$ are nonforking relations, and which does not embed infinite $R$-cliques. Then $M$ is imprimitive.
\label{PropImprimitivity}
\end{proposition}
\begin{proof}
	Suppose for a contradiction that $M$ is primitive. Then $R$ is a forking relation, since otherwise the Independence Theorem and homogeneity would allow us to find arbitrarily large $R$-cliques. Since $M$ is primitive, we have that each of $R,S,T$ is non-algebraic, by $\omega$-categoricity. Consider $R(a)$ for any $a\in M$; this is an infinite set which cannot contain infinite cliques of colour $S$ or $T$ because that would witness that $R$ is nonforking ($S$- and $T$-cliques form Morley sequences over $\varnothing$), and cannot contain infinite $R$-cliques either, because $M$ does not embed infinite $R$-cliques. This contradicts Ramsey's theorem.
\end{proof}
\begin{remark}
Notice that if $M$ is a primitive simple homogeneous 3-graph not embedding infinite $R$-cliques and satisfying $S\sim^ST, S\sim ^TT$, then $S$ and $T$ are nonforking. More generally, if $R$ divides and $S\sim^ST, S\sim ^TT$, then $S$ and $T$ are nonforking.
\end{remark}
\begin{proposition}
	Let $M$ be a simple unstable homogeneous 3-graph in which $R$ is a stable relation and $S,T$ form infinite cliques. If $M$ does not embed infinite $R$-cliques then $M$ is imprimitive.
\end{proposition}
\begin{proof}
	Since $R$ is stable, we have $S\sim T$, and as $M$ does not embed infinite $R$-cliques, we have either $S\sim^ST$ or $S\sim^TT$. Suppose for a contradiction that $M$ is primitive. Then $R$ is not algebraic, by $\omega$-categoricity, and divides, by primitivity and the Independence Theorem. This implies that one of the unstable relations, say $S$, is nonforking. 

	If $S\sim^ST$ and $S$ does not divide, then $T$ does not divide by simplicity (as $T(b)$ contains infinite $S$-cliques), so by Proposition \ref{PropImprimitivity} $M$ is imprimitive.

	Therefore, we must have $S\not\sim^ST$ and $S\sim^TT$ and $T$ divides as otherwise we could use the Independence Theorem to embed each finite substructure of an indiscernible half-graph witnessing $S\sim^ST$ into $M$. Therefore $T(a)$ does not contain any infinite $S$- or $R$-cliques and is imprimitive by Proposition \ref{PropTwoInfiniteCliques}. From $S\sim^TT$ we get that $S$, $T$ are unstable in $T(a)$ and therefore $S,T,R\vee S,R\vee T$ do not define equivalence relations in $T(a)$. The formula $S\vee T$ does not define an equivalence relation because its classes would be isomorphic to the Random Graph, which is impossible as $T(a)$ does not contain infinite $S$-classes. So $R$ must define an equivalence relation with finite classes, and is isomorphic to $C(\Gamma^{ST})$ or $\Gamma^{ST}[K_n^R]$. In any case, $T(a)$ embeds infinite $T$-cliques. We have reached a contradiction in every possible case stemming from the assumption of primitivity, so $M$ must be imprimitive.
\end{proof}

\comm
\begin{proposition}
	Let $M$ be a primitive simple unstable 3-graph without infinite $R$-cliques. If $R\sim S$, then $M$ is semilinear.
\end{proposition}
\begin{proof}
	By Proposition \ref{PropOneIsStable}, $T$ is stable. Since $R$ does not form infinite cliques, we have $R\sim^SS$ or $R\sim^TS$ and by the usual arguments $R$ divides.
\begin{claim}
	$R\not\sim^SS$.
\end{claim}
\begin{proof}
	Suppose for a contradiction $S\sim^SR$, so $R$ and $S$ are unstable in $S(a)$. Then $S$ divides because $R$ divides and there are infinite $S$-cliques in $R(a)$. This implies that $T$ is a nonforking relation. By Proposition \ref{PropTwoInfiniteCliques}, $S$ and $T$ form infinite cliques in $M$. The instability $R\sim^SS$ also implies $SSR\in\age(M)$. We have two cases, depending on the $S$-diameter of $M$
\begin{enumerate}
\item{If $\diam_S(M)=3$, then $SST\not\in\age(M)$ and $S(a)$ is an $RS$-graph without infinite $R$-cliques. This is impossible by the Lachlan-Woodrow Theorem, as the only simple unstable graph is the Random Graph.}
\item{If $\diam_S(M)=2$, then $S(a)$ is a 3-graph in which $R,S$ are unstable, and $S(a)$ does not embed infinite $R$- or $T$-cliques. Then $S(a)$ is imprimitive by Proposition \ref{PropTwoInfiniteCliques} and one of $T$, $R\vee S$ defines an equivalence relation on $S(a)$. If $R\vee S$ defines an equivalence relation, then each class is isomorphic to the Random Graph and therefore embeds infinite $R$-cliques, a contradiction. And if $T$ defines an equivalence relation, then $S(a)$ is isomorphic to one of $C(\Gamma^{RS})$ (if $\aut(M/a)$ acts 2-transitively on $S(a)/T$) or $\Gamma^{RS}[K_n^T]$. In any case, there are infinite $R$-cliques.}
\end{enumerate}
Therefore, $R\not\sim^SS$.
\end{proof}
From this claim and $R\sim S$ it follows that $R\sim^TS$ holds. Therefore, $T$ divides, since $R$ divides and $R(a)$ contains infinite $T$-cliques, by $R\sim^TS$. Now we have three cases:
\begin{enumerate}
	\item{If $T(a)$ is a $TS$-graph, then $T(a)$ is an infinite stable graph. By the Lachlan-Woodrow Theorem, it is imprimitive, so by Proposition \ref{PropMultipartite} $T$ is an equivalence relation on $T(a)$ which must have finitely many infinite classes since there are no infinite $S$-cliques in $T(a)$. It follows that $M$ is semilinear of $T$-diameter 3.}
	\item{If $T(a)$ is a $TR$-graph, then $M$ is semilinear by the same argument as in the previous case.}
	\item{If $T(a)$ realises all predicates, then by Proposition \ref{PropTwoInfiniteCliques} $T(a)$ is imprimitive. }
\end{enumerate}
If $T(a)$ is an unstable 3-graph, then there are witnesses to $R\sim S$ in $T(a)$ and $T$ defines an equivalence relation ($R\vee S$ is not an equivalence relation because its classes would be isomorphic to the Random Graph) with infinite classes. It is not possible for $T$ to have infinitely many infinite classes as in that case $M$ would interpret a weak pseudoplane, contradicting Theorem \ref{ThmThomas}. Therefore $T$ has finitely many infinite classes and $M$ is semilinear. 

And if $T(a)$ is stable, then it is isomorphic to one of $P^S[K_\omega^S], Q^T[K_\omega^T], K_\omega^T\times K_n^S, K_\omega^T\times K_n^T$, or a wreath product $K_m^i[K_n^j[K_t^k]]$ where $\{i,j,k\}=\{R,S,T\}$, $i\neq T$ and the subindex corresponding to the $T$ subindex is $\omega$, by Lachlan's Theorem \ref{Lachlan3graphs} and Proposition \ref{PropMultipartite}. This leaves us with only one case to eliminate, namely $T(a)\cong K_m^S[K_\omega^T[K_n^R]$ or $T(a)\cong K_m^R[K_\omega^T[K_n^S]$, in all other cases $M$ is semilinear.

Suppose then that $T(a)$ is isomorphic to $K_m^S[K_\omega^T[K_n^R]$.
\end{proof}
\ent

\begin{proposition}
	Let $M$ be a homogeneous simple unstable 3-graph not embedding infinite $R$-cliques. If $R\sim^SS$, then $M$ is imprimitive.
\label{PropImprimitivity1}
\end{proposition}
\begin{proof}
	Suppose for a contradiction that $M$ is primitive. Then $R$ divides over $\varnothing$ by primitivity and the Independence Theorem, and $S$ divides because $R(a)$ contains infinite $S$-cliques, by $R\sim^SS$, and $T$ is the only nonforking relation. 

	Again by $R\sim^SS$ the homogeneous simple 3-graph induced on $S(a)$ is unstable, contains $S$- and $R$-edges and does not embed infinite $R$- or $T$-cliques. By the Lachlan-Woodrow Theorem, $T$ is realised in $S(a)$ because the only simple unstable graph is the Random Graph. By Proposition \ref{PropTwoInfiniteCliques}, $S(a)$ is imprimitive. Additionally, $R,S$ are unstable in $S(a)$, and so $R,S,R\vee T, S\vee T$ do not define equivalence relations in $S(a)$. Therefore, either $R\vee S$ is an equivalence relation with finitely many infinite classes in $S(a)$, or $T$ is an equivalence relation with finite classes on $S(a)$. It is not possible for $R\vee S$ to be an equivalence relation because in that case each class would be isomorphic to the Random Graph, contradicting that $S(a)$ does not contain infinite $R$-cliques. 

And if $T$ defines an equivalence relation with finite classes, then by Theorem \ref{ThmCGamma} and Corollary \ref{CorFiniteClasses}, all the imprimitive simple unstable homogeneous 3-graphs with finite classes embed infinite cliques in two colours, contradicting that $S(a)$ does not embed infinite $R$-cliques.
\end{proof}
\begin{proposition}
	Let $M$ be an unstable simple homogeneous 3-graph in which $T$ defines an equivalence relation with infinitely many infinite classes and $R\sim^SS$. Then $M$ embeds infinite $R$-cliques.
\label{PropUnstabInfCliques}
\end{proposition}
\begin{proof}
	First note that $T$ is a forking relation because it is a nontrivial equivalence relation with infinitely many classes. If both $R$ and $S$ are nonforking, then the Independence Theorem implies that we can embed any finite $R,S$-graph as a transversal to some $T$-classes, so in particular we can embed arbitrarily large finite $R$-cliques and the result follows by compactness.

	Suppose then that one of $R$, $S$ divides. The instability $R\sim^SS$ implies in particular that if $R$ divides then so does $S$, and this is not possible because there must be Morley sequences of vertices (so at least one relation is nonforking). As a consequence, $R$ is a nonforking relation and by the Independence Theorem there are infinite $R$-cliques in $M$.
\end{proof}
In the next Proposition, we use the symbol $R\sim^T_TS$ to say that there exists an indiscernible half-graph witnessing $R\sim S$ such that both monochromatic cliques are of colour $T$. This is stronger than $R\sim^TS$, but weaker than $R\sim^TS\wedge R\not\sim^SS$.
\begin{proposition}
	Let $M$ be an unstable simple homogeneous 3-graph in which $T$ defines an equivalence relation with infinite classes and $R\sim^T_TS$. Then $\aut(M)$ acts 2-transitively on $M/T$ and each pair of distinct $T$-classes is isomorphic to the Random Bipartite Graph.
\label{PropUnstableImprimitiveRBG}
\end{proposition}
\begin{proof}
	Consider an indiscernible half-graph witnessing $R\sim^T_TS$. This half-graph shows that there exist two $T$-classes $C,C'$ such that $R$ and $S$ are realised in the structure induced by $M$ on $C\cup C'$, so by homogeneity all pairs of $T$-classes are in the same $\aut(M)$-orbit and so $\aut(M)$ acts 2-transitively on $M/T$. Consequently, we can find witnesses to $R\sim^T_TS$ in the union of any pair of distinct $T$-classes.

\[
\includegraphics[scale=0.7]{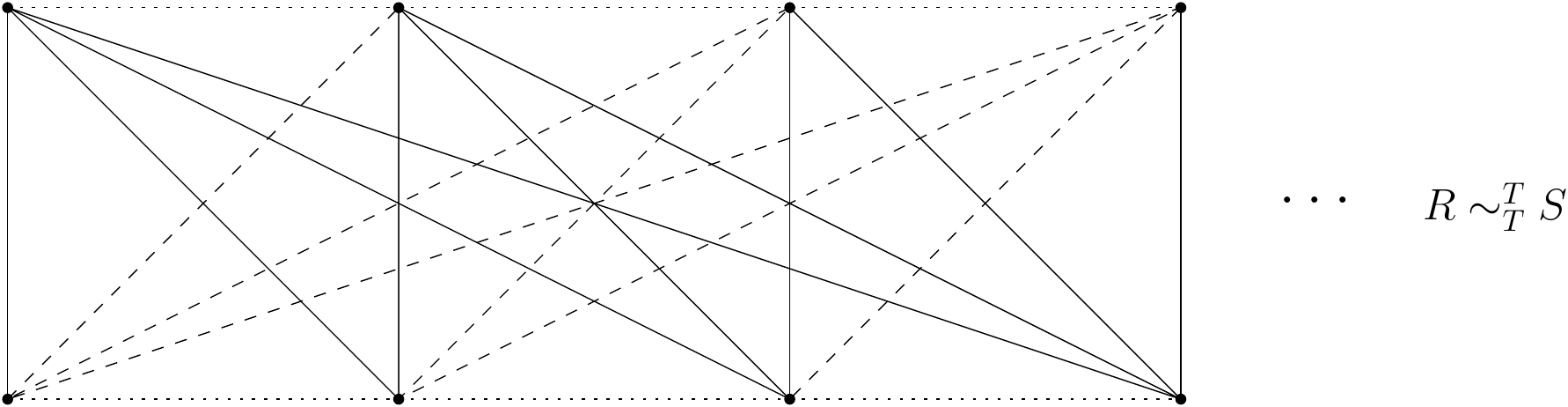}
\]

	Now let $C,C'$ be two distinct $T$-classes, and consider two disjoint finite subsets $X,Y\subset C$ of size $n$ and $m$, respectively. Using our witnesses $a_i,b_i$ ($i\in\omega$) for $R\sim^T_TS$, we know that there exists a vertex $v$ which is $T$-inequivalent to $n+m$ distinct elements (namely, $v=a_n+1$, which is $T$-inequivalent to $b_1,\ldots,b_n$, and $b_{n+1},\ldots,b_{n+m+1}$) and such that $S(v,b_i)$ holds for $i\in\{1,\ldots,n\}$ and $R(v,b_j)$ holds for $j\in\{n+1,\ldots,n+m+1\}$. Therefore, each pair of classes satisfies the extension axioms of the Random Bipartite Graph. 

To see that the union of any two classes $C,C'$ is a homogeneous 3-graph, note that if we take two isomorphic finite sets that meet both classes, then there is by homogeneity of $M$ an automorphism taking $C$ to $C'$, which will fix $C\cup C'$ setwise. And if the finite sets $A,B$ are contained in the same class $C$, then by the argument in the preceding paragraph we can find elements $c_A,c_B$ such that $S(c_A,a)$ for all $a\in A$ and $S(c_B,b)$ for all $b\in B$, and so there is an automorphism $\sigma$ of $M$ fixing $C\cup C'$ such that $\sigma(c_AA)=c_bB$.

Using homogeneity and the fact that all pairs of $T$-classes are in the same $\aut(M)$-orbit, the result follows.
\end{proof}

\begin{proposition}
	Let $M$ be a homogeneous simple unstable 3-graph with $S\sim^ST$ and not embedding infinite $R$-cliques. Then $M$ is imprimitive.
\label{PropImprimitivity2}
\end{proposition}
\begin{proof}
	Suppose for a contradiction that $M$ is primitive. Then $R$ divides, by the Independence Theorem. By $\omega$-categoricity and primitivity, $R$ is not algebraic.
	
	One of $S,T$ is nonforking. If $S$ is nonforking then an indiscernible sequence witnessing $S\sim^ST$ also witnesses that $T$ is nonforking. This is impossible because in this case $R(a)$ is forced to be finite as it does not embed infinite cliques of any colour.

	Assume then that $T$ is nonforking and $S$ divides. First, note that $\diam_S(M)=2$ as we know that $SST\in\age(M)$ by $S\sim^ST$ and so $T(a)\subset S^2(a)$; if $S(a)$ were $R$-free, then it would be isomorphic to the Random Graph by the Lachlan-Woodrow Theorem, and so it would contain infinite $T$-cliques. 

	So $\diam_S(M)=2$ and $S(a)$ is an unstable 3-graph not embedding infinite $R$- or $T$-cliques. By Proposition \ref{PropTwoInfiniteCliques}, $S(a)$ is imprimitive. By $S\sim^ST$, one of $S\vee T$ or $R$ defines an equivalence relation, and it cannot be $S\vee T$ by the Lachlan-Woodrow Theorem (its classes would have to be isomorphic to the Random Graph, and thus embed infinite $T$-cliques), so $T$ defines an equivalence relation with finite classes, which is again impossible as $T$ would be algebraic and we would not be able to find witnesses to $S\sim^ST$ in $S(a)$.
\end{proof}

\section{Homogeneous unstable 3-graphs with two forking relations}\label{SubsecTwoForkingRels}

In this section we prove the non-existence of simple unstable primitive homogeneous 3-graphs with two forking relations. There are two cases: either the nonforking relation is stable, or it is unstable. We treat both cases simultaneously whenever possible.

If all relations are unstable, then none of $R,S,T$ is an algebraic predicate, so we have over any $a$ three infinite orbits of vertices. And if the nonforking relation $R$ is stable, the Independence Theorem and homogeneity guarantee that one can embed infinite $R$-cliques in $M$, so again we get three infinite orbits of vertices over $a$.

\begin{observation}
	Let $M$ be a simple \comm primitive \ent homogeneous 3-graph in which $R$ is the only nonforking relation. Then $R(a)$ is isomorphic to $M$.
\label{IsoToM}
\end{observation}
\begin{proof}
	The $R$-neighbourhood of any vertex $a$, $R(a)$, is homogeneous in the language $\{R,S,T\}$, so it suffices to prove that $\age(M)=\age(R(a))$. Take any finite structure $A$ embedded into $M$, and find a vertex $v$ satisfying $v\indep[\varnothing]A$. Then for each $a\in A$, $R(v,a)$ holds, as the other two relations divide over $\varnothing$.
\end{proof}

\begin{remark}
	In particular, the argument in Observation \ref{IsoToM} implies that the $R$-diameter of any simple primitive homogeneous 3-graph in which $R$ is the only nonforking relation is 2.
\end{remark}

\begin{definition}
	The \emph{diameter triple} of a 3-graph $M$ is $$D(M)=(\diam_R(M), \diam_S(M),\diam_T(M))$$
\end{definition}

This leaves us with four possible diameter triples for a primitive homogeneous simple 3-graph with only one nonforking relation: (2,3,3), (2,3,2), (2,2,3), (2,2,2).

\begin{observation}
	Let $M$ be a simple primitive homogeneous 3-graph in which $R$ is the only nonforking relation. If $S\sim T$, then at least one of the following conclusions holds:
\begin{enumerate}
	\item{The predicates $S$ and $T$ are unstable in $S(a)$ and we can embed infinite $S$-cliques in both $S(a)$ and $T(a)$.}
	\item{The predicates $S$ and $T$ are unstable in $T(a)$ and we can embed infinite $T$-cliques in both $S(a)$ and $T(a)$.}
\end{enumerate}
\label{ObsTwoOutcomes}
\end{observation}
\begin{proof}
	We have either $S\sim^ST$ or $S\sim^TT$ by Proposition \ref{PropCliques}. The conclusion follows.
\end{proof}

\subsection{Simple homogeneous 3-graphs with $D(M)=(2,3,3)$}

\begin{proposition}\label{PropDiam233}
	Let $M$ be a primitive simple homogeneous 3-graph in which all relations are unstable and only one is nonforking, and which satisfies $D(M)=(2,3,3)$. If $S^3(a)=R(a)$ and $R\sim S$, then $T^3(a)=R(a)$.
\end{proposition}
 \begin{proof}
	It follows from $S^3(a)=R(a)$ that $S(a)$ is $R$-free, so $SSR$ is a forbiddden triangle. Suppose for a contradiction $T^3(a)=S(a)$. Then $R(a)=T^2(a)$ and $T(a)$ is $S$-free, so $TTS$ is a forbidden triangle. By Remark \ref{RmkCompatibilityGraph}, we have $R\sim T$ or $S\sim T$.

	We cannot have $S\sim T$, as this implies (by Proposition \ref{PropCliques}) one of $S\sim^ST,S\sim^TT$, and an indiscernible half-graph witnessing either of these compatibilities embeds $TTS$. Therefore, we must have $R\sim^TT$ and $R\not\sim^ST$. This implies that $T(a)$ is an unstable homogeneous $R,T$-graph, isomorphic to the Random Graph by Remark \ref{RmkSimpleGraphs}. But that contradicts Proposition \ref{PropCliques}.
\end{proof}

\begin{observation}\label{ObsSTUnstable}
	Let $M$ be a simple homogeneous 3-graph and suppose that $\age(\Gamma^{S,T})\subset\age(M)$. Then $S$ and $T$ are unstable in $S(a)$ and $T(a)$.
\end{observation}
\begin{proof}
	Any countable $S,T$-graph can be embedded in $S(a)$. In particular, the graph consisting of a half-graph $H$ for $S,T$ and an additional vertex $v$ satisfying $T(v,h)$ for all $h\in H$. This proves the result for $T(a)$; the same argument yields the corresponding statement for $S(a)$.
\end{proof}

\begin{proposition}\label{Prop233SnotsimT}
	Let $M$ be a primitive simple homogeneous 3-graph in which all relations are unstable and only one is nonforking, and which satisfies $D(M)=(2,3,3)$. If $S^3(a)=R(a)$ and $R\sim S$, then $S\not\sim T$.
\end{proposition}
 \begin{proof}
	We know from Proposition \ref{PropDiam233} that under these hypotheses $R(a)=T^3(a)=S^3(a)$, so $T(a)$ and $S(a)$ are $R$-free. If we had $S\sim T$, then $S(a)$ and $T(a)$ are both isomorphic to the Random Graph, by the Lachlan-Woodrow Theorem and Observation \ref{ObsSTUnstable}. Consider any $b\in S(a)$; the set $R(b)$ should be isomorphic to $M$, by homogeneity and Observation \ref{IsoToM}. Notice that, since the $R$-diameter of $M$ is 2, we have $R(b)\cap R(a)\neq\varnothing$. 
\begin{claim}
	$R(b)\cap T(a)\neq\varnothing$.
\end{claim}
\begin{proof}
	We know that $S(a)$ is isomorphic to the Random Graph in predicates $S,T$. Therefore, all triangles with edges in $S,T$ are realised in $M$. From the diameter triple we get that the triangles $SSR,TTR$ are forbidden. From this it follows by primitivity that $RST\in\age(M)$, as otherwise $S\vee T$ would define an equivalence relation on $M$. Therefore, $R(b)\cap T(a)\neq\varnothing$.
\end{proof}

Therefore $R(b)$ consists of a nonempty subset of $T(a)$ and a nonempty subset of $R(a)$. Notice that as the $T$-diameter of $M$ is 3 and $R(a)=T^3(a)$, we do not have $T$-edges from $T(a)$ to $R(a)$, so the $T$-graph on $R(b)$ is disconnected, contradicting the primitivity of $M$ by Observation \ref{IsoToM}.
\end{proof}
\setcounter{case}{0}
\begin{lemma}
There are no primitive simple unstable homogeneous 3-graphs in which only $R$ is nonforking and $D(M)=(2,3,3)$.
\label{LemmmaNo233}
\end{lemma}
\begin{proof}
There are two main cases, depending on whether $R$ is stable or unstable.

\begin{case}Suppose that $R$ is a stable nonforking relation and $S,T$ are unstable and forking. In that case $S\sim T$ is the only edge in the compatibility graph, and we cannot have $S\sim^RT$ by Proposition \ref{PropCliques}, so we must have $S\sim^TT$ or $S\sim^ST$. In either case we obtain $R(a)=S^3(a)=T^3(a)$, so each of $S(a)$ and $T(a)$ is $R$-free and unstable, therefore isomorphic to the Random Graph. Now take $b\in T(a)$ and consider $R(b)$. Since $T(a)$ is $R$-free, we must have $R(b)=(R(b)\cap S(a))\cup (R(b)\cap R(a))$, and each of the intersections is nonempty because of the $R$-diameter 2. The set $R(b)$ should be isomorphic to the primitive structure $M$ by homogeneity and Observation \ref{IsoToM}, but $\diam_S(M)=3$ implies that there are no $S-edges$ from $R(b)\cap S(a)$ to $R(b)\cap R(a)$, so the $S$-graph of $R(b)$ is disconnected and $R(b)$ is imprimitive, contradiction.
\end{case}
\begin{case}If $R$ is unstable, we make the following claims:
\begin{claim} $R\not\sim S$\label{Claim233RsimS}
\end{claim}
\begin{proof}
By Proposition \ref{PropDiam233}, $R(a)=S^3(a)=T^3(a)$, so $S(a)$ and $T(a)$ are $R$-free. We claim that $S(a)$ and $T(a)$ are stable graphs. The reason is that the only simple unstable homogeneous graph is the Random Graph, but if any of $S(a), T(a)$ were isomorphic to the Random Graph, then we would be able to find witnesses to $S\sim T$, which is impossible by Proposition \ref{Prop233SnotsimT}. By the Lachlan-Woodrow Theorem, all the infinite stable graphs realising more than one 2-type are imprimitive, so $\aut(M/a)$ acts imprimitively on $S(a)$ and $T(a)$ and therefore $S$ is an equivalence relation on $S(a)$ and $T$ is an equivalence relation on $T(a)$ (this follows from Proposition \ref{PropMultipartite}). 

Ir follows from the stability of $S(a), T(a)$ that $R\sim^TS$ and $R\sim^ST$, so $S$ has infinitely many classes in $S(a)$, and $T$ has infinitely many classes in $T(a)$.
\[
\includegraphics[scale=0.8]{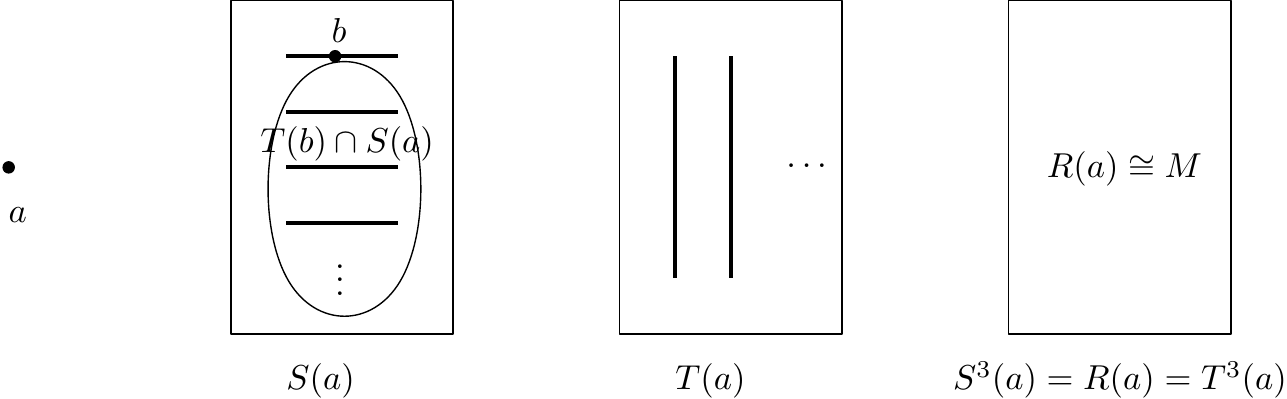}
\]
Then we argue as follows: consider $b\in S(a)$. Its $T$-neighbourhood should be isomorphic to $T(a)$, but $T$ is not an equivalence relation on $T(b)$, since $T(b)\cap T(a)$ consists of infinite $S$-cliques separated by $T$, so we have triangles $TTS$ in $S(b)$.
\end{proof}
\begin{claim}
$S^3(a)\neq R(a)$\label{Claim233S3}
\end{claim}
\begin{proof}
From Claim \ref{Claim233RsimS}, we know that in this case $R\not\sim S$, so by Remark \ref{RmkCompatibilityGraph} we have $R\sim T$ and $S\sim T$. By Observation \ref{ObsTwoOutcomes}, at least one of $S(a),T(a)$ is isomorphic to the Random Graph, so both of them must be isomorphic to the Random graph. Additionally, $T^2(a)=S(a)$, so $T^3(a)=R(a)$. 

	Take any $b\in S(a)$. The $R$-neighbourhood of $b$ consists of a nonempty subset $X$ of $T(a)$ and a nonempty subset $Y$ of $R(a)$. From $T^3(a)=R(a)$ it follows that there are no $T$-edges from $X$ to $Y$, so the $T$ is disconnected in $R(b)$ and therefore $R(b)$ is imprimitive, contradicting the primitivity of $M$ by homogeneity and Observation \ref{IsoToM}.
\end{proof}
\end{case}
	From Claim \ref{Claim233S3}, we get that $S^2(a)=R(a)$, so $S(a)$ is $T$-free. By the Lachlan-Woodrow Theorem, $S(a)$ is a stable graph, because it is simple and does not embed infinite $R$-cliques by Proposition \ref{PropCliques}.

	Our first claim is that we cannot have $R\sim S$ under these hypotheses. The horizontal cliques in any half-graph for $S$ cannot be of type $R$ by Proposition \ref{PropCliques}, and cannot be of type $T$ because $S(a)$ is $T$-free, so they must be of type $S$. But this implies that $S(a)$ is unstable, contradicting the last paragraph.

	Therefore we have $R\sim T$ and $S\sim T$ by Remark \ref{RmkCompatibilityGraph}. In any indiscernible half-graph for $S\sim T$, we cannot have any horizontal cliques of type $R$ by Proposition \ref{PropCliques}, and cannot have horizontal cliques of type $T$ because $S(a)$ is $T$-free. But having a horizontal clique of type $S$ implies that $S$ is unstable in $S(a)$, impossible since there are no infinite $R$-cliques in $S(a)$ (in that case $S(a)$ would be isomorphic to the Random Graph in predicates $R,S$).
\end{proof}

\comm
\begin{proposition}\label{Prop233RsimS}
	There are no primitive simple homogeneous 3-graphs $M$ in which all relations are unstable and only one is nonforking with $D(M)=(2,3,3)$, $S^3(a)=R(a)$ and $R\sim S$.
\end{proposition}
\begin{proof}
	And if $T$ has two classes in $T(a)$ and $S$ has two classes in $S(a)$, then $S(a)$ does not embed infinite $T$-cliques, but $S(b)$ does, for any $b\in T(a)$.
\end{proof}
\begin{proposition}
	There are no primitive simple homogeneous 3-graphs $M$ in which all relations are unstable and only one is nonforking with $D(M)=(2,3,3)$ and $S^3(a)=R(a)$.
\label{Prop233S3}
\end{proposition}
\begin{proof}
\begin{proposition}
	There are no primitive simple homogeneous 3-graphs $M$ in which all relations are unstable and only $R$ is nonforking with $D(M)=(2,3,3)$.
\label{PropNo233}
\end{proposition}
\begin{proof}
\end{proof}

\begin{lemma}\label{Lemming233}
	There are no primitive simple unstable 3-graphs in which only $R$ is nonforking with $\diam(M)=(2,3,3)$.
\end{lemma}
\begin{proof}
	
\end{proof}
\ent

\subsection{Simple homogeneous 3-graphs with $D(M)=(2,3,2)$ or $D(M)=(2,2,3)$}

We can eliminate these two cases simultaneously because of the symmetry of the hypotheses on the predicates $S$ and $T$.

First, the easy case:
\begin{proposition}
	There are no primitive simple unstable 3-graphs in which $R$ is stable and nonforking and $S,T$ are unstable nonforking, with $D(M)=(2,3,2)$ or $D(M)=(2,2,3)$.
\end{proposition}
\begin{proof}
	By the same argument as in Case 1 of Lemma \ref{LemmmaNo233}.
\end{proof}
\begin{observation}\label{Obs232SRfree}
	Let $M$ be a primitive homogeneous simple 3-graph in which all relations are unstable and only $R$ is nonforking. If $D(M)=(2,3,2)$, then $R(a)=S^3(a)$.
\end{observation}
\begin{proof}
	Suppose for a contradiction that $R(a)=S^2(a)$. Then $S(a)$ is $T$-free. We cannot have $S\sim T$ in this case: as usual, we cannot have horizontal $R$-cliques; and since $SST$ is forbidden, we cannot have horizontal cliques of colour $S$ or $T$, either. Remark \ref{RmkCompatibilityGraph} forces then $R\sim S$ and $R\sim T$. 

Let us analyse a set of witnesses for $R\sim S$. The horizontal cliques are forced to be of colour $S$ because $SST$ is forbidden; but this implies that $S(a)$ is an unstable $R,S$-graph, isomorphic to the Random Graph by the Lachlan-Woodrow Theorem. This contradicts Proposition \ref{PropCliques}.
\end{proof}

\begin{observation}
	Suppose that $M$ is a simple unstable homogeneous 3-graph such that $\age(\Gamma^{S,T})\subset\age(M)$. If $M$ embeds $K_n^R$ but not $K_{n+1}^R$, then $M$ is imprimitive and isomorphic to one of $C^R(\Gamma), \Gamma[K_n^R]$, or $K_n^R[\Gamma]$.
\label{ObsUnstableNoRCliques}
\end{observation}
\begin{proof}
	Note first that if $M$ is primitive, then $R$ is a forking relation, because the Independence Theorem guarantees the existence of infinite monochromatic cliques for nonforking relations. From $\age(\Gamma^{S,T})\subset\age(M)$ we get that there are infinite $S$- and $T$-cliques in $M$, and that $S,T$ are unstable and that both $S(a), T(a)$ embed infinite $S$- and $T$-cliques. At least one of $S, T$ is nonforking, so both of them must be nonforking since both embed infinite $S$- and $T$-cliques (and therefore a Morley sequence over $\varnothing$). This is impossible because by primitivity $R$ is non-algebraic, so $R(a)$ contains an infinite monochromatic clique. But $R(a)$ does not contain infinite $S$- or $T$-cliques because $R$ divides, and it clearly does not embed any infinite $R$-cliques. Therefore, $M$ is imprimitive.

	We know that $S\sim^ST$ and $S\sim^TT$ hold, so $S,T,S\vee R, T\vee R$ are not equivalence relations. If $R$ is an equivalence relation, then its classes are finite, isomorphic to $K_n^R$, and $M$ is isomorphic to $C^R(\Gamma)$ or $\Gamma[K_n^R]$ by Corollary \ref{CorFiniteClasses}.

	And if $S\vee T$ is an equivalence relation, then since there are no infinite $R$-cliques in $M$ we know that $S\vee T$ has finitely many classes, each of which is an infinite unstable graph. Therefore, $M$ is isomorphic to $K_n^R[\Gamma]$.
\end{proof}

\begin{proposition}
	Let $M$ be a primitive homogeneous simple 3-graph in which all relations are unstable and only one is nonforking. If $D(M)=(2,3,2)$ and $S\sim^ST$, then $S(a)$ is isomorphic to the Random Graph and $T(a)$ is imprimitive, isomorphic to $C^R(\Gamma)$ or $K_n^R[\Gamma]$ for some $n\geq2$.
\label{PropStructure232}
\end{proposition}
\begin{proof}
	From Observation \ref{Obs232SRfree}, we know that $S(a)$ is $R$-free. From $S\sim^ST$, we get that $S(a)$ is an unstable homogeneous graph, and therefore isomorphic to the Random Graph. It follows from Observation \ref{ObsIsoRG} that $T(a)$ is also unstable, and $\age(\Gamma^{S,T})\subset\age(T(a))$. We know that $R$ is realised in $T(a)$ because the $T$-diameter of $M$ is 2, but $T(a)$ does not embed infinite $R$-cliques. By Observation \ref{ObsUnstableNoRCliques}, $T(a)$ is isomorphic to $C^R(\Gamma)$ or $K_n^R[\Gamma]$, since $SSR$ is a forbidden triangle.

\end{proof}
\begin{proposition}
	There are no primitive homogeneous simple 3-graph in which all relations are unstable and only one is nonforking with $D(M)=(2,3,2)$ and $S\sim T$.
\label{Prop232Graph}
\end{proposition}
\begin{proof}
	From Proposition \ref{PropStructure232}, we know that $S(a)$ is isomorphic to the Random Graph and $T(a)$ is imprimitive, isomorphic to $C(\Gamma)$ or to $K_n^R[\Gamma]$. 

	Consider any $b\in S(a)$. By Observation \ref{Obs232SRfree}, there are no $S$-edges from $S(a)$ to $R(a)$, so $S(b)\setminus\{a\}\subset S(a)\cup T(a)$. Here our proof divides into two cases:

\begin{case}
	If $T(a)\cong K_n^R[\Gamma]$, then $S(b)\setminus\{a\}$ meets only one $S\vee T$-class in $T(a)$. Define an equivalence relation $E$ on $S(a)$ that holds for $b,b'\in S(a)$ if $S(b)$ and $S(b')$ meet the same $S\vee T$-class in $T(a)$. This equivalence relation is invariant under $\aut(M/a)$, but cannot be defined over $a$ without quantifiers, since the structure of $S(a)$ is that of the Random Graph. This contradicts homogeneity.
\end{case}
\begin{case}
	If $T(a)\cong C^R(\Gamma)$, then $S(b)$ contains at most one element from each $R$-class. First note that the indiscernible half-graphs witnessing $S\sim^ST$ imply that $S(b)\cap T(a)$ is an infinite set containing infinite $S$- and $T$-cliques, so it is an infinite and co-infinite subset of $S(b)\cap T(a)$. Similarly, an indiscernible half-graph witnessing $S\sim^TT$ implies that $T(b)\cap T(a)$ is infinite. Let $X$ be the set of vertices in $T(a)$ which are $R$-equivalent to an element of $S(b)\cap T(a)$.

\begin{claim}
	$R(b)\cap T(a)\neq\varnothing$.
\end{claim}
\begin{proof}
	By Remark \ref{RmkCompatibilityGraph}, we must have $R\sim S$ or $R\sim T$. Note that $R\sim^TS$ and $R\sim^ST$ both imply that $R(b)\cap T(a)\neq\varnothing$. Moreover, they imply that $R(b)\cap T(a)$ is infinite.

	It suffices then to prove that we cannot have $R\sim^SS$ or $R\sim^TT$. The first option is impossible as $RSS$ is a forbiden triangle. The second option is impossible as $R$ is stable in $T(a)$.
\end{proof}

\begin{claim}
	We have $X=T(b)\cap T(a)$ or $X=R(b)\cap T(a)$. The set $T(a)\setminus(X\cup S(b))$ is a union of $R$-classes in $T(a)$.
\label{ClaimThreePieces}
\end{claim}
\begin{proof}
	We know that $X\subset T(a)$ is disjoint from $S(b)\cap T(a)$, so it is a subset of $(R(b)\cup T(b))\cap T(a)$.

	Suppose for a contradiction that $X\cap T(b)\neq\varnothing$ and $X\cap R(b)\neq\varnothing$. Take $x\in X\cap R(b)$ and $y\in X\cap T(b)$, and let $x',y'$ be the elements in $S(b)\cap T(a)$ to which $x,y$ are $R$-equivalent in $T(a)$. 
\[
\includegraphics[scale=0.8]{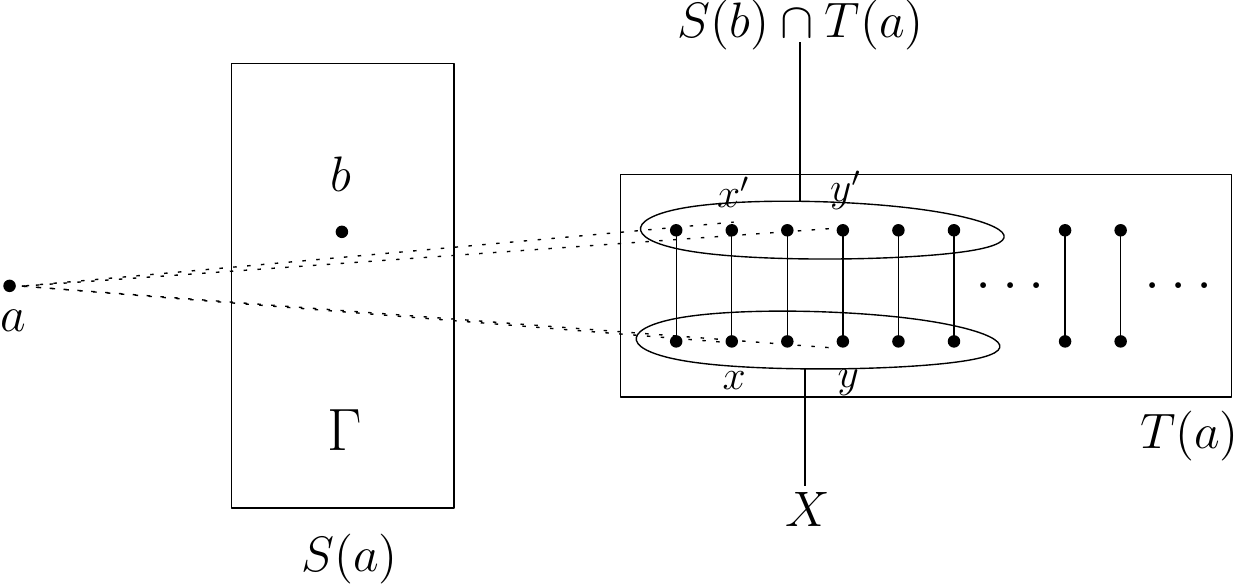}
\]
Then $\qftp(x'/ab)=\qftp(y'/ab)$, but their complete types differ as $x'$ satisfies the formula $\psi(z)=\exists c(R(b,c)\wedge T(a,c)\wedge R(c,z))$, which $y'$ does not satisfy. The second conclusion follows easily from this and $T(a)=(S(b)\cup T(b)\cup R(b))\cap T(a)$.
\end{proof}

Consider now $T(b)$ for the same $b\in S(a)$. Observation \ref{Obs232SRfree} implies that there are no $R$-edges from $S(a)$ to $R(a)$. From Claim \ref{ClaimThreePieces} we get two cases:
\begin{subcase}
	If $T(b)\cap T(a)$ is a union of $R$-classes, then all the elements of $T(b)\cap T(a)$ are already paired in $R$-classes (over $b$), so the elements completing the $R$-classes over $b$ of the $R$-free $T(b)\cap S(a)$ must be in $T(b)\cap R(a)$. The set $T(b)\cap S(a)$ is isomorphic to the Random Graph, so $T(b)\cap (Ra)\cup S(a))$ is a homogeneous imprimitive unstable 3-graph with finite $R$-classes. By Theorem \ref{ThmCGamma}, it should be isomorphic to $C(\Gamma)$, but this is impossible as there are no $S$-edges from $S(b)\cap S(a)$ and $S(b)\cap R(a)$.
\end{subcase}
\begin{subcase}
	If $R(b)\cap T(a)$ is a union of $R$-classes, then we have two $R$-free parts of $T(b)$, namely $T(b)\cap S(a)$ and $T(b)\cap T(a)$; the third part of $T(b)$, $T(b)\cap R(a)$ is not $R$-free. This follows from same argument as in Claim \ref{ClaimThreePieces}: one of $T(b)\cap R(a), T(b)\cap S(a)$, and $T(b)\cap T(a)$ is a union of $R$-classes over $b$, and it can only be $R(a)\cap T(b)$. Therefore the $R$-free parts are paired by $R$. 

There are no $S$-edges from $S(a)$ to $R(a)$, because the $S$-diameter of $M$ is 3 and $R(a)=S^3(a)$, and there are no $R$-edges from $T(b)\cap S(a)$ to $T(b)\cap R(a)$. This implies that $TTR$ embeds in $T(b)$, contradicting $T(b)\cong C(\Gamma)$.
\end{subcase}
\end{case}

This concludes our proof.
\end{proof}
\setcounter{case}{0}
\setcounter{subcase}{0}

For the next result, we need the following definition:
\begin{definition}\label{DefPseudoplane}
	A \emph{pseudoplane} is an incidence structure of points and lines which satisfies the following axioms:
\begin{enumerate}
	\item{There are infinitely many points on each line.}
	\item{There are infinitely many lines through each point.}
	\item{Any two lines intersect in only finitely many points.}
	\item{Any two points lie on only finitely many lines.}
\end{enumerate}
If $M$ is a structure and $\mathcal L$ is a definable family of infinite subsets of $M$, then $P=(M,\mathcal L)$ is a \emph{weak} pseudoplane if the following conditions are satisfied:
\begin{enumerate}
	\item{If $S\neq T\in\mathcal L$, then $|S\cap T|<\omega$.}
	\item{Each $p\in M$ lies in infinitely many elements of $\mathcal L$.}
\end{enumerate}
The weak pseudoplane $(M,\mathcal L)$ is \emph{homogeneous} (binary) if the underlying structure $M$ is homogeneous with respect to its (binary) language.
\end{definition}

The following theorem was proved by Simon Thomas in \cite{thomas1998nonexistence}.

\begin{theorem}
	There is no binary homogeneous weak pseudoplane.
\label{ThmThomas}
\end{theorem}

\comm
\begin{proposition}
	Let $M$ be a primitive homogeneous simple 3-graph in which all relations are unstable and only one is nonforking. If $D(M)=(2,3,2)$ , then $S\not\sim T$.
\label{Prop232Graph}
\end{proposition}
\begin{proof}
	Suppose for a contradiction that we have $S\sim T$. We know from Proposition \ref{Prop232Tcliques} that the horizontal cliques in any indiscernible half-graph witnessing $S\sim T$ are of colour $T$. It follows that $S$ and $T$ are unstable in $T(a)$.
\begin{claim}
	$T(a)$ is primitive.
\end{claim}
\begin{proof}
	By instability of $S, T$, we know that $S,T,R\vee S, R\vee T$ are not equivalence relations on $T(a)$. If $R$ defines an equivalence relation on $T(a)$, then it has finite classes by Proposition \ref{PropCliques}, and is therefore isomorphic to $C(\Gamma)$. But in this case we would be able to find witnesses to $S\sim T$ with horizontal $S$-cliques, contradicting Proposition \ref{Prop232Tcliques}. If $S\vee T$ defines an equivalence relation, then its classes are infinite and isomorphic to the Random Graph, and again we contradict Proposition \ref{Prop232Tcliques}. 
\end{proof}
From this claim and the information we get from the half-graphs, we know that $T(a)$ is a 3-graph in which $S$ and $T$ are nonforking relations (meaning: the formulas $T(x,a)\wedge T(a,b)\wedge S(x,b)$ and $T(x,a)\wedge T(a,b)\wedge T(x,b)$ do not fork over $a$), so we can use the Independence Theorem to show that any finite $S,T$-structure can be embedded into $T(a)$; therefore, we can find an indiscernible half-graph for $S\sim T$ with horizontal $S$-cliques, contradicting Proposition \ref{Prop232Tcliques}.
\end{proof}

\ent
\begin{lemma}
		There are no primitive simple homogeneous 3-graphs $M$ in which all relations are unstable and only one is nonforking with $D(M)=(2,3,2)$.
\label{LemmaNo232}
\end{lemma}
\begin{proof}
	Suppose for a contradiction that $M$ is a primitive simple homogeneous 3-graph with $D(M)=(2,3,2)$. We know from Proposition \ref{Prop232Graph} that $S\not\sim T$. 
\comm
\begin{claim}
	It is not the case that $R\sim^TT$.
\label{Claim2}
\end{claim}
\begin{proof}
	Suppose for a contradiction that we have witnesses for $R\sim^TT$. Our first claim is that $T(a)$ is primitive.

From $R\sim^TT$ we get that $R$ and $T$ are unstable in $T(a)$, so $R,T,R\vee S, T\vee S$ do not form equivalence relations on $T(a)$. Also, $R\vee T$ is not an equivalence in $T(a)$: if it were, then it would necessarily have infinite classes as both $R$ and $T$ are unstable and therefore non-algebraic; each class would be a homogeneous unstable graph, isomorphic to the Random Graph by the Lachlan-Woodrow Theorem. 

To prove that $S$ does not define an equivalence relation on $T(a)$, note first that if it did define an equivalence relation, then it would necessarily have infinitely many infinite classes. This is impossible as Theorem \ref{ThmCGamma} and Corollary \ref{CorFiniteClasses} would imply that $T(a)$ embeds infinite $R$-cliques, contradicting Proposition \ref{PropCliques}. The infinite $T$-cliques in $T(a)$ imply that if $S$ defines an equivalence relation on $T(a)$, then it has infinitely many classes. Therefore, $M$ interprets an infinite 3-graph in which $R,T$ are unstable, $S$ defines an equivalence relation with infinitely many infinite classes, and $R$ does not form infinite cliques. This is impossible as no such 3-graph is simple by Proposition \ref{PropImprimitiveInfClasses}. 

It follows then from $R\sim^TT$ that $M$ interprets (as $T(a)$) a primitive homogeneous 3-graph in which $R$ and $T$ are unstable and divide, $R\sim^TT$, $R$ does not form infinite cliques, and $S$ is nonforking. By Proposition \ref{PropNoSnotTnoCliquesR} (exchanging the roles of $S,T$), this is impossible as $M$ is simple.
\end{proof}
\ent

From $S\not\sim T$ it follows, in particular, that we cannot find half-graphs for $S$ within $S(a)$, so $S$ is stable in $S(a)$ (as is $T$), and $S(a)$ is a stable homogeneous graph realising $S$ and $T$ (indiscernible witnesses to $R\sim^ST$ imply that are infinite $S$-cliques in $M$, and therefore in $S(a)$; similarly, $R\sim^TS$ implies that $S(a)$ embeds infinite $T$-cliques). By the Lachlan-Woodrow Theorem, $S(a)$ is imprimitive. From the fact that $T$ forms infinite cliques in $S(a)$ and Proposition \ref{PropMultipartite}, we derive that $S$ is an equivalence relation on $S(a)$ with infinitely many infinite classes. 

Given two vertices $c,c'\in M$ with $S(c,c')$, define the \emph{line} $\ell(c,c')$ as $c\cup c'/S^c$, where $c'/S^c$ is the imprimitivity block of $c'$ in $S(c)$. It is clear that $\ell(c,c')$ is the maximal $S$-clique in $M$ containing $c$ and $c'$. Let $\mathcal L$ be the set $\{\ell(c,c'):c,c'\in M, S(c,c')\}$. 

\begin{claim}
$(M,\mathcal L)$ is a weak pseudoplane (see Definition \ref{DefPseudoplane}).
\end{claim}
\begin{proof}
We need to verify two conditions:
\begin{enumerate}
\item{Let $\ell,\ell'$ be distinct elements of $\mathcal L$. Then $|\ell\cap\ell'|\leq1$, because the lines are defined as maximal cliques.}
\item{The second condition in the definition of a weak pseudoplane (see Definition \ref{DefPseudoplane}) follows trivially from the fact (proved above) that $S(a)$ contains infinitely many infinite $S$-classes.}
\end{enumerate}
\end{proof}
We have reached a contradiction by Theorem \ref{ThmThomas}.
\end{proof}

The same argument, substituting $T$ for $S$, proves that there are no primitive homogeneous simple 3-graphs in which all three predicates are unstable and only one is nonforking with $D(M)=(2,2,3)$. Therefore, the only possibility for such a 3-graph is to have $D(M)=(2,2,2)$.

\subsection{Simple homogeneous 3-graphs with $D(M)=(2,2,2)$.}
This subsection deals with the most delicate cases in this chapter. The strategy is to prove $S\not\sim T$ first (Proposition \ref{PropOneIsStable} to Lemma \ref{lemmaST}). This reduces the compatibility graph to $R\sim^ST, R\sim^TS$, and the only possibility is for both $S(a)$ and $T(a)$ to be of the form $K_m^i[K_n^j[K_o^k]]$, with only the subindex corresponding to $R$ is finite. Most of the cases can be dealt with directly from easy results from Chapter \ref{ChapGenRes}.

\begin{proposition}
	Let $M$ be a simple unstable 3-graph such that $M$ does not embed infinite $R$-cliques. Then one of $R,S,T$ is stable. Moreover, the stable relation is either an equivalence relation or the complement of an equivalence relation.
\label{PropOneIsStable}
\end{proposition}
\begin{proof}
	Suppose for a contradiction that all predicates are unstable; we will prove that any such $M$ would be imprimitive, so one of $R,S,T$ is an equivalence relation or the complement of an equivalence relation and is therefore stable. 

	If $\aut(M)$ acts primitively on $M$, then $R(x,a)$ divides over $\varnothing$, as otherwise the Independence Theorem and primitivity guarantee that $M$ embeds infinite $R$-cliques. By Lemma \ref{PropNoSimpleOneDividing}, one of $S$, $T$ divides. Let us say without loss of generality that $S$ divides and $T$ is nonforking. By Lemmas \ref{LemmmaNo233} and \ref{LemmaNo232}, we have $D(M)=(2,2,2)$, so for any $a\in M$ each of $R(a),S(a),T(a)$ is a homogeneous simple 3-graph (\ie all relations in the language are realised in each of these sets). By Proposition \ref{PropTwoInfiniteCliques}, both $S$ and $T$ form infinite cliques in $M$. By Proposition \ref{PropCliques}, $S(a)$ and $R(a)$ do not embed infinite $T$-cliques.
\begin{claim}
	$S\not\sim T$.
\label{ClaimSNotSimT}
\end{claim}
\begin{proof}
	Suppose for a contradiction that $S\sim T$ holds. Then we must have $S\sim^ST$ and $S\not\sim^TT$ because there are no infinite $T$-cliques in $S(a)$. From this it follows that $S$ and $T$ are unstable in $S(a)$, so $S(a)$ is a homogeneous simple unstable 3-graph not embedding infinite $R$- or $T$-cliques. By Proposition \ref{PropTwoInfiniteCliques}, $S(a)$ is imprimitive. 

By the instability of $S$, $T$ in $S(a)$, we know that $S$, $T$, $R\vee S$, $R\vee T$ do not define equivalence relations on $S(a)$, so this leaves us with two options for the equivalence relation on $S(a)$: it could be defined by $R$ or by $S\vee T$. The latter case is impossible because each class would be an infinite homogeneous simple unstable $ST$-graph, isomorphic to the Random Graph by the Lachlan-Woodrow Theorem, contradicting the fact that $S(a)$ does not embed infinite $T$-cliques. 

Similarly, if $R$ defines an equivalence relation on $S(a)$ then either $\aut(M/a)$ acts 2-transitively on $S(a)/R$, or it acts transitively, but not 2-transitively on $S(a)/R$. In the former case $S(a)\cong C(\Gamma^{S,T})$ and we can find infinite $T$-cliques in $S(a)$, and in the latter $M$ interprets a Henson graph (cf. the proof of Observation \ref{ObsInterpretedGraph}), contradicting simplicity.

We conclude $S\not\sim T$.
\end{proof}

It follows from Claim \ref{ClaimSNotSimT} and Remark \ref{RmkCompatibilityGraph} that $R\sim S$ and $R\sim T$ hold. By the same argument as before, we have $R\sim^SS$ and $R\not\sim^TS$, $R\sim^ST$ and $R\not\sim^TT$. From $R\sim^SS$ we get that $R,S$ are unstable in $S(a)$, so $S(a)$ is a simple unstable 3-graph not embedding infinite $R$- or $T$-cliques. By Proposition \ref{PropTwoInfiniteCliques}, $S(a)$ is imprimitive and one of $T,R\vee S$ defines an equivalence relation on $S(a)$. As before, $R\vee S$ cannot define an equivalence relation because its classes would be isomorphic to the Random Graph, contradicting that $R$ does not form infinite cliques. Therefore, $T$ defines an equivalence relation with finite classes on $S(a)$ and the same argument from the proof of Claim \ref{ClaimSNotSimT} proves the impossibility of this.

We have reached a contradiction. We conclude that $M$ is imprimitive, so one of $R,S,T$ is either an equivalence relation or its complement. In any case, one of the relations is stable.
\end{proof}
\begin{proposition}\label{PropStrST}
	Let $M$ be a primitive homogeneous simple 3-graph with $D(M)=(2,2,2)$ in which $S\sim T$ holds, and $S,T$ are forking relations. Then for any $a$ the sets $S(a)$, $T(a)$ are imprimitive and each is isomorphic to one of $C(\Gamma^{ST}), \Gamma^{ST}[K_n^R]$, or $K_n^R[\Gamma^{ST}]$. In particular $\age(\Gamma^{ST})\subset\age(M)$.
\end{proposition}
\begin{proof}
	If $S\sim T$ holds, then, as $R$ is nonforking, we have that at least one of $S\sim^ST$ and $S\sim^TT$ holds. Suppose that $S\sim^ST$ holds, so $S,T$ are unstable in $S(a)$. Then $S(a)$ is an unstable 3-graph not embedding infinite $R$-cliques and with witnesses for $S\sim^ST$, so it is imprimitive by Proposition \ref{PropImprimitivity2}. If $R$ defines an equivalence relation on $S(a)$, then $S(a)$ is isomorphic to $C(\Gamma^{ST})$ or to  $\Gamma^{ST}[K_n^R]$. And if $S\vee T$ is an equivalence relation, then $S(a)\cong K_n^R[\Gamma^{ST}]$. 

Notice that any of these conclusions implies $\age(\Gamma^{ST})\subset\age(M)$, so $S\sim^TT$ also holds and we can carry out the same argument for $T(a)$.
\end{proof}

\begin{proposition}
	Let $M$ be a primitive simple 3-graph in which $R,S,T$ are unstable, $R(x,a)$ is nonforking over $\varnothing$, $S,T$ are forking, and $D(M)=(2,2,2)$. Then $R\not\sim^SS$ and $R\not\sim^TT$.
\label{PropRSRT}
\end{proposition}
\begin{proof}
	We will prove only $R\not\sim^SS$. Suppose for a contradiction that $R\sim^SS$ holds, so $R$ and $S$ are compatible in $S(a)$. Then $S(a)$ is imprimitive by Proposition \ref{PropImprimitivity1}, as it does not embed infinite $R$-cliques, so one of $R\vee S$ or $T$ is an equivalence relation on $S(a)$. The relation $R\vee S$ is not an equivalence relation on $S(a)$, as its classes would be isomorphic to the Random Graph. Therefore $T$ is an equivalence relation on $S(a)$.

	The $T$-classes cannot be finite by Theorem \ref{ThmCGamma} and Corollary \ref{CorFiniteClasses}. And if $\aut(M/a)$ does not act 2-transitively on $S(a)/T$, then the same argument as in Claim \ref{ClaimInterpretsHenson} proves that $M$ interprets a Henson graph, contradicting the simplicity of $M$. Since $S$ forms infinite cliques, $T$ must have infinitely many classes. But this contradicts the fact that $S(a)$ does not embed infinite $R$-cliques, by Proposition \ref{PropUnstabInfCliques}.

	The proof for $R\not\sim^TT$ is similar.
\end{proof}

Proposition \ref{PropStrST} leaves us with six cases to analyse under $S\sim T$ (nine in principle, but we can eliminate three of them by the symmetry of the hypotheses on $S$ and $T$), listed in table \ref{TableCases}.

\begin{table}[!h]
\centering
\caption{Possible structures under $S\sim T$}
\begin{tabular}{ccc} 
\hline\hline 
Case & $S(a)$ & $T(a)$ \\ [0.5ex]
\hline 
I & $C(\Gamma^{ST})$ & $\Gamma^{ST}[K_n^R]$ \\
II & $C(\Gamma^{ST})$ & $K_n^R[\Gamma^{ST}]$ \\
III & $\Gamma^{ST}[K_n^R]$ & $K_m^R[\Gamma^{ST}]$ \\
IV & $C(\Gamma^{ST})$ & $C(\Gamma^{ST})$ \\
V & $\Gamma^{ST}[K_n^R]$ & $\Gamma^{ST}[K_m^R]$ \\ 
VI&$K_n^R[\Gamma^{ST}]$&$K_m^R[\Gamma^{ST}]$\\[1ex]
\hline 
\label{TableCases} 
\end{tabular}
\end{table}

We can eliminate Case I easily by considering any $c\in T(a)$, and noticing that its $S$-neighbourhood contains a copy of $K_{2,2}$ in which the two sides of the partition are $R$-edges and the edges are of colour $S$ (or $T$). This is contradicts homogeneity because $C(\Gamma)$ does not embed that graph. Case III can be eliminated in a similar manner: for any $b\in S(a)$ the set $T(b)$ contains complete bipartite graphs of the same kind (\ie two dosjoint $R$-cliques with all other edges of colour $S$), and this does not embed in $T(a)$ (these arguments apply whether $R$ is stable or not). Case II is more complicated.

\begin{proposition}
	There are no primitive simple homogeneous 3-graphs $M$ in which  $R$ is the only nonforking relation, $S\sim T$, $R\sim S$ (if $R$ is unstable), $S(a)\cong C(\Gamma^{ST})$ and $T(a)\cong K_n^R[\Gamma^{ST}]$.
\label{PropCaseII}
\end{proposition}
\begin{proof}
	Suppose for a contradiction that $M$ satisfies all the conditions in the statement. We will prove that there is a formula with the TP2. 

	Consider any $c\in T(a)$. Then $S(c)$ consists of an $R$-free subset of $T(a)$, isomorphic to $\Gamma^{ST}$, and two non-empty subsets $X\subset S(a), Y\subset R(a)$. The sets $S(c)\cap T(a)$ and $S(c)\cap S(a)$ are nonempty because $D(M)=(2,2,2)$, so the triangles $SST, TTS$ are in $\age(M)$. Note that $S(a)\cong C(\Gamma^{ST})$ implies that the triangle $RST$ embeds into $M$, so $S(c)\cap R(a)\neq\varnothing$.

	By homogeneity, $S(c)\cong S(a)$, so one of $X$, $Y$ is a union of $R$-classes in $S(c)$, while the other is $R$-free (homogeneity excludes the possibility of one of $X,Y$ containing both a full $R$-class and an unpaired element). We eliminate these cases in the following claims.
\begin{claim}\label{Claimtp21}
	It is not the case that $S(c)\cap R(a)$ is a union of $R$-classes over $c$ if $R(a)\cap S(c)$ is infinite.
\end{claim}
\begin{proof}
	Suppose for a contradiction that $S(c)\cap R(a)$ is a union of $R$-classes over $c$. The set $S(c)\cap T(a)$ is infinite and isomorphic to the Random Graph, and is in definable bijection (via $R$) with $S(c)\cap S(a)$. 

	Since $S(c)\cap R(a)$ is a union of $R$-classes and $S(c)\cap S(a)$ is $R$-free, then the following structure on four vertices is a minimal forbidden configuration.
\[
\includegraphics[scale=0.7]{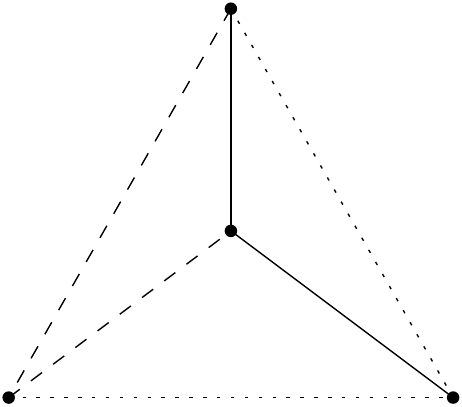}
\]

Clearly, $\age(\Gamma^{S,T})\subset\age(M)$, so in particular we can find witnesses $(a_ib_i)_{i\in\omega}$ to $T\sim^T_TS$. The sequence $\Phi=(a_ib_i)_{i\in\omega}$ also witnesses that the formula $R(x,a)\wedge S(x,b)$ 2-divides over $\varnothing$.

	Since $S(c)\cap R(a)$ is a union of $R$-classes over $c$, we know by the structure of $S(a)$ and homogeneity that there are no $R$-edges from $S(c)\cap R(a)$ to $S(c)\cap S(a)$, and there are edges of colours $S$ and $T$ between $S(c)\cap R(a)$ and $S(c)\cap S(a)$ (this follows from the fact that the triangles $SSR$, $TTR$ are forbidden in $C(\Gamma)$. 

\[
\includegraphics[scale=0.7]{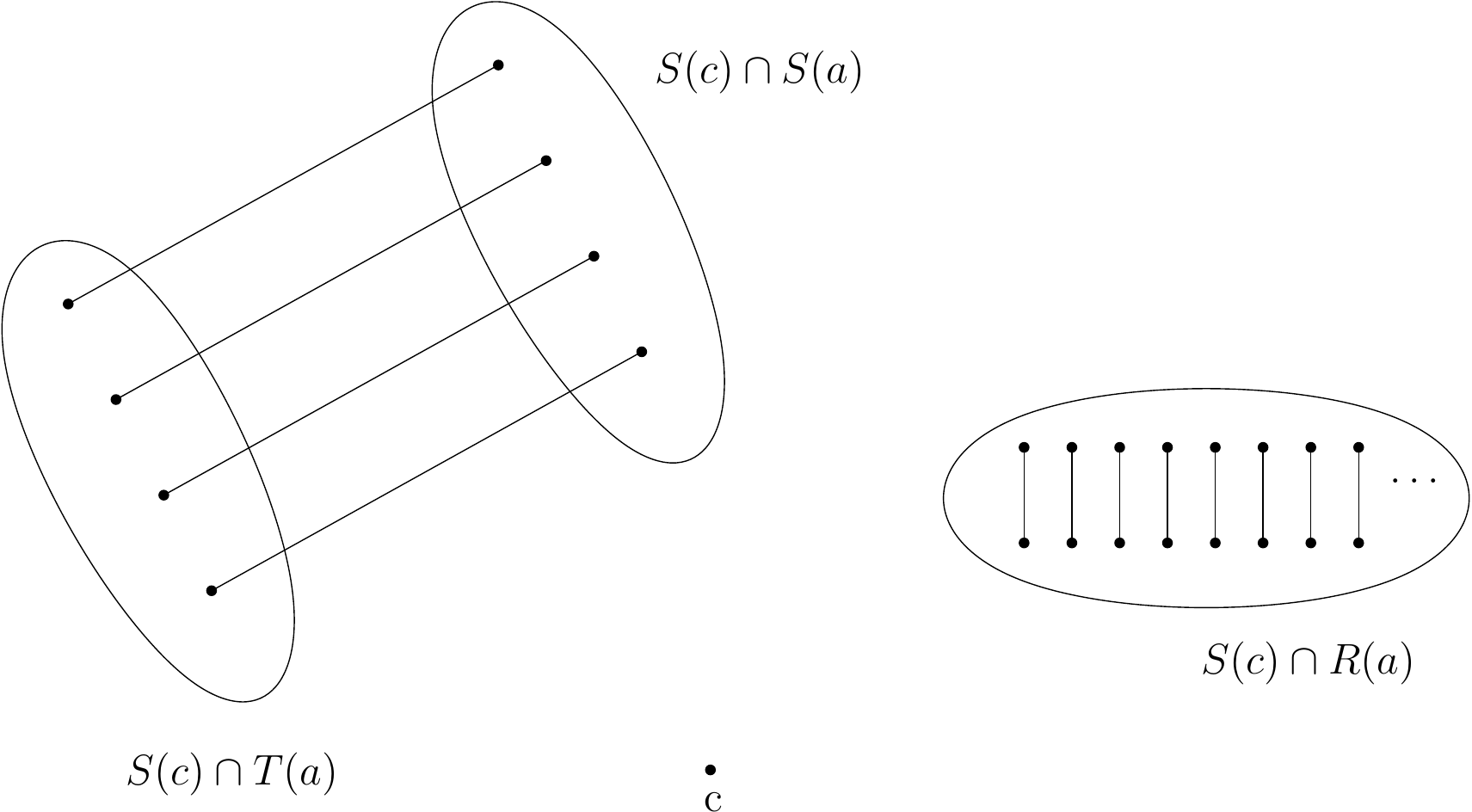}
\]

(Only $R$-edges are shown in the diagram.)

	Given an element $p\in S(c)\cap(R(a)\cup S(a))$, its orbit under $\aut(M/ac)$ is either $S(c)\cap R(a)$ or $S(c)\cap S(a)$. Therefore, $\aut(M/ac)$ acts without finite orbits on $p\in S(c)\cap(R(a)\cup S(a))$ and we can find infinitely many distinct $S$-edges in $S(c)\cap(R(a)\cup S(a))$. Take any infinite sequence $\Sigma$ of $T$-edges in that set. Then $\Sigma$ is spans an $R$-free structure.

	Now we can use the fact that $\age(\Gamma^{ST})\subset\age(M)$ to find an array of parameters $a_i^j,b_k^l$ ($i,j,k,l\in\omega$) such that $\{a_i^j b_k^j:i,k\in\omega\}$ is isomorphic to $\Phi$, and for any $f:\omega\rightarrow\omega$ the set $\{a_{f(i)}^i b_{f(i)}^i:i\in\omega\}$ is isomorphic to $\Sigma$. We conclude that $R(x,a)\wedge S(x,b)$ has the TP2.
\end{proof}

Note that if $R$ is unstable, then since $R\sim^TS$ (by Proposition \ref{PropRSRT}), the set $S(c)\cap R(a)$ is infinite, so Claim \ref{Claimtp21} proves in particular that $S(c)\cap R(a)$ is not a union of $R$-classes in $S(c)$ if $R$ is unstable. And if $R$ is stable and $S(c)\cap R(a)$ is a finite union of $R$-classes in $S(c)$, then we have only the instability $S\sim T$, and can find indiscernible isomorphic half-graphs $X,Y$ witnessing it in $R(a)$ and $S(c)$. There is, by homogeneity, a sequence $(\sigma_i:i\in\omega)$ in $\aut(M)$ that takes increasingly large intial segments of $Y$ to $X$, so by closedness of $\aut(M)$ there is some $\sigma\in\aut(M)$ taking $Y$ to $X$. Then $Y\subset T(\sigma(c))$ and $Y\subset R(\sigma(a))$, so $T(c)\cap R(a)$ is infinite by homogeneity, and Claim \ref{Claimtp21} is also valid when $R$ is stable.

\begin{claim}
	It is not the case that $S(c)\cap S(a)$ is a union of $R$-classes over $c$.
\end{claim}
\begin{proof}
	The proof is similar to that of Claim \ref{Claimtp21}, but in this case the minimal forbidden configuration is 
\[
\includegraphics[scale=0.7]{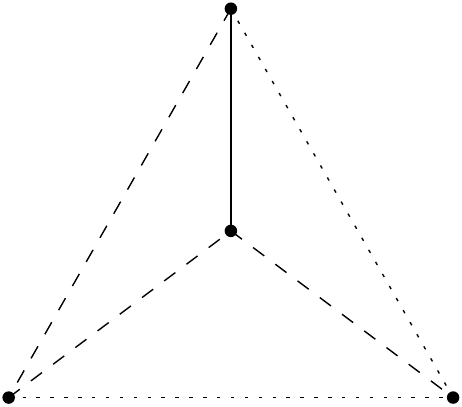}
\]

So there are no $R$-edges from $S(c)\cap T(a)$ to $S(c)\cap S(a)$ and this last set is a union of $R$-classes over $c$. Again, $S(c)\cap S(a)$ is infinite as a consecuence of $R\sim^T S$ if $R$ is unstable, and of $S\sim^ST$ if $R$ is stable, and $R(x,a)\wedge S(x,b)$ 2-divides as witnessed by a sequence of $T$-edges witnessing $S\sim^TT$ in which both monochromatic cliques are of colour $T$. The same argument as in Claim \ref{Claimtp21} proves that we can find an infinite sequence $R$-free sequence of $T$-edges $(a_ib_i)_i\in\omega$ such that $\{R(x,a_i)\wedge S(x,b_i):i\in\omega\}$ is consistent. Now using $\age(\Gamma^{ST})\subset\age(M)$, we can prove the TP2 for $R(x,a)\wedge S(x,b)$.
\end{proof}

We have reached a contradiction as homogeneity implies that one of $S(c)\cap S(a)$ or $S(c)\cap R(a)$ is a union of $R$-classes over $c$.
\end{proof}

\begin{remark}
	The same argument, with the appropriate modifications, shows that there are no primitive homogeneous simple 3-graphs $M$ in which $R,S,T$ are unstable $S\sim T$, $R\sim T$, $T(a)\cong C(\Gamma^{ST})$, and $S(a)\cong K_n^R[\Gamma^{ST}]$. In all cases we find the forbidden structures from the claims and can complete the same arguments.
\label{RmkOtherCases}
\end{remark}
Remark \ref{RmkOtherCases} implies the following:
\begin{proposition}
 There are no primitive homogeneous simple unstable 3-graphs in which only $R$ is nonforking, with $S\sim T$, such that $S(a)\cong C(\Gamma^{ST})$ and $T(a)\cong K_n^R[\Gamma^{ST}]$.
\end{proposition}
Note that the arguments in Claim \ref{Claimtp21} depend only on $S(c)\cap T(a)$ being $R$-free, isomorphic to the Random Graph, and on the structure of $S(a)$. This means that we can apply the same methods to eliminate Case IV.
\begin{proposition}
	There are no primitive simple homogeneous 3-graphs $M$ in which $R,S,T$ are unstable, $S\sim T$, $S(a)\cong C(\Gamma^{ST})$ and $T(a)\cong C(\Gamma^{ST})$.
\end{proposition}
\begin{proof}
	By the same arguments as in Proposition \ref{PropCaseII}.
\end{proof}
\begin{proposition}
	There are no primitive simple unstable homogeneous 3-graphs $M$ in which only $R$ is nonforking, $S\sim T$, $S(a)\cong\Gamma^{ST}[K_n^R]$ and $T(a)\cong\Gamma^{ST}[K_m^R]$.
\end{proposition}
\begin{proof}
	Suppose for a contradiction that $M$ is a homogeneous 3-graph satisfying all the conditions in the statement. 

	First note that $n=m$. If $n<m$ then for any $c\in T(a)$ there are $R$-cliques of size $m$ in $S(c)\cap T(a)$, contradicting homogeneity. The same argument proves that $m$ cannot be smaller than $n$.

	Now consider $b\in S(a)$ and $S(b)$. The structure of $S(a)$ implies that $S(b)\cap S(a)$ is a union of infinitely many $R$-classes in $S(a)$, and in fact isomorphic to $S(a)$. Note that $S(b)\cap R(a)$ cannot be a union of $R$-classes over $b$, since a full $R$-class over $b$ in $R(a)$ would mean that $S(b)$ embeds $K_{n+1}^R$, impossible by homogeneity. From this it follows that $S(b)\cap T(a)$ is not a union of $R$-classes over $b$.

	But now note that $S(b)\cap T(a)$ should embed $K_n^R$, since the structure consisting of two vertices $v,w$ joined by an $S$-edge and $c_1,\ldots,c_n$ forming an $R$-clique, with $T(w,c_i)$ and $S(v,c_i)$ is in $\age(M)$ because it can be embedded in $S(a)$ or $T(a)$. We have reached a contradiction.

\comm
	From this observation and the structure of $T(a)$, it follows that $S(c)\cap T(a)\cong S(a)$, and in particular is an infinite union of $R$-classes over $c$. 

The set $S(c)\cap S(a)$ is also infinite because $\age(\Gamma^{ST})\subset\age(M)$, so in particular we can find any finite $ST$-graph $G$ such that there are two special vertices $g_1,g_2$ satisfying $T(g_1,g_2)$ and $S(x,g_1), S(x,g_2)$ for all $x\neq g_1,g_2$ in $G$. But $S(c)\neq S(a)$ because otherwise we would be able to define an equivalence relation $Q$ on $M$. Clearly, $S(c)\cap S(a)$ cannot be $R$-free, since $S(c)$ is isomorphic to $S(a)$ and $S(c)\cap T(a)$ is a union of classes.

Moreover, $S(c)\cap S(a)$ is a union of $R$-classes, as it is a proper subset of $S(a)$ and if it were not a union of $R$-classes then we could find elements $u,v\in S(a)\setminus S(c)$ such that there exists $u'\in S(c)\cap S(a)$ with $R(u,u')$ and the $R$-class of $v$ does not meet $S(c)\cap S(a)$. In this case we have $\qftp(u/ac)=\qftp(v/ac)$ but $\tp(u/ac)\neq\tp(v/ac)$, contradicting homogeneity. From this it follows that each of $S(c)\cap S(a), S(c)\cap T(a)$, and $S(c)\cap R(a)$ is a union of $R$-classes. 

Finally, $S(c)\cap R(a)$ is also infinite because $R\sim^TS$ witnesses the existence of an infinite $T$-clique in $S(c)\cap R(a)$. Therefore there are no $R$-edges between any two of $S(c)\cap S(a)$, $S(c)\cap T(a)$ and $S(c)\cap R(a)$. 
This means that the following six configurations are forbidden in $M$:
\[
\includegraphics[scale=0.7]{StableForking8.pdf}
\]

	The hypotheses already give us that the triangle $RST$ is forbidden in $S(a)$ and $T(a)$, as $\aut(M/a)$ does not act 2-transitively on $S(a)/R$. Therefore, the following two are also forbidden:
\[
\includegraphics[scale=0.7]{StableForking9.pdf}
\]

Clearly, the structure
\[
\includegraphics[scale=0.7]{StableForking10.pdf}
\]
is in $\age(M)$, as the triangle $SST$ can be embedded in $T(a)$. Analysing the structures $C_1$ and $\alpha$, we conclude that the amalgamation problem 
\[
\includegraphics[scale=0.7]{StableForking11.pdf}
\]
has $R$ as a forced, \ie unique, solution ($T(b,c)$ gives $C_1$, and $S(b,c)$ gives $\alpha$). Similarly from $C_4$ and $\beta$ we get that 
\[
\includegraphics[scale=0.7]{StableForking12.pdf}
\]
has $R$ as a forced solution; call that structure $D$. From $C_1$ and $C_3$ we find that the problem
\[
\includegraphics[scale=0.7]{StableForking13.pdf}
\]
has $T$ as a forced solution. It follows that the Amalgamation Property fails as the problem
\[
\includegraphics[scale=0.7]{StableForking14.pdf}
\]
has no solution in $\age(M)$: by our preceding arguments, the structures on $a_0,a_1,a_2,b$ and $a_0,a_1,a_2,c$ are both in $\age(M)$, the former being $B$ and the latter $D$, but $b,c,a_0,a_1$ has forced solution $R(b,c)$ and $b,c,a_1,a_2$ has forced solution $T(b,c)$. Homogeneity fails.
\ent
\end{proof}
\begin{proposition}
	There are no primitive simple homogeneous 3-graphs $M$ in which $S,T$ are unstable, $R$ is the only nonforking relation, $S\sim T$, $S(a)\cong K_n^R[\Gamma^{ST}]$ and $T(a)\cong K_m^R[\Gamma^{ST}]$.
\label{PropLastCaseSsimT}
\end{proposition}
\begin{proof}
	Suppose for a contradiction that $M$ is a primitive simple homogeneous 3-graph as in the statement. 

Given any $a\in M$, let $b\in S(a)$ and $c\in T(b)\cap S(a)$. By homogeneity, $T(b)\cong T(a)$.

	The set $T(b)\cap S(a)$ is contained in one $S\vee T$-class over $b$, by the structure of $S(a)$. We use the notation $x/P^y$, where $x,y$ are vertices of $M$ and $P$ is a formula defining an equivalence relation in $T(y)$, to denote the $P$-class of $x$ in $T(y)$.

\begin{claim}
	If $c/(S\vee T)^b\cap T(a)\neq\varnothing$, then for any $d\in T(b)$ such that $R(c,d)$ holds we have $d/(S\vee T)^b\subset R(a)$.
\label{ClaimClasses1}\end{claim}
\begin{proof}
	If the claim were false, we would be able to $d\in c/(S\vee T)^b$ and $d'\in T(b)\cap T(a)$ such that $R(d,d')$. These two elements have the same quantifier-free type over $ab$, but there is no $\sigma\in\aut(M/ab)$ taking $d$ to $d'$, since $d$ satisfies the existential formula $\varphi(y)=\exists x(T(x,b)\wedge S\vee T(y,x)\wedge T(b,y))$ (where $c$ is the quantified $x$), but $d'$ does not satisfy such a formula (see the illustration below, assuming $T(c,d)$).
\[
\includegraphics[scale=0.7]{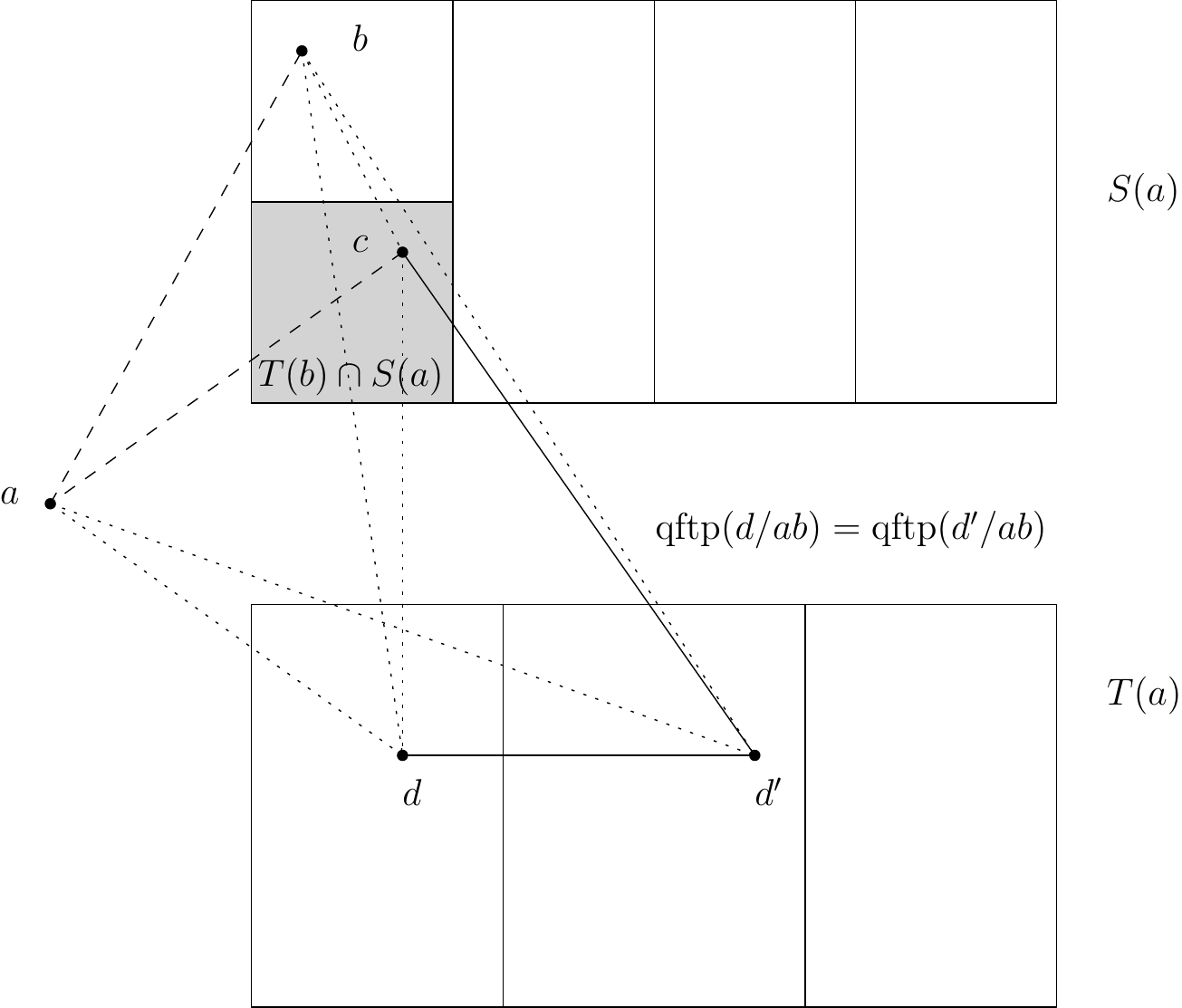}
\]
\end{proof}
By the same argument, 
\begin{claim}
	If $c/(S\vee T)^b\cap R(a)\neq\varnothing$, then for any $d\in T(b)$ such that $R(c,d)$ holds we have $d/(S\vee T)^b\subset T(a)$.
\label{ClaimClasses2}\hfill$\Box$
\end{claim}
From Claims \ref{ClaimClasses1} and \ref{ClaimClasses2}, it follows that $c/(S\vee T)^b\cap T(a)=\varnothing$ or $c/(S\vee T)^b\cap R(a)or\varnothing$, as otherwise $T(b)$ would consist of only one $S\vee T$-class, contradicting homogeneity as $S\vee T$ is a proper equivalence relation on $T(a)$. What happens if $c/(S\vee T)^b$ extends to only one of $T(a),R(a)$?
\begin{claim}
	$c/(S\vee T)^b\cap R(a)=\varnothing$.
\label{ClaimClasses3}
\end{claim}
\begin{proof}
	Suppose that $c/(S\vee T)^b\cap R(a)\neq\varnothing$. Then, by Claim \ref{ClaimClasses2} $c/(S\vee T)^b\cap T(a)=\varnothing$ and the other $S\vee T$-classes of $T(b)$ are contained in $T(a)$. It follows in particular that there are no $S$- or $T$-edges from $T(b)\cap T(a)$ to $T(b)\cap S(a)$. In particular, the structure
\[
\includegraphics[scale=0.7]{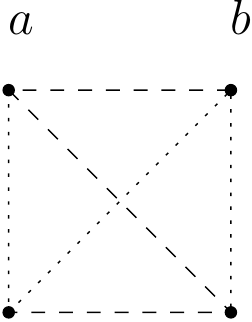}
\]
is forbidden, contradicting the fact (implied by $S(a)\cong K_n^R[\Gamma]$) that $\age(\Gamma^{ST})\subset\age(M)$.
\end{proof}
\begin{claim}
	$c/(S\vee T)^b\cap T(a)=\varnothing$.
\label{ClaimClasses4}
\end{claim}
\begin{proof}
	Arguing as in Claim \ref{ClaimClasses3}, we find that if $c/(S\vee T)^b\cap T(a)\neq\varnothing$, then there are no $S$- or $T$-edges from $T(b)\cap S(a)$ to $T(b)\cap R(a)$ or from $T(b)\cap T(a)$ to $T(b)\cap R(a)$. In particular, the following two structures are forbidden:
\[
\includegraphics[scale=0.7]{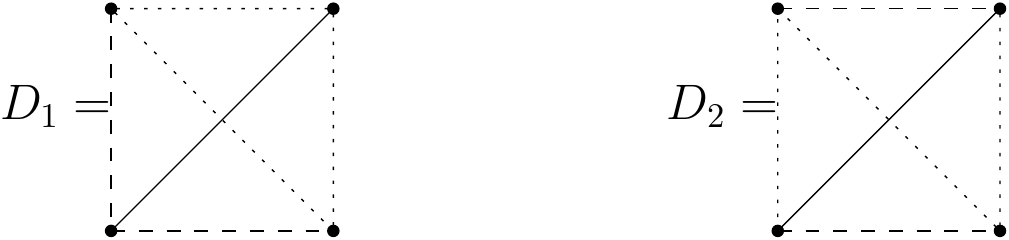}
\]
This is impossible as we have three unstable relations, so by Remark \ref{RmkCompatibilityGraph} at least one of $R\sim S$, $R\sim T$ holds. By Proposition \ref{PropRSRT}, one of $R\sim^ST,R\sim^TS$ holds as $R\not\sim^SS$ and $R\not\sim^TT$ and $R$ is the only nonforking relation. But $D_1$ is in the age of an indiscernible half-graph witnessing $R\sim^TS$ and $D_2$ is in the age of the half-graph for $R\sim^ST$.
\end{proof}
From Claims \ref{ClaimClasses1} to \ref{ClaimClasses4}, we conclude that $c/(S\vee T)^b=T(b)\cap S(a)$.
\begin{claim}
	There are at least three $S\vee T$-classes in $T(a)$.
\label{ClaimClasses5}
\end{claim}
\begin{proof}
	We know that $S\vee T$ is a proper equivalence relation in $T(a)$. Suppose for a contradiction that there are only two $S\vee T$-classes in $T(a)$. Then the $S\vee T$-class in $T(b)$ of any $d$ with $R(d,c)$ meets both $R(a)$ and $T(a)$ (it meets $R(a)$ because $T\sim^SR$, and it meets $T(a)$ because $\age(\Gamma^{ST})\subset\age(M)$). This implies that there are no $S$- or $T$-edges from $T(b)\cap S(a)$ to $T(b)\cap T(a)$ or from $T(b)\cap S(a)$ to $T(b)\cap R(a)$; in particular, this implies that the structure
\[
\includegraphics[scale=0.7]{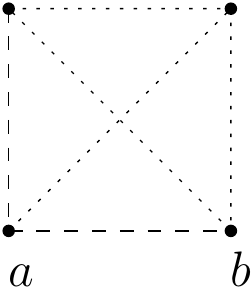}
\]
is forbidden, contradicting $\age(\Gamma^{ST})\subset\age(M)$.
\end{proof}
The same argument from Claim \ref{ClaimClasses5} shows that there are no $S\vee T$-classes $C$ in $T(b)$ such that $C\cap T(a)\neq\varnothing$ and $C\cap R(a)\neq\varnothing$. From this we conclude that each $S\vee T$-class in $T(b)$ is contained in one of $S(a),T(a),R(a)$. But this implies that there are no $S$- or $T$- edges from $T(b)\cap T(a)$ to $T(b)\cap S(a)$, again contradicting $\age(\Gamma^{ST})\subset\age(M)$.
\end{proof}
Propositions \ref{PropStrST} to \ref{PropLastCaseSsimT} prove 
\begin{lemma}\label{lemmaST}
	There are no primitive homogeneous simple 3-graphs in which all relations are unstable and only $R$ is nonforking satisfying $S\sim T$.
\hfill$\Box$\end{lemma}
\begin{corollary}\label{CorTwoInstabilities}
	If $M$ is a primitive homogeneous simple 3-graph in which all relations are unstable and only $R$ is nonforking, then $S\not\sim T$, $R\not\sim^SS$, $R\not\sim^TT$, $R\sim^TS$, $R\sim^ST$. In particular, there are no such graphs in which $R$ is stable.
\end{corollary}
\begin{proof}
	By Remark \ref{RmkCompatibilityGraph}, Proposition \ref{PropRSRT}, and Lemma \ref{lemmaST}.
\end{proof}

We have seen that $S\sim T$ implies that both $S(a)$ and $T(a)$ are unstable 3-graphs. Is it possible to have $S(a)$ and $T(a)$ stable?

\begin{observation}
	If $M$ is a primitive homogeneous simple 3-graph in which $R,S,T$ are unstable, only $R$ is nonforking, and $S(a),T(a)$ are stable 3-graphs, then each of $S(a),T(a)$ is of the form $K^i_m[K^j_n[K^k_o]]$, where $\{i,j,k\}=\{R,S,T\}$ and only the index from $m,n,o\in\omega+1$ corresponding to $R$ is finite.
\label{ObservationStableNeighbourhoods}
\end{observation}
\begin{proof}
	By Remark \ref{CorTwoInstabilities}, $R\sim^ST$ and $R\sim^TS$ are the only instabilities witnessed in $M$, and from this it follows that $S(a)$ and $T(a)$ embed infinite $S$- and $T$-cliques. The observation follows from Theorem \ref{Lachlan3graphs} by inspection.
\end{proof}

Observation \ref{ObservationStableNeighbourhoods} implies that there are, in principle, 36 cases to analyse under the hypothesis of stability for $S(a)$ and $T(a)$. We can eliminate 20 of these cases using only Proposition \ref{PropMultipartite}, since in $K_n^i[K_m^j[K_o^k]]$ (where $\{i,j,k\}=\{R,S,T\}$) the relations $k$ and $j\vee k$ are equivalence relations. In other words, $S(a)$ cannot be of the form $K_\omega^S[K^i_n[K^j_m]]$, where $\{i,j\}=\{R,T\}$, and similarly $T(a)$ is not of the form $K_\omega^T[K^{i'}_n[K^{j'}_m]]$ ($\{i',j'\}=\{R,S\}$). Of the sixteen remaining cases, twelve (those in which $R\vee S$ or $R\vee T$ are the coarsest equivalence relations in $S(a)$, $T(a)$) can be eliminated by looking at the set of forbidden triangles in $S(a)$, $T(a)$ and in $T(b)\cap S(a)$ or $S(c)\cap T(a)$, where $b\in S(a)$ and $c\in T(a)$. For example, if $S(a)\cong K_\omega^T[K_n^R[K_\omega^S]]$ and $T(a)\cong K_\omega^S[K_\omega^T[K_m^R]]$, then $T(b)\cap S(a)$ contains triangles $RRS$, which are forbidden in $T(a)$, contradicting homogeneity.

This leaves us with the following four cases, listed in Table \ref{table2}.
\begin{table}\label{table2}
\centering
\caption{Cases with $S(a),T(a)$ stable.}
\begin{tabular}{ccc} 
\hline\hline 
Case & $S(a)$ & $T(a)$ \\ [0.5ex]
\hline 
I & $K_n^R[K_\omega^S[K_\omega^T]]$ & $K_n^R[K_\omega^S[K_\omega^T]]$ \\
II & $K_n^R[K_\omega^S[K_\omega^T]]$ & $K_n^R[K_\omega^T[K_\omega^S]]$ \\
III & $K_n^R[K_\omega^T[K_\omega^S]]$ & $K_n^R[K_\omega^S[K_\omega^T]]$ \\
IV & $K_n^R[K_\omega^T[K_\omega^S]]$ & $K_n^R[K_\omega^T[K_\omega^S]]$\\
\hline 
\label{TableCases} 
\end{tabular}
\end{table}

Cases I, III, IV are easily eliminated via Theorem \ref{ThmThomas} because the maximal $S$- (or $T$-) cliques containing an $S$- ($T$-) edge form the lines of a weak pseudoplane. This leaves only one case under the assumption of stability for $S(a),T(a)$, namely $S(a)\cong K_n^R[K_\omega^S[K_\omega^T]]$ and $T(a)\cong K_m^R[K_\omega^T[K_\omega^S]]$.

\begin{proposition}\label{propfinalcase}
	Let $M$ be a primitive simple unstable 3-graph in which $R,S,T$ are unstable relations and $R$ is the only nonforking relation. Then $S(a)$ and $T(a)$ are isomorphic to unstable 3-graphs.
\end{proposition}
\begin{proof}
By the paragraph preceding the statement, it suffices to eliminate the case where $S(a)\cong K_n^R[K_\omega^S[K_\omega^T]]$ and $T(a)\cong K_m^R[K_\omega^T[K_\omega^S]]$.

Consider $b\in S(a)$ and $T(b)$. We have $T(b)=(T(b)\cap S(a))\cup(T(b)\cap T(a))\cup(T(b)\cap R(a))$.

By the structure of $S(a)$, $T(b)\cap S(a)$ is the $T$-class of $b\in S(a)$, excluding $b$ itself. By homogeneity, $T(b)\cong T(a)$, so each element of an infinite $T$-clique in $T(b)$ is contained in an $S$-class over $b$. 

Since the $S$-diameter of $M$ is 2, the triangle $TTS$ is in $\age(M)$ and $T(b)\cap S(a)\neq\varnothing$. And since the triangle $RST$ is also in $\age(M)$ (this follows from $R\sim^ST$), we have $R(b)\cap T(a)\neq\varnothing$. By $R\sim^TS$, $R(b)\cap T(a)$ is infinite and contains infinite $T$-cliques; by $R\sim^ST$, $R(b)\cap T(a)$ also contains infinite $S$-cliques. A similar argument proves that $S(b)\cap T(a)$ is infinite and contains infinite $S$- and $T$-cliques.

Let $X$ be the union of $S$-classes $C$ in $T(a)$ such that there are some $c\in T(b)$ and $d\in C$ with $T(d,b)\wedge S(c,d)$, and let $Y$ be the union of $S$-classes $C$ in $T(a)$ such that there are $c\in T(b)$ and $d\in C$ with $S(d,b)\wedge S(c,d)$.Then either $R(b)\cap T(a)$ is contained in $X$ or is disjoint from $X$, and likewise with $Y$. Otherwise, we would be abel to find $d\in T(b)\cap T(a)$ and $e\in R(b)\cap T(a)$ with equal quantifier-free types over $ab$ but distinct types over $ab$.

From this it follows that $R(b)\cap T(a)$ is disjoint from both $X$ and $Y$. By homogeneity, $S(b)\cap T(a)$ and $T(b)\cap T(a)$ are unions of $S$-classes in $T(a)$; but the structure of $S(a)$ must be isomorphic to that of $S(b)$. The only possible way for $S(b)\cap S(a)$ to be a union of $S$-classes in $T(a)$ is for it to consist of exactly one class. But we know that $S(b)\cap T(a)$ contains infinite $T$-classes, so we have reached a contradiction.
\end{proof}

\begin{corollary}\label{CorFinal}
	There are no primitive simple unstable 3-graphs $M$ in which $R,S,T$ are unstable relations and $R$ is the only nonforking relation. 
\end{corollary}
\begin{proof}
By Proposition \ref{propfinalcase}, $S(a)$ or $T(a)$ is unstable. If $S(a)$ is unstable, then, by Corollary \ref{CorTwoInstabilities}, one of $R\sim^TS$ or $R\sim^ST$ holds in each of $S(a)$, $T(a)$. We know by Proposition{PropOneIsStable} that there is a stable relation in $S(a), T(a)$, which in each case must be $R$, which is either an equivalence relation or the complement of one. If $R$ is the complement of an equivalence relation, then each class is infinite and isomorphic to the Random Graph. But then we find $S\sim T$, contradicting Corollary \ref{CorTwoInstabilities}, so $R$ must be an equivalence relation and each of $S(a)$, $T(a)$ is an imprimitive homogeneous unstable 3-graph in which $R$ defines an equivalence relation with finite classes. But then $R$ is stable in $S(a),T(a)$, so they don't embed half-graphs witnessing $R\sim^TS$ or $R\sim^ST$, a contradiction again.
\end{proof}

Now we can conclude:

\begin{theorem}\label{ThmStableForking}
	Let $M$ be a homogeneous primitive simple unstable 3-graph. Then if some relation forks, it is stable.
\end{theorem}
\begin{proof}
We have proved that it is not possible to have one (Lemma \ref{PropNoSimpleOneDividing}) or two forking unstable relations (Corollaries \ref{CorTwoInstabilities}, \ref{CorFinal}). Thus, either we have all relations unstable and nonforking, or the forking relation is stable.
\end{proof}

\chapter{Homogeneous Simple 3-graphs of \SU-rank 1}\label{ChapRank1}
\setcounter{equation}{0}
\setcounter{theorem}{0}
\setcounter{case}{0}
\setcounter{subcase}{0}

This chapter contains more general results than the others. It deals with simple binary homogeneous structures of \SU-rank 1, so in particular we do not restrict ourselves to undirected edges or any particular number of colours. 

Under the assumption of \SU-rank 1, $\tp(a/B)$ forks over $A\subset B$ iff $a\in\acl(B)\setminus\acl(A)$, and algebraic closure on an \SU-rank 1 structure induces a pregeometry.

\subsection{The primitive case}
\begin{proposition}
\label{aclab}
	Let $M$ be a binary homogeneous structure with supersimple theory of \SU-rank 1 such that $\aut(M)$ acts primitively on $M$. Then $\acl(a,b)=\{a,b\}$.
\end{proposition}
\begin{proof}
	Suppose not. Then there is $c\in\acl(ab)\setminus(\acl(a)\cup \acl(b))$ and $a\indep[\varnothing] b$, since by primitivity (Observation \ref{AlgClosure}) $\acl(a)=a$. By primitivity, there is only one strong type of elements over $\varnothing$, and since the rank is finite, this implies that all elements are of the same Lascar strong type. So we have $\Lstp(a)=\Lstp(b)$. Take two elements $c', c''$ realising $\tp(c/a)$ and $\tp(c/b)$ respectively. Note that $c'\indep[\varnothing]a$ and $c''\indep[\varnothing]b$. 

	Therefore we can apply the Independence Theorem to produce $d\models\Lstp(a)\cup\tp(c/a)\cup\tp(c/b)$ with $d\indep[\varnothing]ab$. Since the language is binary, $\tp(d/ab)=\tp(c/ab)$ (an algebraic type), so $d\in\acl(ab)$ which contradicts $d\indep[\varnothing]ab$. 
\end{proof}
A stronger statement is:
\begin{proposition}\label{algclosure}
	Under the hypotheses of \ref{aclab}, $\acl(A)=\bigcup_{a\in A}\acl(a)=A$.
\end{proposition}
\begin{proof}
	We prove this by induction on $|A|$. The case $|A|=1$ is true by primitivity and $|A|=2$ is Proposition \ref{aclab}. Now suppose that the result holds for sets of cardinality $k$, and let $A=\{a_1,\ldots a_{k+1}\}$.

	Suppose that the equality does not hold, and take $b\in \acl(A)\setminus\bigcup_{a\in A}\acl(a)$. By the rank 1 assumption, $a_{k+1}\indep[\varnothing] A_0$, where $A_0=A\setminus\{a_{k+1}\}$. Now take $b_0$ realizing $\tp(b/A_0)$ and $b_1$ realizing $\tp(b/a_{k+1})$. By the induction hypothesis, $b_0\indep[\varnothing]A_0$ and $b_1\indep[\varnothing]a_{k+1}$. By primitivity, $\Lstp(b_0)=\Lstp(b_1)$ (over the empty set), so we can apply the Independence Theorem to get a $\beta\models\tp(b_0/A_0)\cup\tp(b_1/a)$ with $\beta\indep[\varnothing]A$. By rank 1, $\beta$ is not algebraic over $A$.

	But $\tp(\beta/A)=\tp(b/A)$; indeed, $\tp(\beta/A_0)=\tp(b/A_0)$, which implies that $\tp(\beta/\alpha)=\tp(b/\alpha)$ for all $\alpha\in A_0$, and also $\tp(\beta/a)=\tp(b/a)$. Since the language is binary, this implies that $\tp(\beta/A)=\tp(b/A)$. This is a contradiction (because $b$ is algebraic over $A$.)
\end{proof}

	Let $D(\bar x)$ denote the formula expressing that the elements of the tuple $\bar x$ are all different. Recall that the theory of the Random Graph is axiomatised by the set of sentences $\{\phi_{n,m}:n,m\in\omega\}$, where $\phi_{n,m}$ is 
\begin{small}
$$\forall v_1,\ldots,v_n\forall w_1,\ldots,w_m(D(v_1,\ldots,v_n,w_1,\ldots,w_m)\rightarrow\exists x(\bigwedge_{1\leq i\leq n}R(x,v_i)\wedge\bigwedge_{1\leq j\leq m}\neg R(x,w_j))$$ 
\end{small}
When phrased as ``whenever $V_1$ and $V_2$ are finite disjoint sets of vertices in $G$, there exists a vertex $v$ such that for all $v_1\in V_1$ and $v_2\in V_2$ the formula $R(v,v_1)\wedge\neg R(v,v_2)$ holds in $G$," the axiom schema $\phi_{n,m}$ is known as \emph{Alice's restaurant axiom}. 

	We will assume for the rest of this section that $M$ is a binary relational structure, homogeneous in a language $L=\{R_1,\ldots,R_n\}$, and that each 2-type over $\varnothing$ of distinct elements is isolated by one of the relations in the language. Our aim is to show that supersimple primitive binary homogeneous structures are very similar to the random graph, in the sense that we can prove analogues of Alice's restaurant axioms in them. As in other proofs in this chapter, at the core of the argument is the Independence Theorem.

\begin{theorem}
	Let $M$ be a countable relational structure homogeneous in the binary language $L=\{R_1,\ldots,R_n\}$, and assume that each complete 2-type over $\varnothing$ is isolated by one of the $R_i$. Suppose that $R_1,\ldots,R_m$ are symmetric relations and $R_{m+1},\ldots,R_n$ are antisymmetric. If $M$ is primitive and $Th(M)$ is supersimple of \SU-rank 1, then for any collection $\{A_1,\ldots,A_m,A_{m+1},A'_{m+1},\ldots,A_n,A'_n\}$ of pairwise disjoint finite sets of elements from $M$ there exists $v\in M$ such that \[M\models \bigwedge_{i\in\{1,\ldots m\}}(\bigwedge_{v_i\in A_i}R_i(v,v_i))\wedge\bigwedge_{i\in\{m+1,\ldots,n\}}(\bigwedge_{v_i\in A_i}R_i(v,v_i)\wedge\bigwedge_{w_i\in A_i'}R_i(w_i,v))\]
\label{PrimitiveAlice}
\end{theorem}
\begin{proof}
To prove this, we use Proposition \ref{algclosure} and the Independence Theorem. We may assume that all the $A_i, A'_i$ are of the same size, and will prove this proposition for $|A_i|=1$ (it will be clear that the same argument can be iterated for larger sets). By Proposition \ref{algclosure}, $a_1\indep a_2$ if $a_1\neq a_2$, and for any $A,B,C$, $A\indep[C]B$ if $(A\setminus C)\cap(B\setminus C)=\varnothing$. Let $A_i=\{a_i\}$ and $A'_{j}=\{a'_j\}$ for $m+1\leq j\leq n$, and assume all the $a_i$ are different and therefore pairwise independent. Then by homogeneity, there exist $b_i$ with $R_i(a_i,b_i)$, and $\tp(b_i/a_i)$ does not fork over $\varnothing$. By primitivity, $\Lstp(b_1)=\Lstp(b_2)$, so we can apply the Independence Theorem and find $b_{12}\models\Lstp(b_1)\cup\tp(b_1/a_1)\cup\tp(b_2/a_2)$ satisfying $b_{12}\indep a_1a_2$. 

Now we have $b_{12}\indep a_1a_2$ and we know $a_1a_2\indep a_3$ and $\tp(b_3/a_3)$ does not fork over $\varnothing$. Also, by primitivity $\Lstp(b_3)=\Lstp(b_{12})$ and we can apply the independence theorem again. Iterating this process, we find $\alpha\models\Lstp(b_1)\cup\tp(b_1/a_1)\cup\ldots\cup\tp(b_n/a_n)$ independent from $a_1,\ldots,a_m,a_{m+1},a'_{m+1},\ldots,a_n, a'_n$. 
\end{proof}

\subsection{Finite equivalence relations}\label{SubsectFER}
If $M$ is a transitive, imprimitive rank 1 structure in which all the definable equivalence relations have infinite classes, then it follows from the rank hypothesis that each of the equivalence relations has finitely many classes. From homogeneity and transitivity it follows that if $E$ is a definable equivalence relation on $M$ and $\neg E(a,b)$, then $a/E$ and $b/E$ are homogeneous structures with the same age, and each has fewer definable equivalence relations than $M$. By $\omega$-categoricity, there are only finitely many definable equivalence relations, so that $M$ is in fact the union of finitely many primitive homogeneous structures (which are the equivalence classes of the finest definable equivalence relation on $M$ with infinite classes) in which all invariant equivalence relations have finite classes. Our next goal is to describe how two classes of a finite equivalence relation in a rank 1 binary homogeneous structure can relate to each other.

The archetypal example of an imprimitive simple unstable binary homogeneous structure with a finite equivalence relation is the Random Bipartite Graph. It is the Fra\"iss\'e limit of the family of all bipartite graphs with a specified partition or equivalence relation; it is not homogeneous as a graph, but is homogeneous in the language $\{R,E\}$, where $E$ is interpreted as an equivalence relation. To axiomatise this theory, it suffices to express that $E$ is an equivalence relation with exactly two infinite classes, $R$ is a graph relation, and that for any finite disjoint subsets $A_1,A_2$ of the same $E$-class there exists a vertex $v$ in the opposite class such that $R(v,a)$ holds for all $a\in A_1$ and $\neg R(v,a')$ holds for all $a'\in A_2$.

If $A,B$ are different classes of the finest definable finite equivalence relation $E$ on $M$, we will say that a relation $R$ holds \emph{transversally} or \emph{across} $A, B$ if there exist $a\in A$ and $b\in B$ such that $R(a,b)\vee R(b,a)$. Relations which hold transversally for some pair of $E$-classes are refered to as \emph{transversal relations}. Notice that by homogeneity any relation holding across $E$-classes does not hold within a class, and vice-versa. By quantifier elimination and our assumption on the disjointness of the binary relations, $E$ is defined by a disjunction of atomic formulas $\bigvee_{i\in I}R_i(x,y)$ for some $I\subset\{1,\ldots n\}$. Therefore, the transversal relations are those in $L\setminus \{R_i:i\in I\}$. We assume that each 2-type of distinct elements is isolated by a relation in the language; therefore, each relation is either symmetric or antisymmetric.

Given two $E$-classes $A,B$, if only one symmetric relation $R$ holds across $A,B$ then we say that $R$ is \emph{complete bipartite} in $A,B$, for the reason that if we forget the structure within the classes, what we obtain is a complete bipartite graph. All other relations are \emph{null} across $A,B$ in this case, \ie not realised across these classes. 

If $D$ is an antisymmetric relation realised across $A,B$, we say that the ordered pair of classes $(A,B)$ is \emph{directed for} $D$ if all the $D$-edges present in $A\cup B$ go in the same direction, that is, if either $\forall(c,c'\in A\cup B)(D(c,c')\rightarrow c\in A\wedge c'\in B)$ or $\forall(c,c'\in A\cup B)(D(c,c')\rightarrow c\in B\wedge c'\in A)$. A dramatic example of a $D$-directed pair of $E$-classes is when $\forall a\in A\forall b\in B(D(a,b))$. We adopt the convention that if $(A,B)$ is directed for $D$, then the $D$-edges go from $A$ to $B$. If $(A,B)$ is not directed for any $D$, then we say that $(A,B)$ is an \emph{undirected} pair of $E$-classes. 

\begin{observation}
	Let $M$ be a binary homogeneous imprimitive transitive relational structure in which there are proper nontrivial invariant equivalence relations with infinite classes. Let $E$ be the finest such equivalence relation in $M$. If $(A,B)$ is a directed pair of equivalence classes for some $D\in L$, then no symmetric relations are realised across $A,B$ and for all antisymmetric relations $D'$ in the language realised across $A,B$, either $(A,B)$ or $(B,A)$ is directed for $D'$.
\label{DirectedPairsClasses}
\end{observation}
\begin{proof}
	The first assertion follows from the fact that if $R(a,b)$ for some symmetric relation $R$, where $a\in A$ and $b\in B$, then by homogeneity there would exist an automorphism taking $a\rightarrow b$ and $b\rightarrow a$, which is impossible by invariance of $E$ and the fact that $(A,B)$ is directed for $D$. Similarly, if for some directed relation $D'$ we had $a,a'\in A$ and $b\in B$ with $D'(a,b)\wedge D'(b,a')$ then by homogeneity there would exist an automorphism of $M$ taking $ab$ to $ba'$, again impossible since $(A,B)$ is directed for $D$. 
\end{proof}

\begin{observation}
	Let $M$ be a binary homogeneous imprimitive transitive relational structure with supersimple theory of \SU-rank 1 in which there are proper nontrivial invariant equivalence relations with infinite classes. Let $E$ be the finest such equivalence relation in $M$, and assume that $\aut(M)$ acts primitively on each $E$-class. If $a_1,\ldots,a_n$, $n\geq2$, are distinct $E$-equivalent elements of $M$, then $a_1\indep a_2,\ldots,a_n$.
\label{IndepInClass}
\end{observation}
\begin{proof}
	We proceed by induction on $n$. For the case $n=2$, let $a_1,a_2$ be distinct elements of $M$, $E(a_1,a_2)$. In the situation described, each of the relations that imply $E$ is non-algebraic, since otherwise the action of $\aut(M)$ on $a_1/E$ would not be primitive. It follows that the relation isolating $\tp(a_1a_2)$ is nonforking, so $a_1\indep a_2$.
	
	Now suppose that any $k$ distinct $E$-equivalent elements of $M$ are independent. Suppose for a contradiction that $a_1,\ldots,a_{k+1}$ are pairwise independent $E$-equivalent elements of $M$, and $a_{k+1}\notindep a_1,\ldots a_k$. By the induction hypothesis, $a_1\indep a_2,\ldots,a_k$, $a_{k+1}\indep a_1$ and $a_{k+1}\indep a_2,\ldots,a_k$. Let $b_1\models\tp(a_{k+1}/a_1)$ and $b_2\models\tp(a_{k+1}/a_2,\ldots,a_k)$; these are nonforking extensions of the unique 1-type over $\varnothing$ to $a_1$ and $a_2,\ldots,a_k$, and are of the same strong type. Therefore, by the Independence Theorem, there exists $c$ satisfying $\tp(a_{k+1}/a_1)\cup\tp(a_{k+1}/a_2,\ldots,a_k)$ in the same class as $a_{k+1}$, independent (\ie non-algebraic) from $a_1,\ldots,a_k$. But then $\tp(c/a_1,\ldots,a_k)=\tp(a_{k+1}/a_1,\ldots,a_k)$ because the language is binary, which is impossible as the type on the left-hand side of the equality is non-algebraic, while the other one is algebraic. 
\end{proof}

Given a pair of $E$-classes $A,B$, denote the set of nonforking transversal relations realised in $A\cup B$ by $\mathcal I(A,B)$. If $(A,B)$ is a directed pair of classes, then $\mathcal I^*(A,B)$ is the set of nonforking relations $D$ realised in $A\cup B$ such that $D(a,b)$ for some $a\in A, b\in B$. Note that for directed pairs, $\mathcal I(A,B)=\mathcal I^*(A,B)\cup\mathcal I^*(B,A)$.

\begin{proposition}
	Let $M$ be a binary homogeneous imprimitive transitive relational structure with supersimple theory of \SU-rank 1 in which there are proper nontrivial invariant equivalence relations with infinite classes. Let $E$ be the finest such equivalence relation in $M$, and assume that $\aut(M)$ acts primitively on each $E$-class. Suppose that $(A,B)$ is a $D_1$-directed pair of $E$-classes. Enumerate $\mathcal I^*(A,B)=\{D_1,\ldots, D_n\}$ and $\mathcal I^*(B,A)=\{Q_1,\ldots,Q_m\}$. Then for all finite disjoint $V_1,\ldots,V_n\subset B$ and $W_1,\ldots,W_m\subset A$ there exist $c\in A$ and $d\in B$ such that $D_i(c,v)$ holds for all $v\in V_i$ ($1\leq i\leq n$) and $Q_j(w,d)$ holds for all $w\in W_j$ ($1\leq j\leq m$).
\end{proposition}
\begin{proof}
We will prove only that for all finite disjoint $V_1,\ldots,V_n\subset B$ there exists $c\in B$ such that $D_i(c,v)$ holds for all $v\in V_i$; the same argument produces the $d$ from the statement. 

We proceed by induction on $k=|V_1|+\ldots+|V_n|$, with an inner induction argument. If $k=n$, so $V_i=\{b_i\}$ then by Observation \ref{IndepInClass} we have $b_1\indep b_2$. There exist $a,a'\in A$ such that $D_1(a,b_1)\wedge D_2(a',b_2)$; since $D_1$ and $D_2$ are nonforking relations, $a\indep b_1$ and $a'\indep b_2$, and since $a,a'$ are $E$-equivalent, they have the same strong type. By the Independence Theorem, there exists $c_{12}\in A$ such that $c_{12}\indep b_1b_2$ and $D_1(c_{12},b_1)\wedge D_2(c_{12},b_2)$. Now suppose that for $t\leq n-1$, we can find $c_{1\ldots t}\indep a_1,\ldots, a_t$ such that $D_1(c_{1\dots t},b_1)\wedge\ldots\wedge D_t(c_{1\ldots t}, b_t)$. Given distinct $b_1,\ldots, b_{t+1}$ with $t+1\leq n$, it follows from Observation \ref{IndepInClass} that $b_{t+1}\indep b_1,\ldots,b_t$. By the induction hypothesis, there exists $c_{1\ldots t}\indep b_1,\ldots,b_t$ satisfying $\bigwedge_{i=1}^t D_i(c_{1\ldots t},b_i)$; and we know that there exists $c_{t+1}\in A$ such that $D_{t+1} (c_{t+1},b_{t+1})$. Since $D_{t+1}$ is nonforking, $c_{t+1}\indep b_{t+1}$, and by the Independence Theorem, there exists $c_{1\ldots t+1}\indep b_1,\ldots,b_{t+1}$ such that $\bigwedge_{i=1}^{t+1}D_i(c_{1\ldots t+1},b_i)$. This concludes, by induction, the case $k=n$. The same argument proves the inductive step on $k$.
\end{proof}

By the same argument, we can prove:
\begin{proposition}
	Let $M$ be a binary homogeneous imprimitive transitive relational structure with supersimple theory of \SU-rank 1 in which there are proper nontrivial invariant equivalence relations with infinite classes. Let $E$ be the finest such equivalence relation in $M$, and assume that $\aut(M)$ acts primitively on each $E$-class. Suppose that $(A,B)$ is an undirected pair of $E$-classes, $\mathcal I(A,B)=\{R_1,\ldots,R_k\}\cup\{D_1,\ldots, D_s\}$, where each $R_i$ is symmetric and each $D_j$ is antisymmetric. Then for all finite disjoint subsets $V_1,\ldots,V_k,W_1,\ldots,W_s,W'_1,\ldots,W'_s\subset B$ there exists $c\in A$ such that $R_i(c,v)$ for all $v\in V_i$, $D_j(c,w)$ for all $w\in W_j$, and $D_j(w,c)$ for all $w\in W_j'$.
\label{ComplicatedAlice}
\end{proposition}

We remark here that if all the relations are symmetric, Proposition \ref{ComplicatedAlice} says that a nonforking transversal relation $R$ occurs across a pair of $E$-classes $A,B$ in one of three ways, namely:
\begin{enumerate}
	\item{Complete, that is, only one relation is realised across $A,B$,}
	\item{Null, so $R$ is not realised in $A\cup B$}
	\item{Random bipartite: it satisfies that given two disjoint nonempty finite subsets $V,V'$ of $A$ ($B$), there is a vertex $v$ in $B$ ($A$) that is $R$-related to all vertices from $V$ and to none from $V'$}
\end{enumerate}

The results in this section tell us exactly what to expect from binary supersimple homogeneous structures of \SU-rank 1. Even though we did not phrase it as a list of structures, Proposition \ref{ComplicatedAlice} is essentially a classification result for imprimitive binary homogeneous structures of \SU-rank 1 in which one of the relations defines an equivalence relation with infinite classes. Our next proposition is, in the same sense, a classification of unstable imprimitive simple 3-graphs (language $\{R,S,T\}$, all relations symmetric and irreflexive, each pair of distinct vertices realises exactly one of them) in which one of the predicates defines a finite equivalence relation. This result is of interest in the final sections of this chapter; we make implicit use Proposition \ref{CatSimpLow}:

\begin{proposition}
	Let $M$ be a transitive simple unstable homogeneous 3-graph in which $R$ defines an equivalence relation with $m<\omega$ classes. Then $M$ has supersimple theory of \SU-rank 1, the structure induced on each pair of classes is isomorphic to the Random Bipartite Graph, and for all $k\leq m$ and all $k$-sets of $R$-classes $X$, any $S,T$-graph of size $k$ is realised as a transversal to $X$.
\label{PropImprimitive3Graphs}
\end{proposition}
\begin{proof}
	The first assertion follows easily from transitivity (only one 1-type $q_0$ over $\varnothing$) and the fact that if $\varphi(x,\bar a)$ is a formula not implying $x=a_i$ for some $a_i\in\bar a$, then $\varphi$ does not divide over $\varnothing$, so the only forking extensions to the unique 1-type over $\varnothing$ are algebraic. To see this, consider any such $\varphi(x,\bar a)$. We may assume that $\varphi$ is not algebraic, as in that case we would already know that any extension of $q_0$ implying $\varphi$ is algebraic and so of \SU-rank 0. Let $c$ realise this formula, $c\not\in\bar a$. We wish to prove that $c\indep\bar a$; by simplicity, this is equivalent to proving $\bar a \indep c$. 

	Let $\varphi'(\bar x,c)$ be the formula isolating $\tp(\bar a/c)$. Consider any $\varnothing$-indiscernible sequence $I=(c_i:i\in\omega)$ such that $c\in I$. This is an infinite sequence contained in the $R$-class of $c$. Colour the elements of $I$ according to the types they realise over $\bar a$. Since $\bar a$ is finite, there are only finitely many colours, and by the pigeonhole principle there is an infinite monochromatic subset $I'$ of $I$. Then we have $I'\equiv_c I$ and $I'$ is indiscernible over $\bar a$, so $\varphi(x,\bar a)$ does not divide over $\varnothing$ and the \SU-rank of $q_0$ (and therefore $M$) is 1.

	The relation $R$ is clearly stable in $M$, so $S$ and $T$ must be unstable. By instability, there are parameters $a_i$, $b_i$ ($i\in\omega$) such that $S(a_i,b_j)$ holds iff $i\leq j$. Since $R$ is stable, we have $T(a_i,b_j)$ for all $j<i$ in this sequence of parameters. If we consider the $a_ib_i$ as pairs of type $S$ and colour the pairs of distinct pairs in the sequence by the type they satisfy over $\varnothing$, then using Ramsey's theorem we can extract an infinite $\varnothing$-indiscernible sequence of pairs, which we also call $a_i, b_i$. By indiscernibility, the new $a_i$ and $b_i$ form monochromatic cliques, which are of colour $R$ because there are no other infinite monochromatic cliques in $M$. This proves that $S$ and $T$ are realised as transversals to any pair of $R$-classes. By homogeneity, all pairs of classes are isomorphic.

	The relation $R$ is clearly nonforking in $M$. By instability, both $S$ and $T$ are non-algebraic, so for any $a\in M$ the sets $S(a)$ and $T(a)$ contain infinite $R$-cliques. It follows that $S$ and $T$ are nonforking transversal relations, so by Proposition \ref{ComplicatedAlice} the structure on any pair of $R$-classes is isomorphic to the Random Bipartite Graph. Using the Independence Theorem, we can embed any $S,T$-graph of size $k$ as a transversal to a union of $k$ $R$-classes, for any $k\leq m$.
\end{proof}

The structures isolated by Proposition \ref{PropImprimitive3Graphs} consist of a finite number $n$ of infinite $R$-cliques, with $S$ and $T$ realised randomly between them. Let us call the structure with $n$ infinite $R$-classes $\mathcal B_n^{ij}$, where $i,j\in\{R,S,T\}$ are the unstable relations (these structures will appear again in Theorem \ref{ThmClassification}).
\chapter{The Rank of Homogeneous Simple 3-graphs}\label{ChapPrimitive}
\setcounter{equation}{0}
\setcounter{theorem}{0}
\setcounter{case}{0}
\setcounter{subcase}{0}

In this chapter we will find all the primitive simple unstable 3-graphs and those imprimitive ones with finitely many infinite classes, but the main result is that primitive homogeneous simple 3-graphs cannot have $\SU$-rank 2 or higher. The proof of this fact consists of two parts: first we prove that there are no such structures of rank 2, and then prove by induction that there are no structures of any higher finite rank. It was recently proved by Koponen that binary homogeneous simple structures are supersimple and have finite $\SU$-rank, so this is enough. 

\section{The Rank of Primitive Homogeneous Simple 3-graphs}\label{SectRank2}

	We will prove in this section that all primitive simple unstable 3-graphs have \SU-rank 1. From this and some basic results from Chapter \ref{ChapGenRes}, it will follow that the only such 3-graph is the random 3-graph (the Fra\"iss\'e limit of the class of all finite 3-graphs).

	Let $M$ be a simple homogeneous unstable 3-graph. Of the three relations $R,S,T$ (all of which are realised in $M$), we assume that $R$ is stable and forking, and $S,T$ are nonforking. This assumption (not needed in the proof of Theorem \ref{NoRank2Graphs} below) is justified by Theorem \ref{ThmStableForking}. Given any $a\in M$, consider $R(a)$. This is a definable set of rank at most 1 by our assumptions on $R$ and the rank of $M$. What is the structure of $R(a)$? 

The main theorem of this chapter is:

\begin{theorem*}\label{MainThm}
Let $M$ be a primitive simple homogeneous 3-graph. Then the theory of $M$ is of \SU-rank 1.
\end{theorem*}

	To prove this theorem, we prove first 
\begin{theorem*}\label{NoRank2Graphs}
	There are no simple primitive homogeneous 3-graphs of \SU-rank 2.
\end{theorem*}

This last result is proved by arguing first that $R$ defines an equivalence relation on $R(a)$ with finitely many classes; we use the imprimitivity blocks of the $R$-neighbourhoods to define an incidence structure. This incidence structure is a semilinear space. The analysis divides into two main cases, depending on the $R$-diameter of the 3-graph; most of the work goes into proving the non-existence of primitive homogeneous simple 3-graphs of \SU-rank 2 and $R$-diameter 2. The case with $\diam_R(M)=3$ is considerably easier.

	The proof of the first of these theorems rests on the possibility of defining the semilinear space. We use this observation to start an inductive argument on the rank of the structure, and the second theorem is the basis for induction in that proof.

For most of the chapter, we will assume that the \SU-rank of $\Th(M)$ is 2; by Theorem \ref{ThmStableForking}, the elements in a primitive simple homogeneous 3-graph satisfy the statement ``if $\tp(a/B)$ divides over $A\subseteq B$, then dividing is witnessed by a stable formula". In statements where the language is $\{R,S,T\}$, we assume that $R$ is a forking relation ($R(a,b)$ implies $\tp(a/b)$ divides over $\varnothing$), and therefore stable. In view of Lachlan's classification of stable homogeneous 3-graphs (see Theorem \ref{Lachlan3graphs}), we may suppose that $\Th(M)$ is unstable. Since any Boolean combination of stable formulas is stable, it follows that both $S$ and $T$ are unstable, therefore nonforking. Statements for the language $\{R_1,\ldots,R_n\}$ may be more general and refer to homogeneous $n$-graphs. Note that if all relations are nonforking then a primitive structure $M$ is random in the sense that all its minimal forbidden structures are of size 2 (examples: the Random Graph, Random $n$-edge-coloured graphs), by the Independence Theorem argument used in the proof of Theorem \ref{PrimitiveAlice}.

	Recall that for any relation $P$ and tuple $\bar a$, $P(\bar a)=\{\bar x\in M| P(\bar a, \bar x)\}$. We sometimes refer to this set as the $P$-neighbourhood of $\bar a$. In Definition \ref{DefnGraph}, we defined an $n$-graph to be a structure $(M, R_1,\ldots,R_n)$ in which each $R_i$ is binary, irreflexive and symmetric; also, we assume that for all distinct $x,y\in M$ exactly one of the $R_i$ holds and $n\geq2$. Finally, if $M$ is a homogeneous $n$-graph, we assume that for each $i\in\{1,\ldots,n\}$ there exist $a_i,b_i\in M$ such that $R_i(a_i,b_i)$ holds in $M$. 

By simplicity, forking and dividing coincide, so in our statements and arguments we usually prove or use dividing instead of forking. We assume that all relations in the language are realised in $M$.

At this point, we know that there are two possibilitiese for the structure of 3-graphs of rank 2 with relations $R,S,T$: either $(M,R)$ has diameter 2, or it has diameter 3. In the latter case, since $\aut(M)$ preserves the $R$-distance, for any $a\in M$ the sets $S(a)$ and $T(a)$ correspond to $R$-distance 2 and 3 from $a$, so $\aut(M,R)=\aut(M,R,S,T)$. 

\section{Semilinear 3-graphs of \SU-rank 2}\label{SectRank2}

	In this section we establish the basis for the induction that will eventually yield the main result of this chapter, namely that all primitive homogeneous simple 3-graphs have $\SU$-rank 1. By Theorem \ref{ThmStableForking}, we may assume that the forking relation $R$ is stable. We cannot have more than one forking relation because we assume that each relation in the language isolates a 2-type, so by Theorem \ref{ThmStableForking}, if we had two forking relations then both would be stable, which would imply that the third (which is equivalent to the negation of the other two) is also a stable relation, and the theory of the homogeneous 3-graph would be stable; and by Theorem \ref{Lachlan3graphs} (due to Lachlan), there are no primitive stable 3-graphs. Here we start a case-by-case analysis of these graphs.

\begin{observation}
	If $M$ is a primitive simple $\omega$-categorical relational structure of \SU-rank 2, and $R$ is a forking relation, then $R(a)$ is a set of rank 1.
	\label{Rank1}
\end{observation}
\begin{proof}
Given any $a\in M$, $R(a)$ is a set of rank at most 1. If it were of rank 0, then the set of solutions of $R(x,a)$ would be finite, and therefore any element satisfying it would be in the algebraic closure of $a$, impossible by Observation \ref{AlgClosure}. Therefore, the rank of $R(a)$ is 1. 
\end{proof}

\begin{proposition}\label{EmbedsAllRComplete}
	Suppose that $M$ is a simple primitive homogeneous $R,S,T$-graph, the formula $R(x,a)$ forks, and $S,T$ are unstable, nonforking relations. Then $M$ embeds $K_n^R$ for all $n\in\omega$.
\end{proposition}
\begin{proof}
Being $K_n^R$-free would either force $R(a)$ to be algebraic, contradicting primitivity by Observation \ref{AlgClosure}, or contradict, by Ramsey's Theorem, Observation \ref{NoInfiniteCliques}.
\end{proof}

Note that if $M$ is a simple 3-graph in which $R$ is stable, then $R$ is still a stable relation in the (homogeneous, simple) structure $R(a)$, since any model of the theory of $R(a)$ can be defined in a model of the original theory, and therefore witnesses for instability in $R(a)$ theory would also witness instability in the original theory. 

What can we say about $R(a)$? We will show in the next section that the action of $\aut(M/a)$ is imprimitive on $R(a)$, and that the vertices together with the imprimitivity blocks of their neighbourhoods form a semilinear space. In our argument, we will use Lachlan's classification of stable homogeneous 3-graphs (Theorem \ref{Lachlan3graphs}).

We summarise some properties of some of the infinite stable homogeneous 3-graphs in the table in page \pageref{table}. We present only those structures that may appear as $R(a)$ in a primitive homogeneous 3-graph.

\begin{table}[h]\label{table}
\caption{Some stable homogeneous 3-graphs}
\centering
\begin{tabular}{ccc}
\hline\hline
Structure&Equivalence relations&U-rank\\
$P^R[K_{\omega}^R]$&$R$&1\\
$K_{\omega}^R[Q^R]$&$S\vee T$&1\\
$Q^R[K_{\omega}^R]$&$R$&1\\
$K_{\omega}^R[P^R]$&$S\vee T$&1\\
$K_{\omega}^R\times K_n^S$&$R,S$&1\\
$K_{\omega}^R\times K_n^T$&$R,T$&1\\
$K_{\omega}^R[K_n^S[K_p^T]]$&$S\vee T,T$&1\\
$K_{\omega}^R[K_n^T[K_p^S]]$&$S\vee T,S$&1\\
$K_m^S[K_{\omega}^R[K_p^T]]$&$T\vee R,T$&1\\
$K_m^T[K_{\omega}^R[K_p^S]]$&$S\vee R,S$&1\\
$K_m^S[K_n^T[K_{\omega}^R]]$&$R\vee T,R$&1\\
$K_m^T[K_n^S[K_{\omega}^R]]$&$R\vee S,R$&1\\
\hline
\end{tabular}
\end{table}

\subsection{Lines}\label{lines}
In this subsection we define the main tool that we will use to eliminate candidates to be primitive homogeneous 3-graphs of \SU-rank 2, a family of definable sets we call \emph{lines}. Thus we interpret an incidence structure in $M$ in which lines are infinite and each point belongs to a finite number of lines. It is tempting to try to see this structure as a pseudoplane and use a general result of Simon Thomas on the nonexistence of binary omega-categorical pseudoplanes (see \cite{thomas1998nonexistence}), but our incidence structure falls short of being a pseudoplane or even a weak pseudoplane, which is what Thomas uses in his proof. It is a semilinear space (see Definition \ref{DefSemilinear}), which under some conditions also qualifies as a generalised quadrangle (cf. Observation \ref{Quadrangle}, see the paragraph preceding it for the definition of generalised quadrangle).

\begin{notation}
If the $R$-diameter of $M$ is 3, we adopt the convention that $S(a)$ and $T(a)$ correspond to $R^2(a)$ and $R^3(a)$ (cf. the paragraph after Proposition \ref{Year2}). Note that in $R$-diameter 3, the triangle $RRT$ is forbidden, and therefore the $R$-neighbourhood of any vertex $a$ is an $R,S$-graph, stable by the stability of $R$. 
	\label{Diam2Diam3}
\end{notation}

\begin{definition}
A semilinear space $S$ is a nonempty set of elements called \emph{points} provided with a collection of subsets called \emph{lines} such that any pair of distinct points is contained in at most one line and every line contains at least three points. 
\label{DefSemilinear}
\end{definition}

\begin{remark}
	{\rm As we have mentioned before, these structures are related to weak pseudoplanes. Given a structure $M$ and a definable family $\mathcal B$ of infinite subsets of $M$, the incidence structure $P=(M,\mathcal B)$ is a weak pseudoplane if for any distinct $X,Y\in\mathcal B$ we have $|X\cap Y|<\omega$ and each $p\in M$ lies in infinitely many elements of $\mathcal B$. The connection between our semilinear spaces and weak pseudoplanes is, then, that a semilinear space interpreted (\ie the lines form a definable family of subsets of $M$) in a homogeneous structure in which each line is infinite and each point lies in infinitely many lines is a weak pseudoplane. In all the semilinear spaces that we will encounter in this chapter, lines are infinite and each point belongs to finitely many lines.}
\end{remark}

The rest of this chapter consists of a study of the properties of a semilinear space definable in homogeneous primitive 3-graphs of \SU-rank greater than or equal to 2. 

\begin{proposition}
	Let $M$ be an infinite 3-graph such that $Aut(M)$ acts transitively on $M$, $R(a)$ is infinite for $a\in M$, and $R$ defines an equivalence relation on $R(a)$ with finitely many equivalence classes. Denote by $\ell(a,b)$ the maximal $R$-clique in $M$ containing the $R$-edge $ab$. Then $(M,\mathcal{L})$, where $\mathcal{L}=\{\ell(a,b):M\models R(a,b)\}$, is a semilinear space.
\label{Semilinear3Graph}
\end{proposition}
\begin{proof}
	We start by justifying our use of \emph{the} when we said that $\ell(a,b)$ is ``the maximal $R$-clique in $M$ containing the $R$-edge $ab$." Since we chave $R(a,b)$, we know that $b\in R(a)$, so it is an element of one of the finitely many classes of $R$ in $R(a)$. Let $b/R^a$ denote the $R$-equivalence class of $b$ in $R(a)$; then $\{a\}\cup b/R^a$ is an infinite clique containing $a,b$. We claim that any clique containing $a,b$ is a subset of $\{a\}\cup b/R^a$. To see this, let $K$ be an $R$-clique containing $a,b$, and let $x\neq a,b\in K$. Such an $x$ exists because $R$ partitions an infinite set into finitely many subsets. Since $K$ is a clique, we have that $x\in R(a)$, and as $R$ defines an equivalence relation on $R(a)$ and $R(x,b)$ holds, we have that $x\in b/R^a$. Therefore, $x\in\{a\}\cup b/R(a)$ and $\ell(a,b)$ denotes this set. 

	So we have that two distinct points (vertices) belong to at most one element of $\mathcal{L}$. Any line contains at least three points, by transitivity of $M$ and the fact that $R$ forms infinite cliques within $R(a)$. 
\end{proof}

\begin{definition}
	A 3-graph is \emph{semilinear} if it satisfies the hypotheses of Proposition \ref{Semilinear3Graph}. In particular, whenever we refer to a semilinear 3-graph in this chapter we assume that points are incident with only finitely many lines.
\end{definition}

\begin{definition}\label{DefLines}
	If $M$ is a semilinear 3-graph and $R(a,b)$ holds in $M$, then $\ell(a,b)$ is the imprimitivity block in $R(a)$ to which $b$ belongs, together with the vertex $a$. Equivalently, it is the largest $R$-clique in $M$ containing $a$ and $b$. We refer to these sets as \emph{lines}.
\end{definition}

We have introduced semilinear 3-graphs because a good deal of the analysis of homogeneous primitive 3-graphs of \SU-rank 2 depends more on this combinatorial property than on any simplicity or rank assumptions. The next two results establish that anything we prove about semilinear 3-graphs is also true of homogeneous primitive 3-graphs of \SU-rank 2.

\begin{observation}
	Suppose $M$ is a primitive homogeneous simple 3-graph of \SU-rank 2, where $R$ is a forking relation, $S,T$ are nonforking, and $a\in M$. Then $R(a)$ is imprimitive.
	\label{R(a)Imprimitive}
\end{observation}
\begin{proof}
	If the $R$-diameter is 2, then all three predicates are realised in $R(a)$. By Proposition \ref{Rank1}, $R(a)$ is a 3-graph of rank 1, so it cannot be primitive and unstable by Proposition \ref{Year2}, as it would embed infinite $S$-cliques, contradicting Observation \ref{NoInfiniteCliques}. And by Lachlan's Theorem \ref{Lachlan3graphs}, $R(a)$ cannot be primitive and stable (see the table of 3-graphs without infinite $S$- or $T$-cliques in page \pageref{table}).

	If the $R$-diameter is 3, then $R(a)$ is a homogeneous $RS$-graph. It follows from the Lachlan-Woodrow Theorem \ref{LachlanWoodrow} and simplicity that $R(a)$ is isomorphic to $I_m[K_\omega]$ or to $I_\omega[K_n]$ ($m,n\in\omega$). 
\end{proof}

\begin{proposition}
	If $M$ is a homogeneous simple primitive 3-graph of \SU-rank 2, then $R$ defines an equivalence relation on $R(a)$ with finitely many infinite classes.
\label{RDefinesEqRel}
\end{proposition}
\begin{proof}
We know from Observation \ref{R(a)Imprimitive} that $R(a)$ is imprimitive. By quantifier elimination and our assumption that exactly one of $R,S,T$ holds for any pair of vertices in $M$, to show that $R$ defines an equivalence relation on $R(a)$, an invariant equivalence relation on $R(a)$ is defined by a disjunction of at most two predicates from $L$. Our two main cases depend on the $R$-diameter of $M$.

\begin{case}
If $\diam_R(M)=3$, then $R(a)$ is a homogeneous $R,S$-graph, which must be stable since $R$ is stable and in which both $R$ and $S$ are realised, by Observation \ref{NotRComplete}. The formula $S(x,y)$ does not define an equivalence relation on $R(a)$ by Proposition \ref{PropMultipartite}. Therefore, $R$ is an equivalence relation on $R(a)$ and by Observation \ref{NoInfiniteCliques}, this equivalence relation has finitely many classes, each of which is infinite by homogeneity and the fact that $R(a)$ is an infinite set.
\end{case}
\begin{case}
If $\diam_R(M)=2$, then all predicates are realised in $R(a)$. 

By Proposition \ref{PropMultipartite} the relation $S\vee T$ does not define an equivalence relation. 

If $R\vee S$ defines an equivalence relation on $R(a)$, then it must have finitely many classes as any transversal to $R\vee S$ is a $T$-clique and $T$ does not form infinite cliques in $R(a)$. Each $R\vee S$-class in $R(a)$ is a homogeneous graph, so by the Lachlan-Woodrow Theorem \ref{LachlanWoodrow} it must be of the form $K_n^S[K_\omega^R]$, since $K_n^R[K_\omega^S]$ is impossible because $S$ forms infinite cliques in it. It follows that $R(a)$ is isomorphic to $K_m^T[K_n^S[K_\omega^R]]$, and $R$ defines an equivalence relation on $R(a)$ with $m\times n$ infinite classes (see the table on page \pageref{table}). The same argument shows that if $R\vee T$ defines an equivalence relation on $R(a)$, then $R$ is also an equivalence relation there, with finitely many infinite classes.

If $S$ defines an equivalence relation on $R(a)$, then it is a stable relation on $R(a)$, its classes are finite, and $R(a)$ is a stable 3-graph of one of the forms 6-11 from Lachlan's Theorem \ref{Lachlan3graphs}. We can eliminate all those stable graphs in which $S\vee T$, $R\vee S$, or $R\vee T$ defines an equivalence relation, since we have already dealt with those cases. In all other cases (see the table on page \pageref{table}), $R$ defines an equivalence relation with finitely many infinite classes.
\end{case}
\end{proof}
\setcounter{case}{0}

Observation \ref{R(a)Imprimitive} and Proposition \ref{RDefinesEqRel} tell us that in simple homogeneous primitive 3-graphs of \SU-rank 2 the forking predicate $R$ defines an equivalence relation on $R(a)$ with finitely many infinite classes. We summarise this in a lemma for easier reference:

\begin{lemma}
	Primitive homogeneous simple 3-graphs of \SU-rank 2 are semilinear. The lines of the semilinear space are infinite and each point is incident with finitely many lines. 
\label{LemmaSemilinear}
\end{lemma}
\begin{proof}
	By primitivity, none of the relations $R,S,T$ is algebraic (cf. Observation \ref{AlgClosure}), so $R(a)$ is infinite. The transitivity of the 3-graph follows trivially from primitivity. Observation \ref{R(a)Imprimitive} and Proposition \ref{RDefinesEqRel} prove that (the reflexive closure of) $R$ is an equivalence relation on $R(a)$ with finitely many infinite classes.
\end{proof}

We have defined a semilinear space over a homogeneous structure, but there is no reason for it to be homogeneous as a semilinear space. This observation differentiates our work from Alice Devillers' study of homogeneous semilinear spaces (see \cite{DevillersUltrahomogeneous2000}). 

In Devillers' formulation, a semilinear space is a two-sorted structure with one sort for points and another for lines; it is homogeneous if the usual condition on the extensibility of local isomorphisms between finite configurations of points and lines is satisfied. 

Our semilinear space is defined in a primitive homogeneous simple 3-coloured graph. It is clear that we have two types of non-collinear points, corresponding to $S$- and $T$-edges in the coloured graph. If the diameter of the graph is 2, then we will see that $n=|R(c)\cap R(a)|$ and $m=|R(d)\cap R(a)|$ are not necessarily equal for $c\in S(a)$ and $d\in T(a)$, even though $ac$ and $ad$ are isomorphic as incidence structures. Any automorphism of the semilinear space extending the isomorphism $a\mapsto a, c\mapsto d$ would necessarily take $R(c)\cap R(a)$ to $R(d)\cap R(a)$, impossible. Thus, we cannot expect our linear spaces to be homogeneous in the sense of Devillers.

We will use the semilinear space to analyse the structure of \SU-rank 2 graphs.

Any two distinct vertices belong to at most one line and two distinct lines intersect in at most one vertex. Any given vertex belongs only to a finite number of lines, each of which is infinite. As a consequence:

\begin{observation}
	Suppose that $M$ is a semilinear 3-graph and $a\in M$. Then for all $d\in R^2(a)$ and $\ell$ a line through $a$, $|R(d)\cap\ell|<2$.
	\label{NoTriangles}
\end{observation}
\begin{proof}
	If we had two different points $b_1,b_2$ on $\ell\cap R(d)$, then as we have $R(b_1,b_2)$ we get that $b_1,b_2$ belong to the same line through $d$. But then $b_1,b_2\in\ell(a,b_1)\cap\ell(d,b_1)$, contradicting the fact, obvious from Definition \ref{DefSemilinear} that the intersection of two distinct lines in a semilinear space is either empty or a singleton.
\end{proof}

The situation in primitive semilinear 3-graphs is essentially different from that in primitive structures of \SU-rank 1. Compare our next observation with Proposition \ref{aclab}.
\begin{observation}
	Let	$M$ be a primitive homogeneous semilinear 3-graph. If the $R$-distance between $a$ and $b$ is 2, then $\acl(a,b)\neq\{a,b\}$.
\end{observation}
\begin{proof}
	The vertices $a$ and $b$ belong to a finite number of lines. Since the $R$-distance from $a$ to $b$ is 2, there exists at least one element $c\in R(a)$ such that $R(c,b)$ holds. There is at most one such $c$ in any line through $a$. These points are algebraic over $a,b$ and distinct from them.
\end{proof}

Observation \ref{NoTriangles} implies that the lines of the semilinear space interpreted in a semilinear 3-graph do not form triangles.

The sets $R(a),S(a),T(a)$ are homogeneous in the language $L$, so having the same type over $a$ is equivalent to being in the same orbit under $\aut(M/a)$. Therefore, we cannot have more than 2 nested $a$-invariant/definable proper nontrivial equivalence relations in any of them, as we would need more than 3 types of edges to distinguish them. For the same reason, the number of lines through $a$ that $R(c)$ meets for $c\in R^2(a)$ is invariant under $a$-automorphisms (which fix the set of lines through $a$) as $c$ varies in an $a$-orbit. 

\subsection{The nonexistence of primitive homogeneous 3-graphs of $R$-diameter 2 and \SU-rank 2}\label{subsectRDiam2}

We know by Lemma \ref{LemmaSemilinear} that finitely many lines are incident with any vertex $a\in M$ in a primitive simple homogeneous 3-graph of \SU-rank 2. Recall from subsection \ref{lines} that that two lines intersect in at most one point (by Observation \ref{NoTriangles} or by the definition of a semilinear space). The main question to ask is: if $a$ and $b$ are not $R$-related, how many lines containing $a$ can the $R$-neighbourhood of $b$ meet?

\begin{proposition}
	Let $M$ be a homogeneous primitive semilinear 3-graph with simple theory in which $S$ and $T$ are nonforking predicates, and suppose that $\diam_R(M)=2$. If for every $b\in R(a)$ and each line $\ell$ through $b$ other than $\ell(a,b)$ we have that $\ell\cap S(a)$ and $\ell\cap T(a)$ are both nonempty, then either $\ell\cap S(a)$ and $\ell\cap T(a)$ are both infinite, or one of them is of size 1 and the other is infinite.
\label{Intersections}
\end{proposition}
\begin{proof}
	Clearly, at least one of $\ell\cap S(a)$ and $\ell\cap T(a)$ is infinite. Suppose for a contradiction that $1<|\ell\cap S(a)|<\omega$. The formula $S(x,a)$ does not divide over $\varnothing$; therefore, for any indiscernible sequence $(a_i)_{i\in\omega}$ the set $\{S(x,a_i):i\in\omega\}$ is consistent by simplicity. In particular when $R(a_0,a_1)$ holds. Therefore, $S(a)$ embeds infinite $R$-cliques and by homogeneity every $R$-related pair in $S(a)$ is in one such clique. 
\[
\includegraphics[scale=0.8]{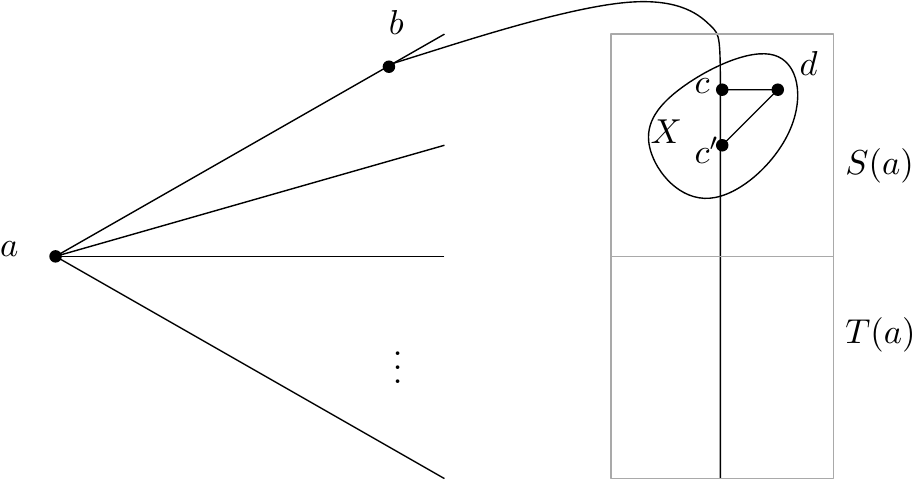}
\]

Take any $c,c'\in\ell\cap S(a)$, and let $X$ be an infinite $R$-clique in $S(a)$ containing them. Consider $d\in X\setminus\ell$; $X\subset\ell(c,d)$ and $b\notin\ell(c,d)$, and the same is true of $c'$. But both belong to $\ell(b,c)$. Therefore, there are two points which lie on two different lines, contradiction.
\end{proof}

Our next observation is crucial to proving that there are no homogeneous 3-graphs of rank 2 and diameter 2. We mentioned before that the incidence structure interpreted in $M$ by the lines and vertices is close to being a generalised quadrangle. Recall that a generalised quadrangle (see \cite{thas2004symmetry}) is an incidence structure of points and lines with possibly infinite parameters $s$ and $t$ satisfying:
\begin{enumerate}
	\item{any two points lie on at most one line,}
	\item{any line is incident with exactly $s+1$ points, and any point with exactly $t+1$ lines, and}
	\item{if $x$ is a point not incident with a line $L$, then there is a unique point incident with $L$ and collinear with $x$.}
\end{enumerate}

In \cite{macpherson1991interpreting}, Macpherson proves:
\begin{theorem}
	Let $M$ be a homogenizable structure. Then it is not possible to interpret in $M$ any of the following:
\begin{enumerate}
	\item{an infinite group,}
	\item{an infinite projective plane,}
	\item{an infinite generalised quadrangle, or}
	\item{an infinite Boolean algebra.}
\end{enumerate}
\label{DugaldQuadrangle}
\end{theorem}

\begin{observation}
	If $M$ is homogeneous primitive semilinear 3-graph and $\diam_R(M)=2$, then it is not the case that for all $b\in R^2(a)$ the set $R(b)$ intersects all lines containing $a$.
	\label{Quadrangle}
\end{observation}

\begin{proof}
	In this case, the incidence structure interpreted in $M$ with lines of the form $\ell(x,y)$  and vertices as points is a generalised quadrangle with infinite lines and as many lines through a point as $R$-classes in $R(a)$, contradicting Theorem \ref{DugaldQuadrangle}.
\end{proof}

The following observation will help us find different points $c,c'$ in $S(a)$ or $T(a)$ such that $R(c)$ and $R(c')$ meet the same lines through $a$. Recall that given a subset $B$ of $M$, the group of all automorphisms of $M$ fixing $B$ setwise is denoted by $\aut(M)_{\{B\}}$.

\begin{observation}
	Let $M$ be a primitive homogeneous semilinear 3-graph of $R$-diameter 2. Let $X$ be a set of lines incident with $a$. Then $\aut(M/a)_{\{\bigcup X\}}$ acts transitively on $\ell\setminus\{a\}$ for all $\ell\in X$.
	\label{TransOnLines}
\end{observation}
\begin{proof}
	This is trivial if only one of $S,T$ is realised in the union of two lines through $a$, so assume that all types of edges are realised in the union of two lines through $a$. 

Note that at least one of $RSS, RTT$ is realised in $R(a)$. Assume without loss of generality that $RSS$ is realised in $R(a)$. Let $b,b'$ be elements of $R(a)$ satisfying $R(b,b')$. Enumerate the lines in $X$ as $\ell_1,\ldots,\ell_k$, and assume $b,b'\in\ell_k\setminus\{a\}$. 
\[
\includegraphics[scale=0.7]{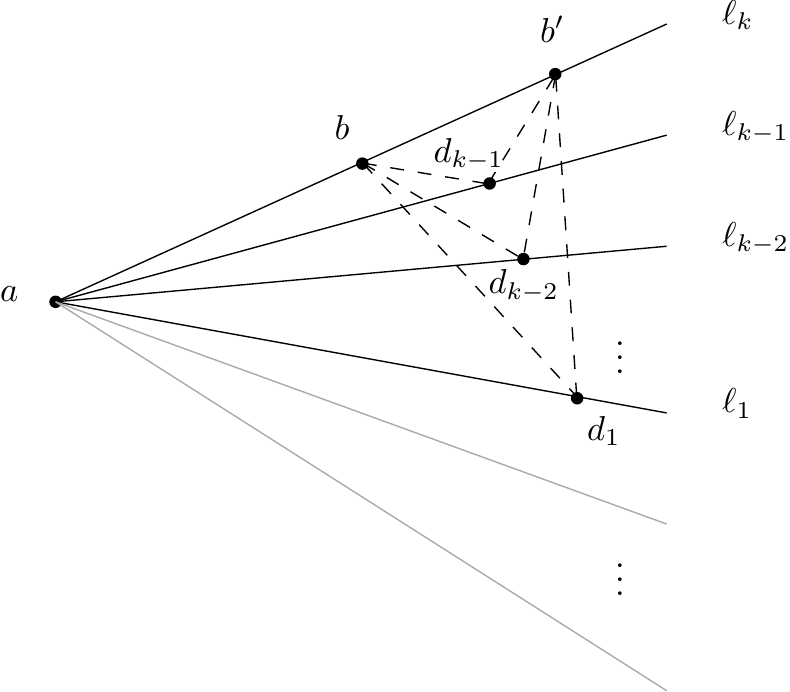}
\]
We can find elements $d_1\in\ell_1,\ldots,d_{k-1}\in\ell_{k-1}$ such that $S(b,d_i)\wedge S(b',d_i)$ for $i\in\{1,\ldots,k-1\}$, so $\tp(b/a,d_1,\ldots,d_{k-1})=\tp(b'/a,d_1,\ldots,d_{k-1})$. By homogeneity, there is an automorphism of $M$ fixing $a,d_1,\ldots,d_{k-1}$ (and therefore fixing $\bigcup X$ setwise) taking $b$ to $b'$.
\end{proof}

If only one of $S,T$ is realised in the union of two lines through $a$, then each pair of $R$-classes in $R(a)$ is isomorphic to a complete bipartite graph (the parts of the partition are $R$-cliques and the edges are of colour $S$ or $T$), so we have two orbits of pairs of lines through $a$. 

\begin{observation}
	Let $M$ be a primitive homogeneous simple semilinear 3-graph of $R$-diameter 2. If in $R(a)$ all relations are realised in the structure induced on a pair of lines through $a$, and there are $m$ lines through $a$, then there is only one orbit of $k$-sets of lines over $a$, for all $k\leq m$.
	\label{OneOrbit}
\end{observation}
\begin{proof}
	There are two cases, depending on whether we can find witnesses to the instability of $S,T$ within $R(a)$.

	If $R(a)$ is a stable structure, then it is isomorphic to $K_m^S\times K_\omega^R$ or to $K_m^T\times K_\omega^R$, by Lachlan's Theorem \ref{Lachlan3graphs}, Observation \ref{NoInfiniteCliques}, and the hypothesis that all relations are realised in the structure induced on a pair of incident lines. In any of these structures there are monochromatic transversal cliques of size $m$, so the observation follows by invariance and homogeneity.

	If we can find witnesses to the instability of $S,T$ within $R(a)$, then $R(a)$ is isomorphic to a simple unstable homogeneous 3-graph in which $R$ defines a finite equivalence relation. By Proposition \ref{PropImprimitive3Graphs}, we can embed transversal monochromatic cliques, and again the observation follows by invariance and homogeneity.
\end{proof}

Notice that if for some element $b\in R(a)$ and some line $\ell$ through $b$ different from $\ell(a,b)$ the sets $\ell\cap S(a)$ and $\ell\cap T(a)$ are both nonempty, then by homogeneity we can transitively permute the lines through $b$ whilst fixing $ab$, and therefore all lines through any $b\in R(a)$, except $\ell(a,b)$, meet both orbits over $a$ in $R^2(a)$. Furthermore, the size of the intersections does not change and is either 1 or infinite, with at least one of them infinite. To put it differently, if one line through $b$ (not $\ell(a,b)$) is almost entirely contained (the point $b$ is assumed to be in $R(a)$) in $S(a)$, then each line is almost entirely contained in one orbit. Now we prove that all lines that meet $R^2(a)$ meet both $S(a)$ and $T(a)$.

We will need the following well-known fact from permutation group theory (see, for example, 2.16 in \cite{cameron1990oligomorphic}) to strengthen Observation \ref{TransOnLines} :

\begin{theorem}
Let $G$ be a permutation group on a countable set $\Omega$, and let $A,B$ be finite subsets of $\Omega$. If $G$ has no finite orbits on $\Omega$, then there exists $g\in G$ with $Ag\cap B=\varnothing$.
\label{ThmNeumann}
\end{theorem}

\begin{proposition}
	Suppose that $M$ is a primitive simple homogeneous semilinear 3-graph with $m$ lines through each point, and let $X=\{\ell_i:1\leq i\leq k\}$ be a set of lines through $a\in M$. Then for any transversal $A$ to the $k$ lines in $X$, there exists a transversal $B$ to $X$ such that $B\cong A$ and $B\cap A=\varnothing$.
\label{PropLotsOfTrans}
\end{proposition}
\begin{proof}
This is a direct consecuence of Theorem \ref{ThmNeumann} and Observation \ref{TransOnLines}.
\end{proof}

\begin{proposition}
	Let $M$ be a simple homogeneous primitive semilinear 3-graph of $R$-diameter 2, in which all predicates are realised in the structure induced on a pair of incident lines, and $a\in M$. Then for all $b\in R(a)$, each line $\ell\neq\ell(a,b)$ through $b$ meets both $S(a)$ and $T(a)$.
	\label{Intersections2}
\end{proposition}
\begin{proof}
	First note that it is not possible to have $R(b)\cap S(a)=\varnothing$ or $R(b)\cap T(a)=\varnothing$. To see this, suppose for a contradiction that $R(b)\cap S(a)=\varnothing$; moving $b$ by homogeneity within $R(a)$, it follows that $R(b')\cap S(a)=\varnothing$ for all $b'\in R(a)$, contradicting the assumption that vertices in $S(a)$ are at $R$-distance 2 from $a$. Similarly, $R(b)\cap T(a)\neq\varnothing$. Therefore, this proposition can only fail if we have at least 3 lines through $a$. 

Suppose for a contradiction that there are $m\geq3$ lines incident with $a$ and for all $b\in R(a)$ and $\ell\neq\ell(a,b)$ through $b$, $\ell\setminus\{b\}\subset S(a)$ or $\ell\setminus\{b\}\subset T(a)$. By Observation \ref{Quadrangle}, we may assume that $k=|R(c)\cap R(a)|<m$ for all $c\in S(a)$. We define two binary relations and a binary function on $S(a)$: for $c,c'\in S(a)$, $E(c,c')$ holds if $R(c)$ and $R(c')$ meet the same lines through $a$, and $C(c,c')$ holds if there exists $b\in R(a)$ such that $b,c,c'$ are collinear. Given two elements $x,y\in S(a)$, let $\#(x,y)$ denote the number of $R$-classes in $R(a)$ that $R(x)$ and $R(y)$ meet in common, that is $\#(x,y)=|\{z\in R(x)\cap R(a):\exists w(w\in R(y)\cap R(a)\wedge( R(w,z)\vee w=z))\}|$. 

\begin{case}
If $k=1$, then there are at least four types of unordered pairs of vertices in $S(a)$. We prove this assertion as follows: let $b_c$ denote the unique element in $R(c)\cap R(a)$ for $c\in S(a)$. The relations $\hat P(c,c')$ that hold if $P(b_c, b_{c'})$ is true ($P\in \{R,S,T\}$) are invariant and imply that $b_c,c,c'$ are not collinear. It follows from the assumption that for all $\ell\neq\ell(a,b)$ through $b\in R(a)$ the set $\ell\setminus\{b\}$ is contained in $S(a)$ or in $T(a)$ and the first paragraph of this proof that $C$ is also realised in $S(a)$. That gives us too many types of distinct unordered pairs of elements in $S(a)$.
\end{case}
\begin{case}
If $2\leq k<m$, then we have two subcases:
\begin{subcase}
If $m-k\geq2$, then we can find at least five types of unordered pairs of elements in $S(a)$. The proof is as follows: Observation \ref{OneOrbit} implies that $E$ is a nontrivial proper equivalence relation on $R(a)$ and that it has $m\choose k$ classes. Now we claim that there are at least four types of $E$-inequivalent elements in $S(a)$. By Observation \ref{OneOrbit} there exist pairs of elements $c,c'\in S(a)$ with $\neg E(c,c')\wedge\#(c,c')=k-1$. Using Proposition \ref{PropLotsOfTrans}, we can find pairs which additionally satisfy $R(c)\cap R(c')\cap R(a)\neq\varnothing$ and $R(c)\cap R(c')\cap R(a)=\varnothing$. We can follow the same argument in the case $\#(c,c')=k-2$ to find two more types of unordered pairs of distinct elements from $S(a)$, giving a total of at least five.
\end{subcase}
\begin{subcase}
Suppose then that $m-k=1$, so $E$ has $m$ equivalence classes. There is at least one line through $b$ almost entirely contained in $T(a)$, so we are left with at most $m-2$ lines through $b$ distinct from $\ell(a,b)$ which may meet $S(a)$. 
\begin{claim}
	$R(b)$ meets $m-1$ $E$-classes in $S(a)$.
	\label{ClaimMeetsAllClasses}
\end{claim}
\begin{proof}
	We know that $R(b)\cap S(a)\neq\varnothing$. Let $c\in R(b)\cap S(a)$. By hypothesis, $|R(c)\cap R(a)|=m-1$. Let $X$ denote $R(c)\cap R(a)$. By Observation \ref{OneOrbit}, we can find $a$-translates of $X_i$, $i\leq m-1$, to any of the $m-1$ sets of $m-1$ lines through $a$ that include the $R$-class to which $b$ belongs. And by the transitivity of $\aut(M/a)$ on $R(a)$ we can find translates $Y_i$ in those sets of lines such that $b\in Y_i$. By homogeneity, each of the automorphisms taking $X$ to $Y_i$ moves $c$ to a new $E$-class. 

Clearly, $R(b)$ does not meet the $E$-class of elements whose $R$-neighbourhoods meet all the lines in $R(a)$ except $\ell(a,b)$.
\end{proof}

As the lines are infinite and $E$ has only finitely many classes, for each line $\ell$ through $b$ that meets $S(a)$ there is at least one $E$-class that contains infinitely many elements of $\ell$. Since we have at most $m-2$ lines through $b$ that meet $S(a)$ and $R(b)$ meets $m-1$ $E$-classes, there is at least one line that meets more than one $E$-class. Note that if a line $\ell$ meets more than one $E$-class, then the intersection of $\ell$ with each of the $E$-classes it meets is infinite, by homogeneity as elements in each class have the same type over $ab$ and at least one of the intersections of $\ell$ with an $E$-class is infinite. 

Again by homogeneity (we can permute the lines over $b$ that meet $S(a)$ whilst fixing $ab$), each line through $b$ that meets $S(a)$ meets more than one class. 

As $k\geq 2$, there exist $b_1,b_2\in R(a)$ such that for some $c\in S(a)$ we have $\ell(b_i,c)\setminus\{b_i\}\subset S(a)$ ($i=1,2$). Take $c'\in\ell(b_1,c)\cap S(a)$ and $c''\in\ell(b_2,c)\cap S(a)$, both distinct from $c$ and $E$-equivalent to $c$. Such elements exist because the intersections of lines through $b_i$ with the $E$-class of $c$ are infinite, by homogeneity and the fact that at least one of the intersections is infinite, so we have $E(c,c')\wedge R(c,c')$. Also, $c'$ and $c''$ are not $R$-related (because the lines of the semilinear space do not form triangles, cf. Observation \ref{NoTriangles}), but are $E$-equivalent since we have $E(c,c')$ and $E(c,c'')$, so at least one of $E(c',c'')\wedge S(c'c,'')$ and $E(c',c'')\wedge T(c',c'')$ is realised. This gives us at least two types of $E$-equivalent pairs.

Now we will show that there are at least two types of $E$-inequivalent pairs. By Proposition \ref{PropLotsOfTrans}, we can find pairs of $E$-inequivalent elements with no common $R$-neighbours in $R(a)$ and also pairs of $E$-inequivalent elements with common $R$-neighbours in $R(a)$. Again, we get at least four types of unordered pairs of distinct elements from $S(a)$.
\end{subcase}
\end{case}
\end{proof}
\setcounter{case}{0}
\setcounter{subcase}{0}
\begin{proposition}
	Let $M$ be a primitive homogeneous semilinear 3-graph with simple theory and $R$-diameter 2, and assume that $R$ is a forking relation and $S,T$ are nonforking. Then each element is incident with at least three lines.
\label{atleast3}
\end{proposition}
\begin{proof}
Note first that if in any pair of lines through $a$ only two of the predicates in the language are realised, then we get the result automatically because $R,S,T$ are realised in $R(a)$ by the diameter 2 hypothesis. So we may assume that in the structure induced by $M$ on a pair of lines through $a$ all predicates are realised.

By Observation \ref{NotRComplete}, each vertex belongs to at least two lines.

If $R(a)$ has exactly two imprimitivity blocks, then by homogeneity for any $b\in R(a)$ the set $R(b)$ consists of two infinite $R$-cliques as well, one of which is $\ell(a,b)\setminus\{b\}$. Therefore, $R(b)\cap R^2(a)$ is an infinite $R$-clique, and by Proposition \ref{Intersections2}, $R(b)\cap R^2(a)$ meets both $S(a)$ and $T(a)$, as by the diameter 2 hypothesis both $S$ and $T$ are realised in $R(a)$.

\begin{claim}
	For all $b\in R(a)$, $\ell^b\cap S(a)$ and $\ell^b\cap T(a)$ are infinite.
\end{claim}
\begin{proof}
Suppose that each vertex is incident with two lines. Then for all $b\in R(a)$, there is a unique line through $b$ that meets $R^2(a)$; let $\ell^b$ denote that line, for each $b\in R(a)$.

Proposition \ref{Intersections} tells us that either $\ell^b\cap S(a)$ and $\ell^b\cap T(a)$ are both infinite, or one of them is of size 1 and the other is infinite. As this line is uniquely determined for each $b\in R(a)$, if we had, say $|\ell^b\cap S(a)|=1$ and $|R(c)\cap R(a)|=1$ for all $c\in S(a)$, then this would establish a definable bijection between $R(a)$ and $S(a)$. This is impossible as the rank of $R(a)$ is lower than that of $S(a)$. 

Therefore, in orded to establish the claim, we need to eliminate the case where $|\ell^b\cap S(a)|=1$ and $|R(c)\cap R(a)|=2$. 

By Observation \ref{Quadrangle}, if these conditions are satisfied then $|R(d)\cap R(a)|=1$ for all $d\in T(a)$. Given any $c\in S(a)$, the set $R(c)$ consists of two infinite $R$-cliques by homogeneity; one vertex from each of these two cliques, belongs to $R(a)$. 

Therefore, for any $c\in S(a)$, all relations in the language are realised in $R(c)\cap T(a)$. Define $Q(d,d')$ on $T(a)$ to hold if there exists $c\in S(a)$ such that $R(d,c)\wedge R(d',c)$ (see figure below).
\[
\includegraphics[scale=0.8]{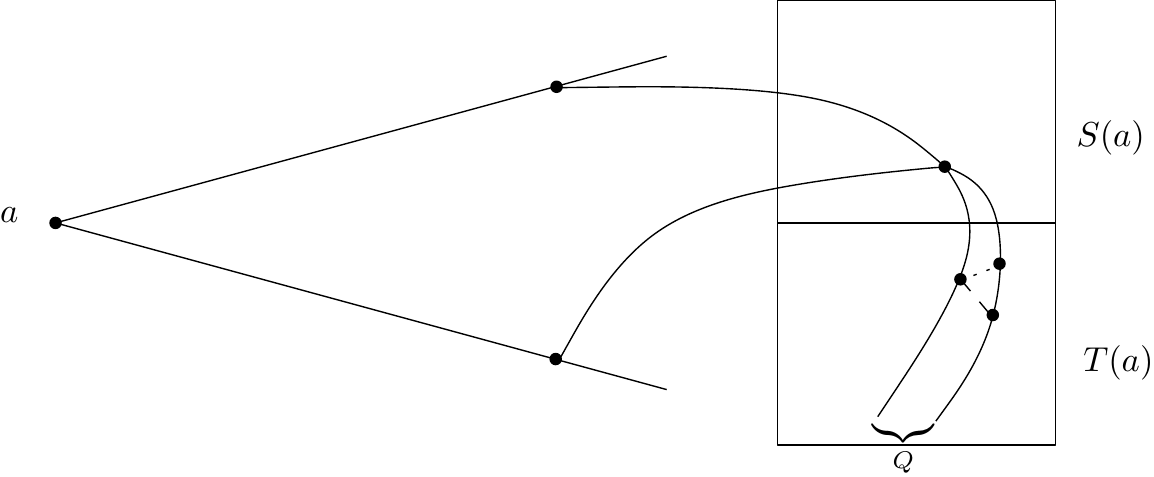}
\]

We claim that $Q\wedge R$, $Q\wedge S$, $Q\wedge T$ are realised in $T(a)$. The reason is that both lines through $c$ are almost entirely contained in $T(a)$: $c$ and the two vertices in $R(c)\cap R(a)$ are the only elements of $R(c)$ not in $T(a)$, since any other element of $R(c)\cap S(a)$ would be forced to be an element of $R(b_1)$ or of $R(b_2)$, contradicting $R(b)\cap S(a)=1$ for all $b\in R(a)$. Our claim follows from the transitivity of $\aut(M/c)$ on $R(c)$ and Theorem \ref{ThmNeumann}. 

Now, since $S$ does not divide over $\varnothing$ in $M$ we must have the triangle $SSR$ in $\age(M)$ (otherwise, $S$ would divide, as witnessed by an $\varnothing$-indiscernible sequence $(e_i)_{i\in\omega}$ with $R(e_0,e_1)$). Notice that we have an additional $a$-definable equivalence relation $F$ on $T(a)$ with two classes, $F(d,d')$ holds if $R(d)$ and $R(d')$ meet the same line through $a$. If $Q$ and $F$ were satisfied simultaneously by a pair from $T(a)$ then $F\wedge R$, $F\wedge S$, $F\wedge T$ (realised because $R,S,T$ are realised in the union of any two incident lines), and $\neg F$ already give us too many relations on $T(a)$. And if they are not simultaneously realised by any pair, then any $F$-equivalent pair is $Q$-inequivalent, so this together with the three relations from the preceding paragraph give us four types of unordered pairs of distinct elements from $T(a)$.
\end{proof}

By Observation \ref{Quadrangle}, we may also assume that $|R(c)\cap R(a)|=1$ for all $c\in T(a)$.

Consider the relation $W(x, y)$ on $T(a)$ that holds if there exists a $b\in R(a)$ such that $R(b,x)\wedge R(y,b)$. This is clearly a symmetric and reflexive relation, and if $W(x,y)$ and $W(y,z)$, then there exist $b,b'\in R(a)$ such that $R(x,b)\wedge R(y,b)$ and $R(y,b')\wedge R(z,b')$. The hypothesis that $|R(c)\cap R(a)|=1$ for all $c\in R^2(a)$ implies $b=b'$, as they are both $R$-related to $y$ and in $R(a)$. Therefore, $x,y,z$ are all collinear with $b$ and $W(x,z)$. Given a vertex $c\in T(a)$ denote by $b_c$ the unique element of $R(c)\cap R(a)$, and define $\hat P(c,c')$ on $T(a)$ if $P(b_c,b_{c'})$ holds for $P\in\{R,S,T\}$. This gives us at least four types of unordered pairs of distinct elements in $T(a)$: $W$-equivalent and three types of $W$-inequivalent pairs (corresponding to $\hat R,\hat S,\hat T$).
\end{proof}

\begin{observation}
	If $M$ is a primitive homogeneous simple semilinear 3-graph with $\diam_R(M)=2$ in which any point $a$ is incident with at least three lines, then it is not possible for all $c\in R^2(a)$ to satisfy $|R(c)\cap R(a)|=1$.
	\label{EachTouchesOne}
\end{observation}
\begin{proof}
	In this case, the sets $R(b)\cap R^2(a)$, $b\in R(a)$, partition the set of maximal rank $R^2(a)$ into infinitely many infinite parts, each consisting of at least 2 infinite $R$-cliques. By Proposition \ref{Intersections}, at least one of $S(a)$ and $T(a)$ is partitioned into infinitely many infinite $R$-cliques by the family of sets $\ell\setminus\{b\}$, where $b\in R(a)$ and $\ell$ is a line through $b$ not containing $a$. We may assume it is $S(a)$.

	Define the relation $Q(c,c')$ on $S(a)$ to hold if there exists $b\in R(a)$ such that $R(b,c)\wedge R(b,c')$ holds. We claim that $Q$ is an equivalence relation. It is clearly symmetric and reflexive. Now suppose $Q(x,y)\wedge Q(y,z)$. Then there exist $b,b'\in R(a)$ such that $R(b,x)\wedge R(b,y)$ and $R(b',y)\wedge R(b',z)$, so $b=b'$ since $|R(y)\cap R(a)|=1$. Therefore, $Q(x,z)$ holds.

	A $Q$-equivalence class consists of a finite number $m>1$ of $R$-cliques, because we assume that at least three lines are incident with $a$. Define the binary relations $\hat P(c,c')$ to hold if $\neg Q(c,c')$ and $P(b,b')$, where $\{b\}=R(c)\cap R(a)$, $\{b'\}=R(c')\cap R(a)$ and $P\in\{R,S,T\}$. This gives 3 types of $Q$-inequivalent pairs, plus at least two more types of $Q$-equivalent pairs (collinear and not collinear), so we have too many types of unordered pairs of elements from $S(a)$.
\end{proof}

By Observation \ref{NoInfiniteCliques}, not all $R$-free structures can be embedded into $R(a)$. If $R(a)$ is a stable 3-graph, then it must be of one of the forms 6-11 in Theorem \ref{Lachlan3graphs}, as all the others are finite. Observation \ref{NoInfiniteCliques} implies that only one of $m,n,p$ is $\omega$ (and the corresponding superindex is $R$).

The sets of relations realised with endpoints in different classes of an equivalence relation partition the set of types of pairs of classes in a homogeneous binary structure. In our case, there can be no more than 2 types of pairs of $R$-classes in $R(a)$. This is implicitly used in the proof of our next result:

\begin{proposition}
	There are no primitive simple homogeneous 3-graphs of $R$-diameter 2 such that all relations are realised in the union of any two maximal $R$-cliques in $R(a)$.
\label{MataCasiTodos}
\end{proposition}
\begin{proof}
	By Observations \ref{Quadrangle} and \ref{EachTouchesOne}, we have two cases to analyse:
\begin{case}
	For some $c\in R^2(a)$, $|R(c)\cap R(a)|=1$. By homogeneity, this is true for all the elements of the orbit of $c$ under the action or $\aut(M/a)$. Without loss of generality, assume $S(a,c)$. We can define $E(x,y)$ on $S(a)$ if $R(x)$ and $R(y)$ meet the same line through $a$, and refine this equivalence relation with $E'(x,y)$ if they meet the same line at the same point. These two are equivalence relations, and $E'(x,y)\rightarrow E(x,y)$. For $E$-inequivalent pairs, since both $S$ and $T$ are realised in $R(a)$, we can define $\hat S(x,y)$ and $\hat T(x,y)$ if $S$ (respectively, $T$) holds between the elements of the intersections $R(x)\cap R(a)$ and $R(y)\cap R(a)$. Notice that both $\hat S$ and $\hat T$ are realised, as any element in $R(a)$ has a neighbour in $S(a)$. We have too many 2-types of distinct elements over $a$, since $E'(x,y)\wedge x\neq y, E(x,y)\wedge \neg E'(x,y), \hat S(x,y)\wedge\neg E(x,y), \hat T(x,y)\wedge\neg E(x,y)$ are all realised.\label{MCT1}
\end{case}
\begin{case}
	For no element $b$ of $R^2(a)$ does $|R(b)\cap R(a)|=1$ hold. Then, without loss of generality, the elements of $S(a)$ satisfy $|R(b)\cap R(a)|=k$, where $1<k<m$. Define $E(c,c')$ on $S(a)$ if $R(c)$ and $R(c')$ meet the same lines through $a$. There are three subcases to analyse:
\begin{subcase}
	If $m-k\geq3$, then define $P_i(c,c')$ on $S(a)$ for $0\leq i\leq\min\{k,m-k\}$ to hold if $R(c)\cup R(c')$ meet a total of $k+i$ lines through $a$. The $P_i$ are invariant under $\aut(M/a)$ and mutually exclusive; therefore all cases with $\min\{k,m-k\}\geq3$ are impossible, as we would get at least four types of pairs of distinct elements in $S(a)$. This leaves us with only one more possible case, namely $m-k\geq3, k=2$, since the case $m-k\geq3, k=1$ is covered in Case \ref{MCT1}

	Suppose then that $m-k\geq3$ and $k=2$. We claim that there are two types of pairs satisfying $P_1$. Let $\{b,b'\}=R(c)\cap R(a)$ for some $c\in S(a)$, and take any line $\ell$ through $a$ not including $b$ or $b'$. By homogeneity, there exists a $b''\in\ell$ satisfying the same relation with $b'$ as $b$. Therefore, there exists $c'\in S(a)$ satisfying $P_1(c,c')$ and the relation $Q(c,c')$ defined by $\exists x(R(a,x)\wedge R(c,x)\wedge R(c',x))$. Using Proposition \ref{Intersections2}, we can find pairs $d,d'$ in $S(a)$ satisfying $P_1(d,d')$ and $R(d)\cap R(d')\cap R(a)=\varnothing$. Therefore, we have at least four types of pairs of distinct elements from $S(a)$, as the relations $E$, $P_1\wedge Q$, $P_1\wedge\neg Q$, $P_2$ are all realised.
\label{CasedefsP}
\end{subcase}
\begin{subcase}
	Suppose $m-k=1$. By Proposition \ref{TransOnLines}, there exist unordered pairs of distinct elements satisfying $E$ in $S(a)$, and $P_1$ (defined as in Case II\ref{CasedefsP}) is realised by homogeneity and Observation \ref{OneOrbit}. 

Notice that there are two types of pairs satisfying $P_1(c,c')$, namely those with $R(c)\cap R(c')\cap R(a)=\varnothing$, and those with $R(c)\cap R(c')\cap R(a)\neq\varnothing$. Both are realised by Proposition \ref{PropLotsOfTrans}.

	This leaves us with two possibilities: for distinct $c,c'\in S(a)$, either $E(c,c')$ implies $R(c)\cap R(c')\cap R(a)=\varnothing$ (this can happen if the structure on any pair of lines through $a$ is that of a perfect matching and $R(c)$ picks a transversal clique of the matching colour), or we can have $E(c,c')\wedge R(c)\cap R(c')\cap R(a)\neq\varnothing$. In the latter case, we have found four types of pairs of unordered distinct elements from $S(a)$.

	Therefore, assume that $E(c,c')$ implies $R(c)\cap R(c')\cap R(a)=\varnothing$ for all $c\neq c'$ in $S(a)$. We claim that this can only happen in the situation described before, namely if the structure on two lines is that of a matching and for all pairs $b,b'\in R(c)\cap R(a)$, the edge $bb'$ is of the colour of the matching predicate, say $T$. This claim follows from the argument of Proposition \ref{TransOnLines}: if for some edge $bb'$ in $R(c)\cap R(a)$ we were able to find some $b''$ collinear with $b'$ such that $bb''$ and $bb'$ are of colour $T$, then by homogeneity we could find a $c'$ $E$-equivalent to $c$ with $b\in R(c)\cap R(c')\cap R(a)$. 

	It follows that in the situation we are considering $T$ is an algebraic predicate in $R(a)$ and the set of $K_{m-1}^T$ in $R(a)$ is in definable bijection with $S(a)$ by the function taking a $T$-clique $\bar c$ to the unique element of $\bigcap\{R(c):c\in\bar c\}\cap S(a)$. This is impossible, since the rank of $S(a)$ is greater than that of the set of $T$-cliques in $R(a)$, as $T$ is algebraic.
\end{subcase}
\begin{subcase}
	If $m-k=2$, then the relations $E, P_1, P_2$ defined in Case \ref{CasedefsP} are realised in $S(a)$. As in Case \ref{CasedefsP}, there are two types of pairs $c,c'$ satisfying $P_1$: some with $R(c)\cap R(c')\cap R(a)\neq\varnothing$ and some with $R(c)\cap R(c')\cap R(a)=\varnothing$, by the same argument as in Case {CasedefsP}.
\end{subcase}
\end{case}
\end{proof}
\setcounter{case}{0}
\setcounter{subcase}{0}

Proposition \ref{MataCasiTodos} eliminates all cases where $R(a)$ is unstable, as in this case for some infinite $R$-cliques $A,B$ in $R(a)$ the induced structure is isomorphic to the Random Bipartite Graph. But Proposition \ref{MataCasiTodos} also covers some stable cases (for example, if $S$ or $T$ is a perfect matching on the union of the two $R$-cliques). The only cases that remain are those in which $R(a)$ is stable and the induced structure on any pair of $R$-cliques in $R(a)$ is isomorphic to a complete bipartite graph, that is, those cases in which for all pairs of lines $\ell_1,\ell_2$ through $a$ and all $(b_1,b_2),(c_1,c_2)\in\ell_1\times\ell_2$, $\tp(b_1b_2)=\tp(c_1c_2)$. In all of these cases, $R(a)$ is stable. 

\begin{proposition}
	Let $M$ be a homogeneous primitive semilinear 3-graph of $R$-diameter 2 with finitely many lines through each point. If all types of pairs are realised in $R(a)$, but not in any pair of lines through $a$, then it is not possible for any $c\in R^2(a)$ to satisfy $|R(c)\cap R(a)|>3$.
	\label{SmallIntersection}
\end{proposition}
\begin{proof}
	\begin{claim} 
		If $\tp(ac)=\tp(ac')$, then $\tp(R(c)\cap R(a))=\tp(R(c')\cap R(a))$.
		\label{Claim1}
	\end{claim}
		\begin{proof}
			By homogeneity, there exists an automorphism $\sigma\in\aut(M/a)$ taking $c\mapsto c'$; this automorphism takes $R(c)\cap R(a)$ to $R(c')\cap R(a)$. 
		\end{proof}
	\begin{claim}
		Under the hypotheses of Proposition \ref{SmallIntersection}, the isomorphism type of $R(c)\cap R(a)$ for any $c\in R^2(a)$ depends only on the set of lines through $a$ that $R(c)$ meets.
		\label{Claim2}
	\end{claim}
		\begin{proof}
			By Observation \ref{NoTriangles}, the set $R(c)\cap R(a)$ is transversal to a set of $k$ lines through $a$, and by the hypotheses of Proposition \ref{SmallIntersection}, all transversals to the same set of $k$ lines are isomorphic.
		\end{proof}

Now suppose that for some $c\in S(a)$ we have $|R(c)\cap R(a)|>3$. By Claim \ref{Claim1}, the intersections of the $R$-neighbourhood of any two elements of $S(a)$ with $R(a)$ are isomorphic; let $E$ be the (not necessarily proper) equivalence relation on $S(a)$ that holds for elements that meet the same set of lines through $a$. Claim \ref{Claim2} says that if $A=R(c)\cap R(a)$ for some $c\in S(a)$ and we take any other set $B$ transversal to the same set of $k>3$ lines then there exists an automorphism taking $A$ to $B$ over $a$ that moves $c$ to an $E$-equivalent element of $S(a)$. Therefore, the $a$-invariant relations $P_i(c,c')$ holding if $E(c,c')\wedge |R(c)\cap R(a)|=i$ for $i\in\{0,\ldots,k-1\}$ are all realised. As $k\geq4$, this gives us too many invariant relations on pairs over $a$. This completes the proof of Proposition \ref{SmallIntersection}.
\end{proof}

\begin{lemma}
	There are no homogeneous primitive 3-graphs of \SU-rank 2 and $R$-diameter 2.
	\label{NoDiam2}
\end{lemma}

\begin{proof}
	We know by Proposition \ref{atleast3} that the number $m$ of lines through $a$ is greater than or equal to 3, and that all types of pairs are realised in $R(a)$, but not in any pair of lines through $a$ (Proposition \ref{MataCasiTodos}). By Proposition \ref{SmallIntersection}, for all $c\in R^2(a)$ we have $|R(c)\cap R(a)|\leq 3$. Assume that $k=\max\{|R(c)\cap R(a)|,|R(d)\cap R(a)|\}$, where $c\in S(a)$ and $d\in T(a)$.

\begin{case}
First we prove that $k=3$ is impossible. Let $E(c,c')$ be the equivalence relation on $S(a)$ that holds if $R(c)$ and $R(c')$ meet the same lines through $a$. The key observation in this case is that the graph induced on $R(c)\cap R(a)$ is a finite homogeneous graph of size 3, so it must be a monochromatic triangle (see also Gardiner's classification \cite{gardiner1976homogeneous} of finite homogeneous graphs).

We start by arguing that $E$ is always a proper equivalence relation on $S(a)$ if $k=3$. By the preceding paragraph, $R(c)\cap R(a)$ is a complete graph in $S$ or $T$. If $E$ were universal in $S(a)$, then it follows either that there are only three lines through $a$ (impossible as in that case one of the predicates would not be realised in $R(a)$), or, assuming without loss that $R(c)\cap R(a)$ is isomorphic to $K_3^S$, that $R(a)$ is isomorphic to $K_m^T[K_n^S[K_\omega^R]]$. In the latter case, we must have $n=3$ because otherwise we could move by homogeneity the $K_3^S$ corresponding to $R(c)\cap R(a)$ to another set of 3 lines in the same $R\vee S$-class and find $E$-inequivalent elements. Finally, if $m>1$ then again we have that $E$ is a proper equivalence relation, depending on which $R\vee S$-class in $R(a)$ the set $R(c)$ meets. We reach a contradiction in any case; $E$ is a proper equivalence relation on $R(a)$.

Suppose for a contradiction that for $c\in S(a)$ we have $|R(c)\cap R(a)|=3$. Since $E$ is a proper equivalence relation, we have at least 4 invariant and exclusive relations on $S(a)$: $E$-inequivalent and three ways to be $E$-equivalent, as we can define $I_i(c,c')$ on $S(a)$ to hold if $E(c,c')$ and $|R(c)\cap R(c')\cap R(a)|=i$ for $i\in\{0,1,2\}$ (these relations are realised because the intersection of the $R$-neighbourhoods of $c$ and $a$ is a complete monochromatic graph, so any two transversals to the lines that $R(c)$ meets are isomorphic); this already gives us too many invariant relations on pairs from $S(a)$.
\end{case}
\begin{case}
Assume $\max\{|R(c)\cap R(a)|,|R(d)\cap R(a)|\}\leq 2$ ($c\in S(a), d\in T(a)$). By Observation \ref{EachTouchesOne} and Proposition \ref{atleast3}, it must be equal to 2. Suppose that the maximum is reached in $S(a)$. The equivalence relation $E(c,c')$ that holds on $S(a)$ if $R(c)$ and $R(c')$ meet the same lines through $a$ is proper: since $m\geq3$ and $k=2$, we can use homogeneity to move an element of $R(c)\cap R(a)$ to any line not containing any elements of $R(c)\cap R(a)$; this automorphism moves $c$ to an element of $S(a)$ that is not $E$-equivalent with $c$. Therefore we have at least four types of pairs on $S(a)$: two satisfying $E(c,c')$ (one with $R(c)\cap R(c')\cap R(a)$ empty, the other with $R(c)\cap R(c')\cap R(a)$ nonempty), and, similarly, two with $\neg E(c,c')$.
\end{case}
	We have exhausted the list of possible cases. The conclusion follows.
\end{proof}

\subsection{The nonexistence of primitive homogeneous 3-graphs of $R$-diameter 3 and \SU-rank 2}\label{subsectRDiam3}

By homogeneity, if the $R$-diameter of the graph is 3, then, since $R$-distance is preserved under automorphisms, if there are $a,b,c$ such that $S(a,c)\wedge R(a,b)\wedge R(b,c)$, then all pairs $c,c'$ with $S(c,c')$ consist of vertices at $R$-distance 2; and similarly $T(a)$ would be the set of vertices at $R$-distance 3 from $a$. From this point on, we will follow the conventions $S(a)=R^2(a)$ and $T(a)=R^3(a)$. 

The situation in diameter 3 is considerably simpler than in diameter 2, as the sets $S(a)$ and $T(a)$ are more clearly separated. The first thing to notice is that if the $R$-diameter of $M$ is 3, then $RRT$ is a forbidden triangle, as $T$ corresponds to $R$-distance 3. 

\begin{proposition}
	Suppose that $M$ is a semilinear homogeneous primitive 3-graph of $R$-diameter 3 and that each point $a$ is incident with $m<\omega$ lines. Then it is not possible for any $b\in S(a)$ to be collinear with $m$ elements from $R(a)$.
	\label{Intersection3}
\end{proposition}
\begin{proof}
	The $R$-neighbourhood of $b$ has $m$ $R$-connected components by transitivity. But by homogeneity and diameter 3, $b$ is adjacent to some element of $R^3(a)$. Therefore, if $R(b)$ meets each line through $a$, then $R(b)$ has at least $m+1$ $R$-connected components, contradicting homogeneity.
\end{proof}
\begin{proposition}
	Let $M$ be a semilinear homogeneous primitive 3-graph with $\diam_R(M)=3$ and $m<\omega$ lines through each point, and let $k$ denote $|R(b)\cap R(a)|$ for any $b\in S(a)$. Then $k=1$.
	\label{equalsk}
\end{proposition}
\begin{proof}
	By Proposition \ref{Intersection3}, $k<m$. The main point here is that we get the conclusion of Observation \ref{OneOrbit} for free in this situation, as the intersection of any pair of lines through $a$ with $R(a)$ is isomorphic to a complete bipartite graph (edges given by $S$, non-edges given by $R$). We can define an equivalence relation $E$ on $S(a)$ holding for $c,c'$ if $R(c)$ and $R(c')$ meet the same lines through $a$. By Proposition \ref{Intersection3} and homogeneity, $E$ is a nontrivial proper equivalence relation on $S(a)$ with $m\choose k$ classes. Notice that for any $E$-equivalent $c,c'$, the isomorphism types of $R(c)\cap R(a)$ and $R(c')\cap R(a)$ are the same over $a$, and in fact are the same as the isomorphism type of any set transversal to $k$ lines. Therefore, we can define $P_i(c,c')$ for $0\leq i< k$ if $E(c,c')\wedge |R(c)\cap R(c')\cap R(a)|=i$. All of these relations are realised by homogeneity, and invariant over $a$. This implies $k\leq 2$. 

	Now we eliminate the case $k=2$. If $|R(c)\cap R(a)|=2$ for $c\in S(a)$, then by Proposition \ref{Intersection3} we have at least 3 lines through $a$, and the relation $E$ defined in the preceding paragraph is a proper nontrivial equivalence relation. By the same argument, there are at least two types of $E$-equivalent pairs, plus at least two types of $E$-inequivalent pairs, depending on whether the intersections of their $R$-neighbourhoods meet $R(a)$ or not. The conclusion follows.
\end{proof}

The situation is similar to what we had in diameter 2 after Observation \ref{Quadrangle}, but we have the additional information $|R(b)\cap R(a)|=1$ for $b\in S(a)$.

\begin{proposition}
	Let $M$ be a semilinear primitive homogeneous 3-graph of $R$-diameter 3 and $m<\omega$ lines through each point. Then $m=2$.
\label{twolines}
\end{proposition}
\begin{proof}
	By Proposition \ref{equalsk}, for any $b\in S(a)$ we have $|R(b)\cap R(a)|=1$. Let $m$ denote the number of lines through $a$. We know by Observations \ref{NotRComplete} and \ref{R(a)Imprimitive} that $m\geq 2$. Now suppose for a contradiction that $m\geq 3$. Define $E_1,E_2$ on $S(a)$ by 
\begin{equation*}
E_1(c,c')\leftrightarrow R(c)\cap R(a)=R(c')\cap R(a)\\
\end{equation*}
\begin{equation*}
E_2(c,c')\leftrightarrow R(b,b')\vee b=b'
\end{equation*}
where $\{b\}=R(c)\cap R(a)$ and $\{b'\}=R(c')\cap R(a)$. The relation $E_2$ holds iff $R(c)$ and $R(c')$ intersect the same line through $a$; $E_1$ holds iff they meet $R(a)$ at the same point. There are $m$ $E_2$-classes and each of them contains infinitely many $E_1$-classes. Since $m\geq3$ and the $R$-diameter of $M$ is 3, each $E_1$-class contains at least two infinite disjoint cliques, corresponding to the lines through a particular $b\in R(a)$. Therefore, we can define an invariant $F(c,c')$ if $E_1(c,c')\wedge R(c,c')$, breaking each $E_1$-class into finitely many $R$-cliques. 

We have only three 2-types over $a$ in $S(a)$, corresponding to $R,S,T$, but we need at least four invariant relations for these three nested equivalence relations. 
\end{proof}
\begin{lemma}
	There are no primitive homogeneous 3-graphs of \SU-rank 2 and $R$-diameter 3.
	\label{NoDiam3}
\end{lemma}
\begin{proof}
	We know by Propositions \ref{equalsk} and \ref{twolines} that under the hypotheses of this proposition we have $|R(c)\cap R(a)|=1$ for all $c\in S(a)$ and there are exactly two lines through each point in $M$. So far, the main characters in our analysis have been $R(a)$ and $R^2(a)$. Now the structure on $R^3(a)$ will also come into play. The structure of $S(a)$ in diameter 3 and a single element in $|R(a)\cap R(c)|$ consists, by Proposition \ref{twolines}, of two $E_2$-classes, each divided into infinitely many $E_1$-classes ($R$-cliques), where $E_1,E_2$ are as in the proof of Proposition \ref{twolines}. We have two subcases:
\begin{case}
	Suppose that $S$ holds between $E_1$-classes contained in the same $E_2$-class. Take $d\in T(a)$. The set $R(d)\cap R^2(a)$ meets each $E_1$-class in at most one vertex and one $E_2$ class ($T$ holds across $E_2$-classes; if $R(d)\cap R^2(a)$ met both $E_2$-classes, then the triangle $RRT$ would be realised, contradicting our assumption that $T(a)=R^3(a)$). Therefore, we can define an equivalence relation on $T(a)$ with two classes: define 
\begin{equation*}
F(d,d')\leftrightarrow\exists(c,c'\in S(a))(c\in R(d)\cap S(a)\wedge c'\in R(d')\cap S(a)\wedge E_2(c,c'))
\end{equation*}
So $F(d,d')$ holds iff $R(d)$ and $R(d')$ meet the same $E_2$-class in $R^2(a)$. We have a further subdivision into cases, depending on how many $E_1$-classes $R(d)$ meets:
\begin{subcase}
	If $|R(d)\cap R^2(a)|=1$, then we can define on $T(a)$ two more equivalence relations:
\begin{equation*}
	F'(e,e')\leftrightarrow E_1(c,c')
\end{equation*}
\begin{equation*}
	F''(e,e')\leftrightarrow R(e)\cap S(a)=R(e')\cap S(a)
\end{equation*}
where $\{c\}=R(e)\cap S(a)$ and $\{c'\}=R(e')\cap S(a)$. The condition $|R(d)\cap R^2(a)|=1$ ensures that these relations are transitive. Clearly, $F''\rightarrow F'\rightarrow F$; and as there are two lines through any vertex, $F$ is a proper nontrivial equivalence relation. To prove that $F'$ and $F''$ are both realised and different, take any $c\in S(a)$. There are two lines incident with it, one of which is its $E_1$-class, together with some point from $R(a)$; the other line, $\ell$, through $c$ is almost entirely contained in $T(a)$. Two points on $\ell\cap T(a)$ satisfy $F''$, and $F$-equivalent points in $T(a)$ on lines through different elements from $S(a)$ satisfy $F'\wedge\neg F''$ if the elements from $S(a)$ belong to the same $E_1$-class, and they satisfy $F\wedge\neg F'$ if the elements from $S(a)$ are $E_2$-equivalent and $S$-related.This gives us three nested invariant equivalence relations in $T(a)$. This rules out the possibility of $|R(d)\cap R^2(a)|=1$ in the situation of Case I(i). 
\end{subcase}
\begin{subcase}
	If $R(d)$ meets more than one $E_1$-class, then by homogeneity, since any vertex lies on two lines, it has to intersect exactly two of $E_1$-classes. Note that $R(d)\cap S(a)$ is contained in a single $E_2$-class, because the triangle $RRT$ is forbidden. Again, we find too many types realised on $T(a)$. For any pair $d,d'\in T(a)$, the number of $E_1$-classes that $R(d)\cup R(d')$ meets is invariant under $a$-automorphisms. Notice that it is not possible for $|(R'(d)\cup R(d'))\cap R^2(a)|$ to be 2, as in that case $d$ and $d'$ would belong to two different lines: by homogeneity, each element $c\in S(a)$ lies on two lines, one of which is its $E_1$-class; therefore, if $d,d'\in T(a)$ are such that $R(d,c)\cap R(d',c)\neq\varnothing$, then $c,d,d'$ must be collinear. Define $F_1(d,d')$ on $T(a)$ if $R(d)$ and $R(d')$ meet the same two $E_1$-classes, and $P(d,d')$ if $R(d)\cap R(d')\cap S(a)\neq\varnothing$. There are pairs satisfying all of $F_1\wedge P, F_1\wedge\neg P, \neg F_1\wedge P, \neg F_1\wedge\neg P$, giving us four invariant relations on pairs from $T(a)$.
\end{subcase}
\end{case}
\begin{case}
If $T$ holds between $E_1$-classes contained in the same $E_2$-class, then $S$ holds between $E_2$-classes (as each $E_1$-class is an $R$-clique). Again, we have two subcases, depending on $|R(d)\cap S(a)|$ for $d\in T(a)$:
\begin{subcase}
If $|R(d)\cap S(a)|=1$ for $d\in T(a)$, then we can define an equivalence relation $E'(e,e')$ on $T(a)$ holding if $R(e)$ and $R(e')$ meet the same $E_2$-class in $S(a)$. We will show that we already have three invariant and mutually exclusive relations on unordered pairs in each of the $E'$ classes. Define $\hat R, \hat T$ on $T(a)$ by $\hat P(e,e')$ iff $P$ holds for the points in the intersection of $R(e)$ and $R(e')$ with $S(a)$ ($P\in\{R,T\}$), and $C(e,e')$ if $e,e'$ are collinear with some $c\in S(a)$, which happens if $R(e)\cap R(e')\cap S(a)\neq\varnothing$. We would need at least one more predicate to separate the $E'$-classes.
\end{subcase}
\begin{subcase}
And if $|R(d)\cap S(a)|=2$ for $d\in T(a)$, then the intersection with each $E_2$-class is of size one, as otherwise the triangle $RRT$ would be realised. Then we can count the total number of $E_1$-classes that $R(e)$ and $R(e')$ meet, which can be 4, 3, or 2. And in the cases where this number is 3 or 2, we have another two relations, depending on whether $R(e)\cap R(e')\cap S(a)$ is empty or not. Again, we find too many invariant and mutually exclusive relations on unordered pairs of distinct elements from $T(a)$. 
\end{subcase}
\end{case}
\end{proof}

We can now prove that no primitive homogeneous simple 3-graphs have \SU-rank 2.

\begin{theorem}\label{NoRank2}
	There are no homogeneous primitive simple 3-graphs of \SU-rank 2.
\end{theorem}
\begin{proof}
By Observation \ref{diam}, the diameter of a primitive homogeneous simple 3-graph of \SU-rank 2 is either 2 or 3. Lemmas \ref{NoDiam2} and \ref{NoDiam3} say that both situations are impossible.
\end{proof}

\section{Higher rank}\label{SectHigher}

We have now proved that there are no homogeneous primitive simple 3-graphs of \SU-rank 2. In this section, we see that result as the basis for an inductive argument on the rank of the theory, using Theorem \ref{ThmStableForking}. We remark that in the course of the proof of nonexistence of simple 3-graphs of rank 2, we only use the rank 2 hypothesis to prove that we can define in $M$ a semilinear space with finitely many lines through each point. Also, for most of the analysis simplicity suffices, and we require supersimplicity only in Propositions \ref{atleast3} and \ref{MataCasiTodos} (and, indirectly, Lemma \ref{NoDiam2} because the proof uses Propositions \ref{atleast3} and \ref{MataCasiTodos}); in these results we use the fact that the theory is ranked by \SU, but the specific value of its rank is irrelevant. 

Therefore, if we prove that simple homogeneous 3-graphs of rank 3 or greater are semilinear with finitely many lines through each point, then the rest of the argument from Section \ref{SectRank2} is valid in higher rank.

\begin{lemma}
	Let $M$ be a homogeneous primitive supersimple 3-graph of \SU-rank $k\geq2$. Then $M$ is semilinear.
\label{Inductive}
\end{lemma}
\begin{proof}
	Independently of the rank, if $\diam_R(M)=3$, then $R(a)$ is a stable $RS$-graph. It cannot be primitive by Theorem \ref{LachlanWoodrow}. And $S$ is not an equivalence relation by Proposition \ref{PropMultipartite}; therefore, $R$ is an equivalence relation on $R(a)$ with finitely many infinite classes (by Observation \ref{NoInfiniteCliques}).

	So we need only worry about those cases with $\diam_R(M)=2$. We proceed by induction on $k$ (common induction, as opposed to transfinite induction, suffices by Koponen's Theorem \ref{Koponen}). The case $k=2$ corresponds to Lemma \ref{LemmaSemilinear}. Suppose that up to $k\geq3$, we know that there are no primitive homogeneous simple 3-graphs of \SU-rank $k-1$ (for $k=3$, this is the content of Theorem \ref{NoRank2}). 

If we are given a homogeneous primitive simple 3-graph of \SU-rank $k+1$ and $R$-diameter 2, then we may assume by Theorem \ref{ThmStableForking} that $S$ and $T$ are nonforking predicates, so we know that $R(a)$ is a simple homogeneous 3-graph of rank at most $k$. It follows that $R(a)$ is either imprimitive or of rank 1 as a structure in its own right (it could have a higher rank as a subset of $M$ due to external parameters). If $R(a)$ is imprimitive, the same arguments as in Proposition \ref{RDefinesEqRel} show that $R$ is an equivalence relation; by Observation \ref{NoInfiniteCliques}, it has finitely many classes. 

Now we argue that $R(a)$ is not primitive. By the induction hypothesis, if $R(a)$ were primitive, then its rank would be 1. 

The structure on $R(a)$ cannot be stable, as in that case it would be one of Lachlan's infinite stable 3-graphs from Theorem \ref{Lachlan3graphs}, all of which are imprimitive. 

And $R(a)$ cannot be isomorphic to a primitive unstable 3-graph of rank 1, as by Proposition \ref{Year2} primitivity contradicts the stability of $R$. Therefore, $R(a)$ is imprimitive and $R$ defines an equivalence relation on $R(a)$ with finitely many classes, by Observation \ref{NoInfiniteCliques}. This proves the proposition for all natural numbers $k\geq 3$.
\end{proof}

Lemma \ref{Inductive} tells us that we can define a semilinear space on $M$ just as we did in subsection \ref{lines}. The analysis from subsections \ref{subsectRDiam2} and \ref{subsectRDiam3} translates verbatim to this more general setting, as the rank hypothesis was used there only to ensure that $M$ interprets a semilinear space. As a consequence,

\begin{theorem}
	Let $M$ be a primitive simple homogeneous 3-graph. Then the theory of $M$ is of \SU-rank 1.
\label{NoHigherRank}\hfill$\Box$
\end{theorem}

As a consequence of Theorems \ref{ThmStableForking} and \ref{NoHigherRank}, all the homogeneous simple unstable primitive 3-graphs of finite \SU-rank have rank 1. We know from Chapter \ref{ChapRank1} that those are random, and that in the case of imprimitive structures with finitely many classes, the transversal relations in a pair of classes are null, complete or random. This leaves us only with imprimitive structures of rank 2.

\chapter{Simple 3-graphs of \SU-rank 2}\label{ChapRank2}
\setcounter{equation}{0}
\setcounter{theorem}{0}
\setcounter{case}{0}
\setcounter{subcase}{0}

We proved in Chapter \ref{ChapPrimitive} that all primitive homogeneous 3-graphs with simple theory have $\SU$-rank 1. Therefore, all such 3-graphs with rank higher than 1 are imprimitive. Since the invariant equivalence relation is, by quantifier elimination, definable as a disjunction of atomic formulas and formulas defining equivalence relations are stable, we are left with two cases:
\begin{enumerate}
	\item{Either a single predicate, say $R$ defines an equivalence relation with infinitely many infinite classes, or}
	\item{The disjunction of the unstable predicates $S,T$ defines an equivalence relation with infinitely many infinite classes}
\end{enumerate}

Note that the case in which the disjunction of a stable predicate and an unstable predicate defines an equivalence relation does not occur, since the complement of an equivalence relation is also a stable predicate. 

\begin{remark}
	The $\SU$-rank of a homogeneous simple unstable 3-graph $M$ cannot be higher than 2, since that would mean that $M$ is imprimitive and the equivalence relation $E$ has infinitely many infinite classes, each of which is of rank at least 2, implying again imprimitivity and thus stability (we used at least one predicate to define $E$ and one more to define the finer equivalence relation; it follows that all predicates are stable).
\end{remark}

Two structures are easy to isolate:

\begin{proposition}\label{PropInfInf1}
	Let $M$ be an unstable homogeneous 3-graph in which the disjunction of two unstable predicates $S,T$ defines an equivalence relation with infinitely many infinite classes. Then each class is isomorphic to the Rado graph.
\end{proposition}
\begin{proof}
	Let $C$ be an $S\vee T$-class and consider two finite isomorphic substructures $A,B$ of $C$. By homogeneity of $M$, there exists an automorphism of $C$ taking $A$ to $B$ (namely, the restriction to $C$ of the automorphism of $M$ that exists by homogeneity, which clearly fixes $C$). 

	Clearly, any half-graph witnessing the instability of $S$ and $T$ is contained in a $S\vee T$-class, so it follows that each class is isomorphic to the Rado graph by the Lachlan-Woodrow Theorem. The structure is therefore $K_\omega^R[\Gamma]$.
\end{proof}

\begin{proposition}\label{PropInfInf2}
	Let $M$ be an unstable homogeneous simple 3-graph in which $R$ defines an equivalence relation with infinitely many infinite classes, and suppose that in the union of any two classes only two predicates are realised. Then $M\cong\Gamma[K_\omega^R]$.
\end{proposition}
\begin{proof}
This follows from the same argument used in Observation \ref{ObsInterpretedGraph}
\end{proof}

\section{The remaining case}
Propositions \ref{PropInfInf1} and \ref{PropInfInf2} leave us with the work of classifying those homogeneous simple unstable 3-graphs in which $R$ defines an equivalence relation with infinitely many infinite classes and the induced action of $\aut(M)$ on $M/R$ is 2-transitive. As in Chapter \ref{ChapImprimitiveFC}, we start our analysis by identifying the structure induced on a pair of classes. 

The argument from Observation \ref{NotCompBipart} implies that if we fix some vertex $a$, then $S(a)\cap C\neq\varnothing$ and $T(a)\cap C\neq\varnothing$ for all $R$-classes $C$ not containing $a$. 

\begin{proposition}\label{PropPerfectMatching}
	Let $M$ be a homogeneous simple unstable 3-graph in which $R$ defines an equivalence relation with infinitely many infinite classes and $\aut(M)$ acts 2-transitively on $M/R$. If $S(a)\cap C$ is finite for some class $C$ not containing $a$, then the structure induced by $M$ on a pair of distinct classes $C,D$ is a perfect matching.
\end{proposition}
\begin{proof}
	We claim that the structure induced on $C,D$ is homogeneous. Consider two isomorphic finite subsets $A,B$ of $C\cup D$. If $S$ or $T$ is realised in $A$, then the usual argument (see Proposition \ref{PropInfInf1}) proves that there exists an automorphism of $C\cup D$ taking $A$ to $B$. 

And if $A\subset C,B\subset D$ are $R$-cliques of the same size, then  we can find vertices $v_A, v_B$ in the opposite side of $A,B$ such that $T(v_A,a)$ and $T(v_B,b)$ for all $a\in A, b\in B$. This follows from the fact that, by homogeneity, the sets $\bigcup_{a\in A}S(a)\cap D$ is finite and $D$ is infinite. Now $Av_A$ and $Bv_B$ are isomorphic and we can apply the same argument as in the preceding paragraph. The same idea works when $A,B$ are subsets of the same class.

Thus $C\cup D$ is a homogeneous stable (because the $S$-neighbourhood of a vertex in $C\cup D$ is finite and $R$ is an equivalence relation) 3-graph in which $R$ is an equivalence relation with two infinite classes and no disjunction of two predicates is an equivalence relation. Going through Lachlan's list, we conclude that $C\cup D\cong K_\omega^R\times K_2^S$.
\end{proof}

In Chapter \ref{ChapImprimitiveFC} we proved that the only ``interesting" imprimitive homogeneous simple unstable 3-graph has the property that $S$ (and $T$) form perfect matchings in the union of any two classes. In that case, since the classes are of size 2 if one of the predicates matches the classes, then so does the other. One could imagine that an analogue with infinite classes would have either $S$ or $T$ as a matching between any two classes, possibly embedding all $S,T$-graphs as transversals. Fortunately, no such monster exists:

\begin{proposition}\label{PropIntersectionTypes}
	Let $M$ be an imprimitive homogeneous simple unstable 3-graph in which $R$ defines an equivalence relation with infinitely many infinite classes, and suppose that both $S$ and $T$ are realised in a pair of classes. If a pair of distinct classes $C,D$ embed a half-graph for $S,T$, then for any finite disjoint $A,B\subset C$ of the same size there exist $d,d'\in D$ such that $\tp(d/A)=\tp(d'/\sigma(B))$ for some permutation $\sigma$.
\end{proposition}
\begin{proof}
Note first that the existence of a half-graph for $S,T$ in $C\cup D$ implies that $S(c)\cap D$ and $T(c)\cap D$ are infinite (as are the corresponding neighbourhoods in $C$ for any vertex in $D$). 

We proceed by induction on the size of $A$. For $|A|=1$, this follows from the fact that $S(a)\cap E\neq\varnothing$ for all classes $E$ not including $a$, so there exist vertices in $D$ which are $S$-related to $A$, $B$. 

Now suppose that up to $|A|=n$ we can find vertices in $D$ connected to $A$ and $B$ only by $S$-edges, and let $a_{n+1},b_{n+1}$ be vertices in subsets $A',B'$ of $C$ of size $n+1$. We know that there exist subsets $X$ isomorphic to $A'$ such that there exists a vertex in $D$ that is $S$-related to all the elements of $X$, namely any vertex in the $D$-side of a half-graph embedded in $C\cup D$ with index larger than any of the first $n+1$ elements of the half-graph. The result follows by homogeneity.
\end{proof}
\begin{corollary}\label{CorHomPairs}
	Let $M$ be an imprimitive homogeneous simple unstable 3-graph in which $R$ defines an equivalence relation with infinitely many infinite classes, and suppose that both $S$ and $T$ are realised in a pair of classes. The structure induced on a pair of classes is homogeneous.
\end{corollary}
\begin{proof}
	Fix two classes $C,D$, and consider two finite isomorphic subsets $A\subset C\cup D, B\subset C\cup D$. If either of $A, B$ includes an $S$- or $T$-edge, then the isomorphism extends to an automorphism of $M$, which fixes $C\cup D$ setwise $C,D$ (and therefore its restriction to $C\cup D$ is an automorphism of the induced structure on $C\cup D$).

	Now if each of $A,B$ is a subset of $C$ or $D$, we can use Proposition \ref{PropIntersectionTypes} to find an element in the opposite class $S$-related to $A$ and to $B$ and apply the same argument as before to the extended subsets. 
\end{proof}
\begin{theorem}\label{ThmNoNewStr}
	There is no homogeneous unstable 3-graph $M$ in which $R$ defines an equivalence relation with infinitely many infinite classes and the structure induced by $M$ on any pair of distinct classes is $K_\omega^R\times K_2^T$. 
\end{theorem}
\begin{proof}
\comm
Since $R$ defines an equivalence relation, any such $M$ would satisfy $S\sim T$ by instability. Note that $S\sim^RT$ is impossible because it implies that $T(a)$ contains $R$-edges. It follows that either $S\sim^ST$ or $S\sim^TT$ holds. Any of them implies $SST,STT\in\age(M)$. Since $T$ is a perfect matching between two classes, the triangles $RST$ and $SSR$ are also in $\age(M)$.

The forbidden triangles of $M$ are $TTR$ (because $T$ is a perfect matching between two $R$-classes), $RRT$, and $RRS$ (because $R$ is transitive). Note that in any amalgamation problem of one-point extensions $B,C$ of $A$, if there is an element $a\in A$ such that $R(b,a)\wedge T(c,a)$ then the edge $S(b,c)$ is forced by the forbidden triangles $TTR, RRT$. If the amalgamation property holds for $\age(M)$, then the following two structures are also in $\age(M)$ (they are the only possible solution to the amalgamation problem of $b,c$ over the bottom edge): 
\[
\includegraphics[scale=0.8]{Rank206.pdf}
\]
But now if we try to amalgamate $U,U'$ over the common $K_3^S$ we obtain the following problem without a solution in $\age(M)$:
\[
\includegraphics[scale=0.8]{Rank207.pdf}
\]
Contradicting the amalgamation property.
\ent
This is a direct consequence of Observation \ref{ObsEasyCase}.
\end{proof}
\begin{corollary}\label{Corollary1}
The only simple homogeneous 3-graph in which $R$ defines an equivalence relation with infinitely many infinite classes and $T$ is a perfect matching between any two distinct $R$-classes is $K_\omega^R\times K_\omega^T$.
\end{corollary}
\begin{corollary}\label{Corollary2}
	Let $M$ be an imprimitive homogeneous simple unstable 3-graph in which $R$ defines an equivalence relation with infinitely many infinite classes, and suppose that the induced action of $\aut(M)$ on $M/R$ is 2-transitive. Then the structure induced by $M$ on the union of any two distinct classes $C,D$ is isomorphic to the Random Bipartite Graph.
\end{corollary}
\begin{proof}
It follows from Theorem \ref{ThmNoNewStr} and Corollary \ref{Corollary1} that there must be half-graphs present in the union of $C$ and $D$, that is, $S\sim^R_RT$. We will prove that the axioms of the Random Bipartite Graph hold for $C\cup D$. Let $A,B$ be two finite disjoint subsets of $C$; we wish to prove that there exists $d\in D$ such that $S(d,a)$ for all $a\in A$ and $T(d,b)$ for all $b\in B$.

Let $n_A$ and $n_B$ be the sizes of $A,B$, and let $(\alpha_i)_{i\in\omega}\subset C$ and $(\beta_i)_{i\in\omega}\subset D$ be a half-graph for $S,T$ in $C\cup D$. Consider the first $n_A+n_b+1$ elements of the sequence $(\alpha_i)_{i\in\omega}$: the vertex $\beta_{n_A+1}$ is such that $S(\alpha_i,\beta_{n_A+1})$ for all $i\leq n_A$, and $T(\alpha_i,\beta_{n_A+1})$ for all $i>n_A$. Now we can apply Corollary \ref{CorHomPairs} to find an element in $D$ that satisfies over $A\cup B$ the same type as $\beta{n_A+1}$ over $\alpha_1,\ldots,\alpha_{n_A+n_b+1}$.
\end{proof}
\begin{corollary}\label{Corollary3}
	If $M$ is an imprimitive homogeneous simple unstable 3-graph in which $R$ defines an equivalence relation with infinitely many infinite classes, and such that $S$ and $T$ are realised in any pair of classes. Then $M$ embeds infinite cliques in all colours.
\end{corollary}
\begin{proof}
	We know already that $M$ embeds infinite $R$-cliques, and by Ramsey's Theorem it must embed an infinite clique in at least one of $S,T$. Let us suppose without loss that it embeds infinite $S$-cliques. 

	Note that $R$ is the only equivalence relation, so in particular there are no equivalence relations with finitely many classes on $M$. This follows from instability, quantifier elimination, and Corollary \ref{Corollary2}: any such relation would be definable as a disjunction of atomic formulas, but we know by instability that neither $S$ nor $T$ define equivalence relations. Corollary \ref{Corollary2} implies that the triangle $RST$ is in $\age(M)$, so no disjunction of two predicates is an equivalence relation.

	If $M$ does not embed infinite $T$-cliques, then it must be the case that $T(x,a)$ forks, since otherwise the Independence Theorem would guarantee the existence of infinite $T$-cliques. From this it follows that the only nonforking relation in $M$ is $S$, so $T(a)$ does not embed infinite $S$-cliques. But $T(a)$ meets all classes except $a/R$, so there is an infinite transversal in $T(a)$. This contradicts Ramsey's Theorem.
\end{proof}
\begin{proposition}
	Let $M$ be a homogeneous simple unstable 3-graph in which $R$ defines an equivalence relation with infinitely many infinite classes, and suppose that between any two $R$-classes both $R$ and $S$ are realised. Then both $S$ and $T$ are nonforking.
\end{proposition}
\begin{proof}
	Since $R$ defines an equivalence relation with infinitely many infinite classes, the formula $R(x,a)$ divides. At least one of $S$ or $T$ is nonforking, since there must be Morley sequences in the only 1-type over $\varnothing$, so let us assume without loss that $T$ is nonforking. If the formula $S(x,a)$ divides, then $S(a)$ does not contain infinite $T$-cliques. We know from Corollary \ref{Corollary2} that the structure between any two classes is isomorphic to the Random Bipartite Graph, so $R$ defines an equivalence relation in $S(a)$ with infinitely many infinite classes, and since $S(a)$ is a structure interpreted in $M$, the theory of $S(a)$ is simple. 

	Note that if $S$ divides, then $T(a)$ is isomorphic to $M$. From this it follows that $\aut(M)_a$ acts 2-transitively on the set of $R$-classes in $S(a)$, so $S(a)$ is either $T$-free or both $S$ and $T$ are realised in $S(a)$ between any two classes. 

\begin{claim}
	$S(a)$ is not $T$-free.
\end{claim}
\begin{proof}
If $S(a)$ is $T$-free, then the triangle $SST$ is forbidden in $M$, which in particular implies that $S\not\sim^ST$, so $S\sim^R_RT$ is forced. From this we derive that there are three forbidden triangles in $M$, namely $RRT, RRS, SST$. Now we will prove that this is inconsistent with homogeneity.

We know from Corollary \ref{Corollary2} that $TTR, SSR$ are in $\age(M)$. 

\comm Since $S(a)$ is $T$-free and $R$ forms a non-trivial equivalence relation on $S(a)$, the structure on four vertices with five $S$-edges and one $R$-edge is also in $\age(M)$. By $S\sim^R_RT$, 
\[
\includegraphics[scale=0.8]{Rank201.pdf}
\]
is also in $\age(M)$. 
Therefore, in the following amalgamation problem $T$ is the only possible solution between $b,c$:
\[
\includegraphics[scale=0.8]{Rank202.pdf}
\]
\ent
And since $T(a)\cong M$ and $SSR\in\age(M)$, 
\[
C=\includegraphics[scale=0.8]{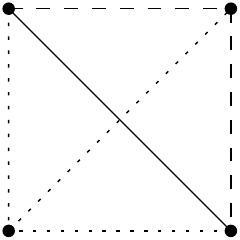}
\]
is also in $\age(M)$. Finally, it follows from $T(a)\cong M$ and $RST\in\age(M)$ that 
\[
B=\includegraphics[scale=0.8]{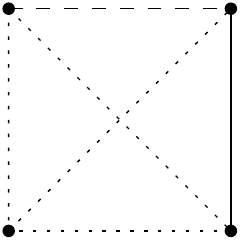}
\]
is in $\age(M)$. Now we amalgamate $B,C$ over a triangle $RTT$:
\[
\includegraphics[scale=0.8]{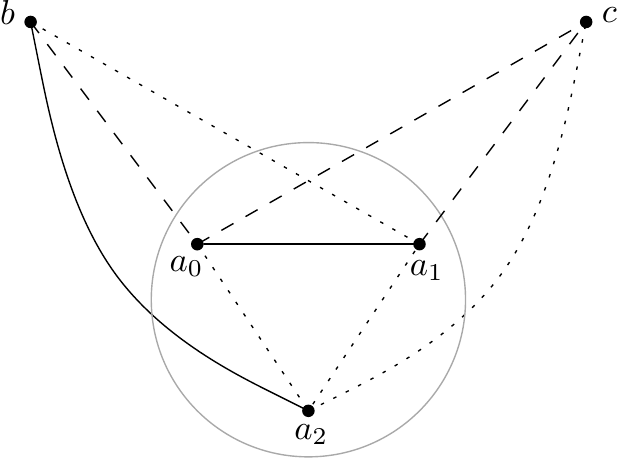}
\]
 We cannot have $R(b,c)$ because $R$ because $RRT$ would appear on $b,a_2,c$; and $T(b,c)$ would form $SST$ on the vertices $b,a_0,c$. Finally, $S(b,c)$ would form $SST$ on $b,c,a_1$. Therefore, the amalgamation problem has no solution in $\age(M)$, contradicting homogeneity. From this we conclude that both $S$ and $T$ are realised in $S(a)$.
\end{proof}
If $S(a)$ is isomorphic to a stable graph, it follows from Corollary \ref{Corollary1} that $S(a)\cong K_\omega^R\times K_2^S$, because each pair of classes is a stable structure. But then $S(a)$ embeds infinite $T$-cliques. 

So $S(a)$ is an unstable structure. By Corollary \ref{Corollary3}, $S(a)$ embeds infinite $T$-cliques. In any case, we have reached a contradiction. Therefore, $T$ is nonforking.
\end{proof}
\comm
\begin{proposition}\label{PropInfInter}
	Let $M$ be a homogeneous simple unstable 3-graph in which $R$ defines an equivalence relation with infinitely many infinite classes and $\aut(M)$ acts 2-transitively on $M/R$. Then for any vertex $a$ and $R$-class $C$ not containing $a$, both $S(a)\cap C$ and $T(a)\cap C$ are infinite.
\end{proposition}
\begin{proof}
We know by 2-transitivity that $S(a)\cap C$ and $T(a)\cap C$ are nonempty for any $C$ not containing $a$. Suppose for a contradiction that $S(a)\cap C$ is finite, so in particular $S(x,a)$ divides (witnessed by an $R$-clique) and $T$ is nonforking. This immediately implies that $T(a)\cap C$ is infinite.

By Proposition \ref{PropPerfectMatching}, $|S(a)\cap C|=1$ and $S(a)$ is $R$-free. Then in particular $S(x,a)$ divides, as witnessed by an $\varnothing$-indiscernible sequence of vertices that forms an infinite $R$-clique. 

We know that $S$ and $T$ are the only unstable relations in $M$, so we must have, in the notation from chapter \ref{ChapStableForking} $S\sim T$. By Proposition \ref{IndiscerniblePairs}, at least one of $S\sim^RT, S\sim^TT, S\sim^ST$ has witnesses in $M$. Of these three options, the first implies that $S(a)\cap C$ is infinite, and thus not $R$-free; the second one implies the existence of infinite $T$-cliques in $S(a)$, impossible since $T$ is the only nonforking predicate and therefore $T$-cliques are Morley sequences over $\varnothing$. Thus, $S\sim^ST$ must hold.

Note that $S(a)$ is a simple $S,T$-graph not embedding infinite $T$-cliques, and therefore it is stable. It follows from the Lachlan-Woodrow Theorem that $S(a)$ is either a union of finite $T$-cliques or an infinite $S$-clique.

If $S(a)$ is an $S$-clique, it follows from Observation \ref{sop} that for any $b\in S(a)$ we have $\{a\}\cup S(a)=\{b\}\cup S(b)$, so we can define an equivalence relation by $x\sim_Sy$ if $S(x,y)\vee x=y$. This relation is clearly symmetric and reflexive, and by transitivity of the action of $\aut(M)$, if $S(a,b)\wedge S(b,c)$, then $S(a,c)$ follows because $a,c\in S(b)$. Therefore $S$ is stable and so is the theory of $M$.

The only possibility is for $S(a)$ to be a union of finite $T$-cliques, but this is impossible since $S\sim^ST$ implies that $S(a)$ is unstable.

We conclude that $S(a)\cap C$ is infinite.
\end{proof}

\begin{proposition}
	Let $M$ be a homogeneous simple unstable 3-graph in which $R$ defines an equivalence relation with infinitely many infinite classes and $\aut(M)$ acts 2-transitively on $M/R$. Then $S$ and $T$ are nonforking.
\end{proposition}
\begin{proof}
The formula $R(x,a)$ divides over $\varnothing$ because it defines one class of an equivalence relation with infinitely many infinite classes; we know by simplicity that there must be a nonforking predicate, since in simple theories there exist Morley sequences for all types. Let us suppose without loss that $T$ is nonforking.

If $S$ divides, then there are no infinite $T$-cliques in $S(a)$. Since $S(a)$ meets all classes except $a/R$, it follows that $S(a)$ contains an infinite $S$-clique; furthermore, by Proposition \ref{PropInfInter}, the intersection of $S(a)$ with any $R$-class is infinite. It follows that $S(a)$ is an infinite homogeneous structure in which $R$ defines an equivalence relation with infinitely many infinite classes.

We know that $S\sim T$ holds because $S,T$ are the only unstable predicates. Since $S(a)$ does not embed infinite $T$-cliques, $S\not\sim^TT$. Also note that by the same argument as in Observation \ref{IsoToM}, $T(a)\cong M$, so $\aut(M)_a$ acts 2-transitively on the set of classes not including $a/R$. From this it follows by invariance of $R$ that $\aut(M)_a$ also acts 2-transitively on $S(a)/R$. Therefore, either $S$ and $T$ are realised in any two distinct $R$-classes in $S(a)$, or only $S,R$ are realised in $S(a)$.

We know by the same argument as in the Claim that $\aut(M)_a$ acts 2-transitively on the set of $R$-classes of $S(a)$, so both $S$ and $T$ are realised in the union of any two $R$-classes in $S(a)$.

\begin{claim}
	$S\not\sim^ST$.
\end{claim}
\begin{proof}
	If $S\sim^ST$, then $S(a)$ embeds an infinite half-clique for $S,T$ and therefore $T$ is not algebraic in $S(a)$
\end{proof}
\end{proof}
\ent
\begin{corollary}\label{Corollary4}
	Let $M$ be a homogeneous simple unstable 3-graph in which $R$ defines an equivalence relation with infinitely many infinite classes, and suppose that between any two $R$-classes both $R$ and $S$ are realised. Then $M$ embeds all finite $S,T$-graphs. In particular, $S\sim^ST$ and $S\sim^TT$.
\end{corollary}
\begin{proof}
We know from the proof of Corollary \ref{Corollary3} that $R$ is the only nontrivial proper equivalence relation on $M$. From this it follows that there is a unique strong type of singletons over $\varnothing$, which is Lascar strong because these theories are low. This allows us to apply the usual argument (see for example Theorem \ref{PrimitiveAlice}) to form an $S,T$-graph.
\end{proof}

Our goal now is to prove that any $S,T$-graph of size $n$ is realised as a transversal in any union of $n$ $R$-classes. To prove this it suffices, by homogeneity, to prove that any $n$ classes embed a $K_n^S$ (so $\aut(M)$ acts highly transitively on $M/R$).

\begin{proposition}\label{PropBasis}
	Let $M$ be a homogeneous simple unstable 3-graph in which $R$ defines an equivalence relation with infinitely many infinite classes, and suppose that between any two $R$-classes both $R$ and $S$ are realised. Then $\aut(M)$ acts 3-transitively on $M/R$.
\end{proposition}
\begin{proof}
	It follows from Corollary \ref{Corollary4} that $S(a)$ and $T(a)$ are unstable 3-graphs, since $S\sim^ST$ and $S\sim^TT$, and from Corollary \ref{Corollary2}, that $R$ defines an equivalence relation with infinitely many infinite classes in $S(a),T(a)$. We wish to prove that $\aut(M)_a$ acts 2-transitively on the set of $R$-classes in $S(a)$.

	Suppose for a contradiction the action is not 2-transitive. Then the structure on any pair of $R$-classes in $S(a)$ must be that of a complete bipartite graph. By Proposition \ref{PropInfInf2} $S(a)\cong T(a)\cong\Gamma[K_\omega^R]$. We have two possibilities: either the structures are ``aligned" (\ie for pairs $R(c,c'), R(b,b')$ with $S(a,c)$, $T(a,c')$ and $S(a,b),T(a,b')$ we have $(S(b,c)\wedge S(b',c'))\vee(T(b,c)\wedge T(b',c'))$) or the opposite relation between two classes in $S(a)$ is realised between the corresponding classes in $T(a)$.

In the latter case, given any $c_1,c_2\in S(a)$ with $S(c_1,c_2)$ we can find a $d_1$ with $R(d_1,c_2)$ and $T(c_1,d_1)$ because $T(c_1)\cap d_1/R\neq\varnothing$ and we know that all the elements of $d_1/R\cap S(a)$ are $S$-related to $c_1$. Similarly, given $d_2$ in a third class with $S(d_2,d_1)$, there is $c_2$ satisfying $R(c_2,d_1)\wedge T(d_2,c_2)$, as in the following figure:
\[
\includegraphics[scale=0.8]{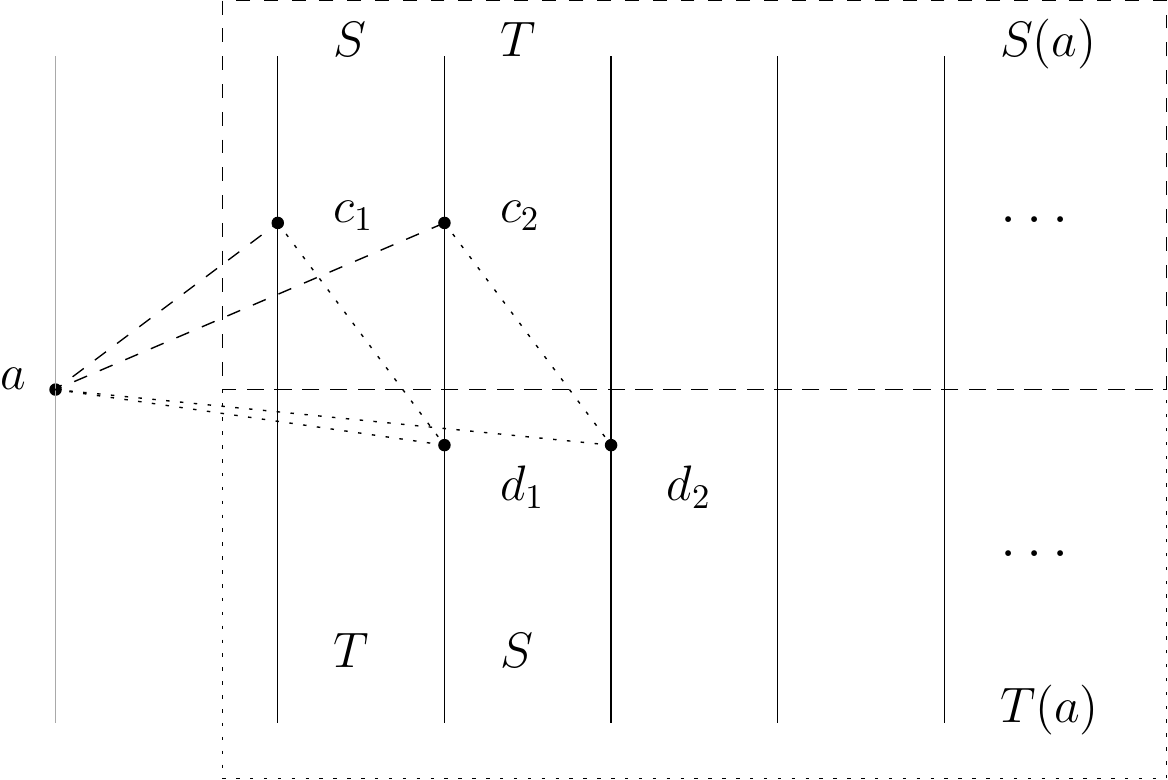}
\]

The pairs $c_1,d_1$ and $c_2,d_2$ are isomorphic over $a$, so there must be $\sigma\in\aut(M)_a$ such that $\sigma(c_1)=c_2, \sigma(d_1)=d_2$. But such automorphism would take $c_2$ to some $e'$ with $R(e',d_2)$, so the $S$-edge $c_1c_2$ would necessarily be taken to a $T$-edge, contradicting homogeneity.

This leaves us with the ``aligned" case. The same argument works here, but we need to justify the existence of the vertices $c_1,c_2,d_1,d_2$ slightly differently. Take three distinct classes $C,D,E$ not containing $a$, such that in $C\cap S(a)$ and $D\cap S(a)$ form an $S$-complete bipartite graph, and $D\cap S(a)$, $E\cap S(a)$ form a $T$-complete bipartite graph. We know that $C=(C\cap S(a))\cup (C\cap T(a))$ (and similar statements for $D,E$), and that each pair of classes is a Random Bipartite Graph. Let $d_0\in S(a)\cap D$ and $d_1\in T(a)\cap D$. By the extension axiom in the Random Bipartite Graph, there exists some $c_1\in C$ such that $S(d_0,c_1)\wedge T(c_1,d_0)$. This $c_1$ must be in $S(a)$, $C\cap T(a),D\cap T(a)$ are complete bipartite in the predicate $S$:
\[
\includegraphics[scale=0.8]{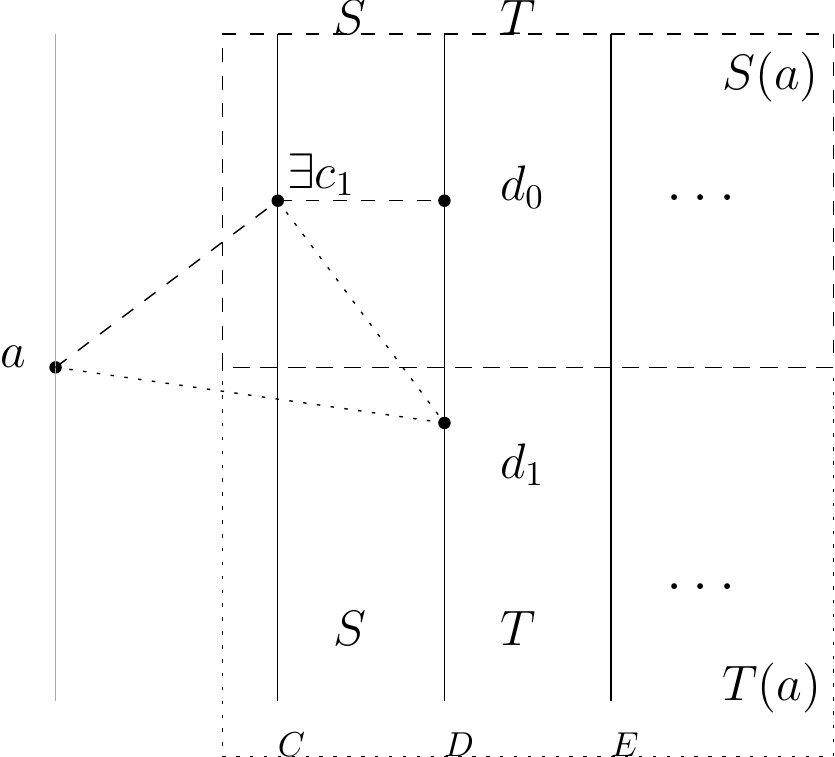}
\]
And between the classes $D,E$ that whose intersections with the neighbouhoods of $a$ form $T$-complete bipartite graphs, take $d_2,d_3\in T(a)\cap E$. As before, there exists $c_2\in D$ such that $T(d_3,c_2)\wedge S(d_2,c_2)$. This $c_2$ must be in $S(a)$.
\[
\includegraphics[scale=0.8]{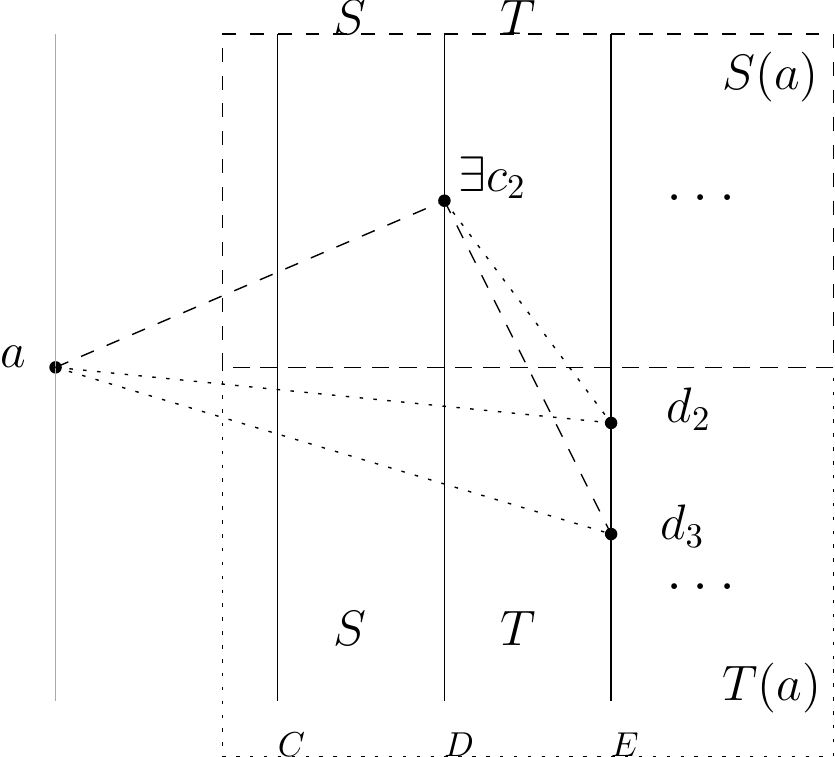}
\]
And we reach a contradiction as before.
\end{proof}

Proposition \ref{PropBasis} suggests a way to prove that $\aut(M)$ acts highly transitively on $M/R$: prove that for any finite $S$-clique $A$, $\aut(M)_A$ acts 2-transitively on the set of $R$-classes of $\bigcap_{a\in A} S(a)$. To do this, we first need to prove that indeed $R$ defines an equivalence relation on $\bigcap_{a\in A} S(a)$.

\begin{proposition}\label{PropIntersectionNbhds}
	Let $M$ be a homogeneous simple unstable 3-graph in which $R$ defines an equivalence relation with infinitely many infinite classes, and suppose that between any two $R$-classes both $R$ and $S$ are realised. Let $A=\{a_1,\ldots,a_n\}\subset M$ be a finite $S$-clique. Then $a_i\indep (A\setminus a_i)$ and $\bigcap_{1\leq i\leq n}S(a_i)$ is a simple unstable homogeneous 3-graph in which $R$ defines an equivalence relation with infinitely many infinite classes.
\end{proposition}
\begin{proof}
	The first conclusion follows from the Independence Theorem and homogeneity (one can realise the type of the vertex over the smaller clique as an via the Independence Theorem). By simplicity, we also have $(A\setminus a_i)\indep a_i$, so for any $\varnothing$-indiscernible sequence $(c_i:i\in\omega)$ the set $\{\bigwedge_{j=1}^n S(x_j,c_i):i\in\omega\}$ is satisfiable. In particular when the $c_i$ form an infinite $R$-clique. There are infinitely many classes in $\bigcap_{1\leq i\leq n}S(a_i)$ as a consequence of Corollary \ref{Corollary4}; instability is also a consequence of Corollay \ref{Corollary4}.
\end{proof}

\begin{lemma}\label{HardLemma}
	Let $M$ be a homogeneous simple unstable 3-graph in which $R$ defines an equivalence relation with infinitely many infinite classes, and suppose that between any two $R$-classes both $R$ and $S$ are realised. Let $A=\{a_1,\ldots,a_n\}\subset M$ be a finite $S$-clique. Then $\aut(M)_A$ acts 2-transitively on the set of $R$-classes in $X_A=\bigcap_{1\leq i\leq n}S(a_i)$.
\end{lemma}
\begin{proof}
	We know by Corollary \ref{Corollary4} that both $S$ and $T$ are realised in $X$, and we know by Proposition \ref{PropIntersectionNbhds} that $X$ is a homogeneous simple unstable 3-graph in which $R$ defines an equivalence relation with infinitely many infinite classes. The same argument as in Proposition \ref{PropIntersectionNbhds} proves that $Y_A=\bigcap_{1\leq i\leq n}T(a_i)$ is the same type of structure. 

For $n=1$, the situation is exactly that of Proposition \ref{PropBasis}, so let us use that result as a basis for induction on $n$. Suppose that up to $|A|=n$, $\aut(M)_A$ acts 2-transitively on $X_A$. Let $A^+$ be an $S$-clique of size $n+1$, and denote a subset of $A^+$ of size $n$ by $A^-$. The induction hypothesis and Proposition \ref{PropInfInf2} imply that $X_{A^-}=\bigcap_{a\in A^-} S(a)$ is isomorphic to the Random Bipartite Graph. if $\aut(M)_{A^+}$ does not act 2-transitively on the set of $R$ classes of $X_{A^+}$, then $X_{A^+}$, is isomorphic to $\Gamma[K_\omega^R]$.

Choose any three $R$-classes $C,D,E$ represented in $X_{A^+}$ such that $(C\cup D)\cap X_{A^+}$ is $S$-complete bipartite and $(E\cup D)\cap X_{A^+}$ is $T$-complete bipartite. Let $c_0,c_1\in C\cap X_{A^+}$. Clearly, $X_{A^+}\subset X_{A^-}$, so there exists an element $d_1\in X_{A^-}$ such that $T(d_1,c_1)\wedge S(c_0,d_1)$. This $d_1$ must be in $X_{A^-}\setminus X_{A^+}$, since $(C\cup D)\cap X_{A^+}$ is $S$-complete bipartite.

\[
\includegraphics[scale=0.8]{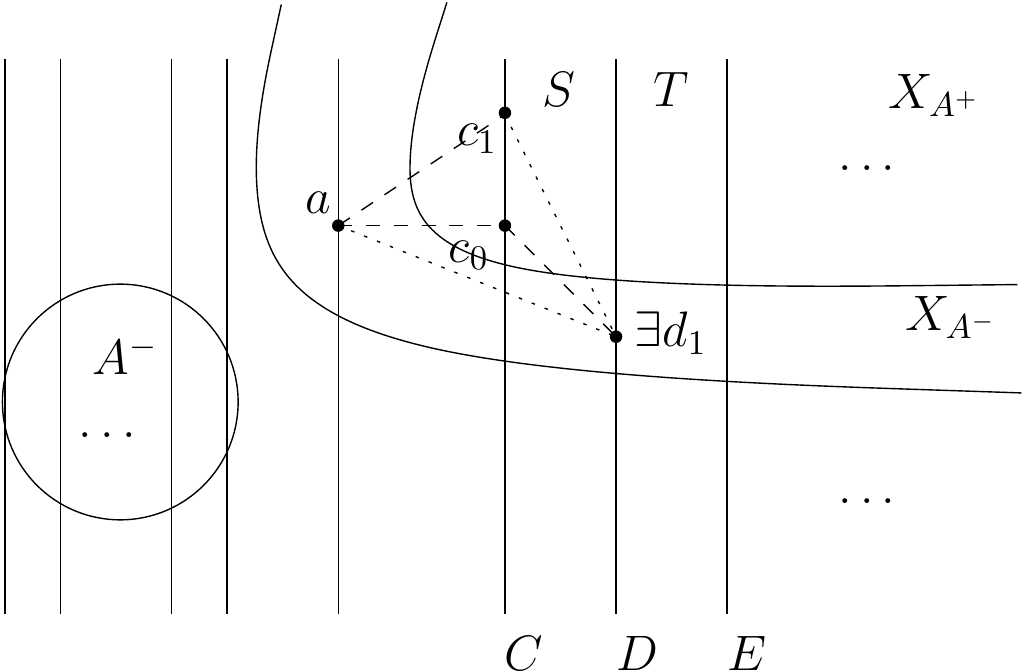}
\]

By a similar argument, one can find $d_2\in D\cap X_{A^+}$ and $e\in E\cap(X_{A^-}\setminus X_{A^+})$ such that $T(d_2,e)$. Now the pairs $c_1,d_1$ and $d_2,e$ have the same type over $A^+$, since they are $c_1$ and  $d_2$ are $S$-related to all vertices in $A^+$, $d_1$ and $e$ are $S$-related to all vertices in $A^-$, and $T(c_1,d_1), T(d_2,e)$. But an automorphism taking $c_1\mapsto d_2, d_1\mapsto e$ would take an $S$-edge in $X_{A^+}$ to a $T$-edge, contradicting homogeneity. We conclude that $\aut(M)_{A^+}$ acts 2-transitively on the set of $R$-classes in $X_{A^+}$. The result follows by induction.
\end{proof}

\begin{corollary}\label{Corollary5}
	Let $M$ be a homogeneous simple unstable 3-graph in which $R$ defines an equivalence relation with infinitely many infinite classes, and suppose that between any two $R$-classes both $R$ and $S$ are realised. Then the induced action of $\aut(M)$ on $M/R$ is highly transitive.
\end{corollary}
\begin{proof}
	By Lemma \ref{HardLemma}, any $n$-tuple of $R$-classes embeds a transversal $K_n^S$. Pick any two $n$-tuples of $R$-classes; each embeds a $K_n^S$, so by homogeneity there is an automorphism of $M$ taking one clique to the other (in any prescribed order). By invariance of $R$, this automorphism maps the one set of cliques to the other.
\end{proof}
The conditions we have isolated so far spell out the age of $M$: a simple inductive argument on $n$ following the lines of Observations \ref{ObsRFree} and \ref{ObsIndepClasses} proves:
\begin{theorem}\label{ThmAge}
	Let $M$ be a homogeneous simple unstable 3-graph in which $R$ defines an equivalence relation with infinitely many infinite classes, and suppose that between any two $R$-classes both $R$ and $S$ are realised. Then for all $n\in\omega$ and distinct $R$-classes $C_i$ in $M$ and disjoint finite $A_1^i,A_2^i\subset C_i$ there exists $x\in M$ such that for all $a^i\in A^i_1$ and $b^i\in A^i_2$ the relations $S(x,a^i_1), T(x,b^i_2)$ hold.\hfill$\Box$
\end{theorem}
Let $\rm{Forb}(RRS,RRT)$ denote the family of all finite 3-graphs not embedding the triangles $RRS,RRT$.
\begin{corollary}\label{CorImprimInf}
	Let $M$ be a homogeneous simple unstable 3-graph in which $R$ defines an equivalence relation with infinitely many infinite classes, and suppose that between any two $R$-classes both $R$ and $S$ are realised. Then $\rm{Forb}(RRS,RRT)=\age(M)$.
\end{corollary}
\begin{proof}
	It is clear that $\age (M)\subset\rm{Forb}(RRS,RRT)$. The other half, $\rm{Forb}(RRS,RRT)\subset\age(M)$, is a consequence of Theorem \ref{ThmAge}.
\end{proof}

Condensing the main results above, we get:
\begin{theorem}\label{ThmBstar}
	Let $M$ be a homogeneous simple unstable 3-graph in which $R$ defines an equivalence relation with infinitely many infinite classes, and suppose that between any two $R$-classes both $S$ and $T$ are realised. Then each pair of classes is isomorphic to the Random Bipartite Graph, and each $n$-tuple of classes embeds all $S,T$-graphs of size $n$ as transversals. Furthermore, there is a up to isomorphism unique such structure.
\end{theorem}
\begin{proof}
	We already knew about the structure of a pair of classes. The second assertion follows from the fact that the age of the Random Graph is contained in the age of $M$, so every finite $S,T$-graph $G$ is realised as a transversal. Since there is only one orbit of finite tuples of $R$-classes, $G$ is realised in all the unions of classes of the same size. Uniqueness is Corollary \ref{CorImprimInf}.
\end{proof}

Let us call the structure from Theorem \ref{ThmBstar}, for lack of a better name, $\mathcal B$. We have proved:
\begin{theorem}\label{ThmClassification}
	The following is a list of all supersimple infinite transitive homogeneous $n$-graphs with $n\in\{2, 3\}$:
\begin{enumerate}
	\item{Stable structures:
		\begin{enumerate}
			\item{$I_\omega[K_n]$ or its complement $K_\omega[I_n]$ for some $n\in\omega+1$}
			\item{$P^i[K_m^i]$}
			\item{$K_m^i[Q^i]$}
			\item{$Q^i[K_m^i]$}
			\item{$K_m^i[P^i]$}
			\item{$K_m^i\times K_n^j$}
			\item{$K_m^i[K_n^j[K_p^k]]$}
		\end{enumerate}
	}
	\item{Unstable structures:
		\begin{enumerate}
			\item{Primitive structures:
				\begin{enumerate}
					\item{The random graph $\Gamma^{i,j}$}
					\item{The random 3-graph $\Gamma^{i,j,k}$}
				\end{enumerate}
			}
			\item{Imprimitive structures with infinite classes:
			\begin{enumerate}
				\item{$K_m^i[\Gamma^{j,k}]$, $m\in\omega+1$}
				\item{$\Gamma^{i,j}[K_\omega^k]$}
				\item{$\mathcal B_n^{i,j}$, $n\in\omega$, $n\geq2$}
				\item{$\mathcal B$}
			\end{enumerate}
			}
			\item{Imprimitive structures in which the equivalence relation has finite classes: 
			\begin{enumerate}
				\item{Structures in which both unstable predicates are realised across any two equivalence classes: $C^i(\Gamma^{j,k})$}
				\item{Structures in which only one of the unstable predicates is realised across any two equivalence classes: $\Gamma^{i,j}[K_n^k]$, $n\in\omega$.}
			\end{enumerate}
			}
		\end{enumerate}
	}
\end{enumerate}
\end{theorem}
Where $\{i,j,k\}=\{R,S,T\}$ and $\mathcal B_n^{i,j}$ is the 3-graph consisting of $n$ copies of $K^R_\omega$ in which the structure on the union of any two maximal infinite $R$-cliques is isomorphic to the random bipartite graph, and all $S,T$-structures of size $k\leq n$ are realised transversally in the union of any $k$ $R$-classes.

\bibliographystyle{plain}
\bibliography{refs}

\printindex

\end{document}